\documentclass[11 pt,amstex,3p,times]{article}

\usepackage{hyperref}
\usepackage{enumerate,fullpage}
\usepackage{amssymb,amsmath,graphicx,amsthm,mathrsfs}
\usepackage{color}

\newcommand{\be}{\begin{eqnarray*}}
\newcommand{\en}{\end{eqnarray*}}
\newcommand{\bes}{\begin{eqnarray}}
\newcommand{\ens}{\end{eqnarray}}

\usepackage{epic, curves}
\def\nn{\nonumber}
\newcommand{\al}{\alpha}

\newcommand{\la}{\lambda}

\newcommand{\ep}{\epsilon}

\newcommand{\N}{\mathbb{N} }
\newcommand{\R}{\mathbb{R}  }

\newcommand{\usf}{\mathsf{u}}

\newtheorem{theorem}{Theorem}[section]

\newtheorem{definition}{Definition}[section]
\newtheorem{lemma}{Lemma}[section]

\newtheorem{remark}{Remark}[section]

\def\bq{\begin{equation}}
\def\eq{\end{equation}}
\def\bqq{\begin{eqnarray*}}
\def\eqq{\end{eqnarray*}}

\def\nn{\nonumber}

\newcommand{\ve}{\mathcal{U}}

\newcommand{\Ss}{\mathscr{S}}

\newcommand{\HL}{ \mbox{ \raisebox{7.2pt} {\tiny$\circ$} \kern-10.7pt} {H_L^1} }
\newcommand{\Wp}{ \mbox{ \raisebox{7.7pt} {\scriptsize$\circ$} \kern-10.1pt} {W^{1,p}} }
\newcommand{\Wpp}{ \mbox{ \raisebox{7.7pt} {\scriptsize$\circ$} \kern-10.1pt} {W^{1,p'}} }
\newcommand{\Sz}{ \mbox{ \raisebox{7.5pt} {\scriptsize$\circ$} \kern-10.1pt} {\Ss} }
\newcommand{\HLnew}{ \mbox{ \raisebox{7pt} {\scriptsize$\circ$} \kern-10.1pt}{H}^1_L }
\newcommand{\HLn}{{\mbox{\,\raisebox{4.7pt} {\tiny$\circ$} \kern-9.3pt}{H}^{1}_{L}  }}
\newcommand{\HLs}{{\mbox{\raisebox{8.7pt} {\scriptsize$\circ$} \kern-10.1pt}{H}^1_L  }}

\DeclareMathOperator*{\tr}{tr_\Omega}

\newcommand{\boxednumber}[1]{\expandafter\readdigit\the\numexpr#1\relax\relax}

\newcommand{\Hsd}{\mathbb{H}^{-s}(\Omega)}

\newcommand{\diff}{\, \mbox{\rm d}}

\newcommand{\Hs}{\mathbb{H}^s(\Omega)}

\newcommand{\Nin}{\,{\mbox{\,\raisebox{6.0pt} {\tiny$\circ$} \kern-10.9pt}\N }}

\newcommand{\calL}{{\mathcal L}}

\title{\bf Continuity of solutions of a class of  fractional equations}
\author{ Duc Trong Dang$^1$, Erkan Nane $^2$  \footnote{Corresponding author: \url{ezn0001@auburn.edu }}, Dang Minh Nguyen$^1$ and  Nguyen Huy Tuan$^1$, \\\\
\small $^1$Fact. Maths and Computer Science, University of Science,\\
\small Vietnam National University, 227 Nguyen Van Cu, Dist.5, HoChiMinh City, VietNam.\\
	\small $^2$ Department of Mathematics and Statistics, Auburn University, Auburn, AL 36849
}
\begin{document}
\date{}
\maketitle
\begin{abstract}
In practice many problems related to space/time fractional equations depend on   fractional parameters. But these fractional parameters are not known a priori in modelling problems. Hence continuity of the solutions with respect to these parameters is important for modelling purposes.
 In this paper we will study the continuity of the solutions of a class of equations including the Abel equations of the first and second kind, and  time  fractional diffusion type equations. We consider continuity with respect to the fractional parameters  as well as the initial value.
\end{abstract}


\tableofcontents

\section{Introduction}

Diffusion is one of the most important transport mechanisms found in nature. At a microscopic level, the diffusion is caused by random motion of individual particles. In fact, let $x(t)$, $t>0$, be the  displacement of a particle at time $t$. From the theory of random walk (see, e.g., \cite{Metzler-Klafter}) the mean squared displacement grows  as
$$\mathbb{E}(x^2(t))=\langle x^2(t)\rangle\sim t^\alpha,$$
where the constant $\alpha>0$ can be called the {\it order of diffusion}.
Classical model of diffusion corresponds to the case $\alpha=1$. In this case, the random  walk limit  is modeled by a   Brownian motion  with the corresponding Laplacian operator. The corresponding model for  $\alpha\not=1$ is called {\it anomalous diffusion}. A growing number of studies on the diffusion phenomena have shown the prevalence of anomalous diffusion in which the mean square variance grows faster (in the case of superdiffusion, i.e., $\alpha>1$) or slower (in the case of subdiffusion, i.e., $0\leq \alpha<1$) than the one for   diffusion processes. From the mathematical point of  view, to model the anomalous diffusion, we use fractional derivatives of $t$ (of order $\alpha\not\in\mathbb{Z}$) instead of the classical derivatives. Experiments showed that the fractional derivative models lead to explaining and understanding complex systems better with  the use of  anomalous diffusion processes. The fractional models, dated back to the 19th century, has never ceased to inspire scientists and engineers to investigate this area of research deeper.  Nowadays, anomalous diffusion became 'normal' in spatially disordered systems, porous media, fractal media, viscoelastic materials \cite{Chechkin}, \cite{Ginoa}, \cite{Nigmatulin}, pollution transport, turbulent fluids and plasmas \cite{Castillo1},\cite{Castillo2}, \cite{Kim}, biological media with traps, binding sites or macro-molecular crowding \cite{Ding},\cite{Djordjevic}, stock price movements \cite{Sabatelli}, \cite{Scalas}.

The parameter  $\alpha$ is an important constant in the model of anomalous diffusion. For example, a simple anomalous diffusion can be described macroscopically by the fractional diffusion equation
\begin{equation}
\partial_t^\alpha u(x,t)=\Delta u(x,t),\ \ \ x\in\Omega\subset \mathbb{R}^k
\label{Simple-Diffusion}
\end{equation}
where $u(x,t)$ is the probability of finding a particle at spatial point $x$ and time $t$ and that $\partial_t^\alpha u=\frac{\partial^\alpha u}{\partial t^\alpha}$ is the Caputo fractional derivative of the function $u$.
The parameter  $\alpha$ can only be determined  experimentally. Therefore, we cannot obtain the exact value of $\alpha$ and we often have a known sequence $\alpha_n$, called the perturbed fractional parameter, satisfying $\alpha_n\stackrel{n\to\infty}{\longrightarrow} \alpha$ in an appropriate sense. In fact, if the sequence $(\alpha_n)$ is deterministic, we have $\lim_{n\to\infty}|\alpha_n-\alpha|=0$ and if the sequence is random, we can assume that
$\lim_{n\to\infty}\mathbb{E}|\alpha_n-\alpha|^\lambda=0$. Hence, in these cases, we only find perturbed solutions of
(\ref{Simple-Diffusion}). The question is whether the perturbed solutions are stable with respect to the  parameters $\alpha_n$.  It is surprising  to see  that the inexact nature of the parameter $\alpha$ is not investigated in the literature of fractional calculus.

For fixed   $\alpha$  the solutions of \eqref{Simple-Diffusion} have been investigated  recently. Existence and uniqueness of the solutions as well as numerical methods for solving forward fractional equations are well developed. One  can find  papers devoted to the Abel (and generalized Abel) equations in  Gorenflo-Vessella \cite{Gorenflo-Vessella} and references therein. Fractional diffusion equations in  Banach spaces  were studied  by
Bazhlekova \cite{Bazhlekova}, Cl\'ement \cite{ClementTAMS},\cite{ClementJDE}. Fractional diffusion equations in Hilbert spaces  were considered by Chen et al. \cite{cmn-2012}, Li and Xu \cite{Li-Xu}, Meerschaert et al.    \cite{mnv-09}, Sakamoto and Yamamoto \cite{Sakamoto}, Zacher \cite{Zacher},  and many others. There is also an ever-growing literature of research on the fractional inverse problem with the exact fractional order.
For example, the backward problems for fractional diffusion processes, which are ill posed, corresponding to the irreversibility of time, have a rich literature: see, for example, {\cite{Caixuan}, \cite{Liu}, \cite{Sakamoto},    \cite{Zheng}. }

For $\alpha>0$, the common  problem  of interest  is  finding the solution of the general problem of the form given by the  equation $A_\alpha u_\alpha=f$ where the function $f$ is given, $u_\alpha$ is unknown and $A_\alpha$ is a kind of fractional operator; for example,
$A_\alpha=\partial_t^\alpha$ is the Caputo fractional derivative. If the inexact nature of the fractional parameter $\alpha$ is present then the continuity of the solution with respect to the parameter $\alpha$ has to be  considered in numerical considerations. Suppose  that a  sequence $(\alpha_n)$ satisfies that  $\alpha_n\to \alpha$
as $n\to\infty$. Since $\alpha$ is unknown, in real life applications, we can only compute an approximation  $u_{\alpha_n}$ of $u_\alpha$. This situation raises the following   natural questions:
\begin{enumerate}[\bf \upshape (a)]
	\item Does $u_{\alpha_n}\to u_\alpha$ in an appropriate sense as $n\to\infty$?
	\item If $u_{\alpha_n}\not\to u_\alpha$ then, is there  another way to recover the convergence?
\end{enumerate}
These questions are related with the continuity of solution of fractional problems with respect to the fractional parameter $\alpha$. For brevity, we shall call the investigation of these questions by  {\it the parameter-continuity problem}. Papers investigating these two questions for the fractional problem is very rare. To the best of our knowledge, there are a few  papers related to these questions. The  paper \cite{Hatano-Nakagawa} was devoted to the problem of determination of the parameter in  a fractional diffusion equation. These authors considered the fractional Cauchy problem
in a domain  $\Omega\subset\mathbb{R}^k$, and determine the parameter $\alpha$ from the observed data measured at a point inside $\Omega$. In \cite[Theorem 4.2]{Bazhlekova}, using the semigroup language, the author gave a formula which described the relation between the solutions of the fractional Cauchy problem for two  parameters  $\alpha,\beta\in (0,1]$. The papers \cite{Bondarenko},  \cite{Cheng}, \cite{Li-Zhang-Jia-Yamamoto},  \cite{Rodrigues} dealt with the problem of simultaneously identifying the fractional parameter and the space-dependent diffusion coefficient from boundary measurements. In these papers, some results on the Lipschitz continuity of solutions with respect to the fractional parameter were proved.  In \cite[Proposition 1]{Li-Zhang-Jia-Yamamoto}, the authors considered the problem of finding a function $u=u_{\gamma,D}(x,t)$ ($0<\gamma<1$) satisfying
\begin{eqnarray*}
\frac{\partial^\gamma u}{\partial t^\gamma}=\frac{\partial}{\partial x}\left(D(x)\frac{\partial u}{\partial x}\right),
\ \ \  x\in (0,1), t\in (0,T),
\end{eqnarray*}
subject to the Neumann boundary condition $u_x(0,t)=u_x(1,t)=0$ and the initial condition $u(x,0)=f(x)$. They proved the Lipschitz continuity
\begin{equation*}
\Vert u_{\gamma,D}(0,t)-u_{\tilde{\gamma},\tilde{D}}(0,t)\Vert_{L^2(0,T)}\leq C(|\gamma-\tilde{\gamma}|
+\Vert D-\tilde{D}\Vert_{C[0,1]})
\end{equation*}
where $\Vert D-\tilde{D}\Vert_{C[0,1]}:=\sup_{0 \le x \le 1} |D(x)-\tilde{D}(x)|. $

The latter result in  \cite{Li-Zhang-Jia-Yamamoto} is a kind of  parameter-continuity result. To the best of our knowledge, there are no papers in the literature that  consider these problems mentioned above systematically.

{
In this paper, inspired by  the above discussion we  study  systematically the continuity of the solution of equations similar to \eqref{Simple-Diffusion} with respect to the parameter $\alpha$. Our methods are different than the method used in \cite{Li-Zhang-Jia-Yamamoto}, and we obtain continuity of the solutions of abstract time fractional equation in Banach space setting with respect to the time fractional parameter as well as the initial function.

 To give a sense of  our results, we mention a  particular case of our results in the simplest case of time fractional diffusion in the interval $(0,L)$.
The equation
\begin{equation*}
\begin{split}
\frac{\partial^2 v(x)}{\partial x^2}&=-\lambda v(x), \ x\in (0,L),\\
v(0)&=0=v(L),
\end{split}
\end{equation*}
is solved by a sequence of eigenvalues $\lambda_n=(n\pi/L)^2$ and  eigenfunctions $\phi_n(x)=\sqrt{\frac{2}{L}}\sin (n\pi x/L)$. It is well known that  the set of functions $\{ \sqrt{\frac{2}{\pi}}\sin (n\pi x/L):\ \ n \in \N \}$ is an {orthonormal basis} of $ L^2(0, L)$.
Let $\alpha\in (0,1)$. Then by separation of variables the solution of time fractional heat equation in $(0,L)$
\begin{equation}\label{Simple-Diffusion-1D}
\begin{split}
\partial_t^\alpha u(x,t)&=\frac{\partial^2 u(x,t)}{\partial x^2} ,\ \ \ x\in (0,L),\ \ t>0,\\
u(x,0)&=\theta(x), \ \ x\in (0,L),\\
u(0,t)&=0=u(L, t)
\end{split}
\end{equation}
is given by
\begin{equation*}
u(x,t)=\sum_{n=1}^\infty  \theta_n  E_{\alpha,1}(-(n\pi/L)^2t^\alpha)\sin (n\pi x/L)
\end{equation*}
where $\theta_n=\int_{0}^L\theta(x)\sin (n\pi x/L)dx$, and  $E_{\alpha,1}(-(n\pi/L)^2t^\alpha)$ is the Mittag-Leffler function defined below in equation \eqref{Mittag-Leffler-function}.
To emphasize the dependence of the solution  of equation \eqref{Simple-Diffusion-1D} to the initial value and $\alpha$, we write $u(x,t)=u_{\theta, \alpha}(x,t)$. A particular case of  our Theorem   \ref{homogeneous-fractional-diffusion-theorem}   shows the following continuity properties of the solution of equation \eqref{Simple-Diffusion-1D}:

Let $\alpha,\alpha'\in (0,1)$,  $\theta,\theta'\in H^1$. ($H^1$ is domain of Laplacian in the interval $(0,L)$ defined below in section \ref{tfpde-hilbert-space})
	\begin{enumerate}[\bf \upshape(i)]
		\item If  $\theta'\to\theta$ in $H^1$, $\alpha'\to\alpha$ then
		$$ || u_{\alpha', \theta'}-u_{\alpha, \theta}||_{L^2(0,L)}\rightarrow 0. $$
		\item If $\theta,\theta'\in H^1$, $1>\rho\geq 0$  $\alpha'\in [\alpha_0,\alpha_1],$  then
		there exists a constant $C=C(\alpha_0,\alpha_1, \rho)$ such that
		$$ || u_{\alpha',\theta'}(.,t)-u_{\alpha,\theta}(.,t)||_{H^\rho}^2\leq 2||\theta'-\theta||_{H^1}^2+C||\theta||^2_{H^1}(|\alpha'-\alpha|)^{2\gamma},  $$
where $\gamma=\min \{1, (1-\rho)/\beta_1\},$ and  $ H^\rho$ is defined in Section \ref{tfpde-hilbert-space}.
	\end{enumerate}
We will prove  continuity properties  of fractional differential equations of Abel type as well as abstract time fractional Cauchy problems in Banach space and Hilbert space settings with external force terms with respect to various parameters including the time fractional derivative parameter $\alpha\in (0,1)$.

  }
Next, we give an outline of the paper.
 In Section 2, we will give  some definitions and    basic properties of fractional derivatives. Using the definition stated in Zygmund \cite[page 134]{Zygmund}  we will  define the fractional derivative in a general form in  Banach spaces. We will also give some properties of Mittag-Leffler functions that  are used frequently in fractional problems in this paper. In Section 3, we investigate the first question of continuity with respect to the parameter $\alpha$. In particular, we will  consider the forward fractional problems of the generalized Abel equations. In section 4 we will consider abstract fractional diffusion equations in the Banach space setting, and Hilbert space setting. In sections  3 and 4, we  will show that the parameter-continuity is    of Lipschitz continuity type in most of the cases we consider. In Section 5, we investigate the second question. For many inverse problems
$A_\alpha u_\alpha=f$,  we have $u_{\alpha_n}\not\to u_\alpha$ as $n\to\infty$. This prevents one to approximate solutions by  numerical methods.  We shall define a new concept of regularization (of the family of  operators $A_\alpha$) for this case and derive a regularization operator for the fractional backward problems.

\section{Fractional derivatives and the Mittag-Leffler functions}

\setcounter{equation}{0}

 We need some notations and properties in order  to state our problems precisely.
We barrow some definitions stated in Zygmund \cite{Zygmund}, page 134-136,  to define fractional derivatives.
Put
\begin{equation*}
 k_\alpha(t)=\frac{1}{\Gamma(\alpha)}t^{\alpha-1}\ \ \ \ {\rm for}\  t>0,\ \alpha>0,
 \end{equation*}
here $\Gamma(\cdot)$ is the standard Gamma function defined by
\begin{equation*}
    \Gamma(z) = \int_0^\infty t^{z-1}e^{-t} dt , \ \ \ \ \Re(z) >0.
\end{equation*}
Let $\mathbb{X}$ be a Banach space and $f\in L^1(0,T,\mathbb{X})$,
we denote the Riemann-Liouville fractional integral operator (see, e.g., Gorenflo-Vessella \cite{Gorenflo-Vessella}) by
\begin{equation*}
J^\alpha f(t)=\frac{1}{\Gamma(\alpha)}\int_0^t(t-s)^{\alpha-1}f(s)ds=k_\alpha*f(t).
\end{equation*}
{For $u\in L^1(0,T,\mathbb{X})$, if $J^{1-\alpha}u$ is absolutely continuous then we define the Riemann-Liouville fractional derivative of order $\alpha\in(0,1)$ of $u$ by
$$ D^\alpha_tu(t) :=  \frac{1}{\Gamma(1-\al)}\frac{d}{dt}\int\limits_{0}^{t}  (t-s)^{-\alpha }u(s)ds. $$}
If $u$ is absolutely continuous and differentiable a.e. then we define  the (left-sided) Caputo fractional derivative of order $\alpha$ by
\begin{equation*}
\partial^\alpha_t u=\frac{\partial ^{\alpha}u}{\partial t^{\alpha}} := \frac{1}{\Gamma(1-\al)}\int\limits_{0}^{t}  (t-s)^{-\alpha }u'(s)ds.
\end{equation*}
Let $ \eta_j\in (0,1]$, $j=1,...,m$, $\sigma_j=\sum_{\ell=1}^j\eta_\ell$.
We introduce the notation for the Miller-Ross sequential derivatives (see \cite{Polubny}, page 108])
\begin{align*}
{\cal D}_t^{\sigma_m} &= D_t^{\eta_m}D_t^{\eta_{m-1}}...D_t^{\eta_1},\\
 {\cal D}_t^{\sigma_m-1} &= D_t^{\eta_m-1}D_t^{\eta_{m-1}}...D_t^{\eta_1}
\end{align*}
with
\begin{equation*}
\sigma_m=\sum_{j=1}^m\eta_j,\ \ 0<\eta_j\leq 1,\ \ \ \ j=1,...,m.
\end{equation*}
From the definition of the Riemann-Liouville fractional derivative,  we have the following lemma which collects some well-known facts  about fractional derivative.

\begin{lemma}\label{first-lemma}
	
	\begin{enumerate}[\bf \upshape(a)]
		\item { Let $0<\alpha<1$ and $u\in L^1(0,T;\mathbb{X})$.  If  there exist $f\in L^1(0,T,\mathbb{X})$ such that $u=J^\alpha f$ then the function $u$ has the fractional derivative $D^\alpha_t u=f$.}
		
		\smallskip
		
		\item If $D^\alpha_t u\in L^p(0,T;\mathbb{X})$ with $1\leq p<\alpha^{-1}$ then $u\in L^q(0,T;\mathbb{X})$ with $q\in [1,\frac{p}{1-\alpha p})$.
		
			\smallskip

		\item  If $D^\alpha_t u\in L^p(0,T;\mathbb{X})$ with $p=\alpha^{-1}$ then $u\in L^q(0,T;\mathbb{X})$ with $q\geq 1$.
		
			\smallskip

		\item
		 If $D^\alpha_t u\in L^p(0,T;\mathbb{X})$ with $p>\alpha^{-1}$ then u is H\"older continuous with exponent $\theta=\alpha-p^{-1}$ and $u(0)=0$.
		
			\smallskip
			
		\item
		
		 For $c\in \mathbb{X}$ we have
		 \begin{equation*}
		 D^\alpha_t c=\frac{c}{\Gamma(1-\alpha)}\frac{d}{dt}\int_0^t(t-s)^{-\alpha}ds=\frac{ct^{-\alpha}}{\Gamma(1-\alpha)}.
		 \end{equation*}
		
		 	\smallskip
		 	
		 \item
		
		 For $u\in L^1(0,T; \mathbb{X})$ and $0<\alpha\leq\beta<1$, we have
		 $$  D^\alpha_t J^\beta u=J^{\beta-\alpha}u.$$
		
		\item  For an absolutely continuous $u$  we have
		$$D^\alpha_t(u(t)-u(0))=\partial^\alpha_t u(t).$$

		\end{enumerate}

\end{lemma}

\begin{proof} We first prove a). We have $J^{1-\alpha}u=J^{1-\alpha}J^\alpha f=J^1f$. It follows that $f=\frac{d}{dt}J^{1-\alpha}u=D^\alpha_tu$.
Proofs of the results (b), (c), (d), (e) can be found in Zygmund \cite{Zygmund}, pages 134-136. The proof of the  results (f) and (g) can be seen in \cite[Chap. 6, page 98]{Gorenflo-Vessella}, .
\end{proof}

 \noindent In this paper  we consider the Mittag-Leffler functions defined by
\begin{equation}\label{Mittag-Leffler-function}
  E_{\alpha,\beta}(z) = \sum_{k=0}^\infty \frac{z^k}{\Gamma(k\alpha+\beta)},\quad z\in \mathbb{C}.
\end{equation}
The Mittag-Leffler function is a two-parameter family of entire functions of $z$ of
order $\al^{-1}$ and type $1$ \cite[Chap.1]{Polubny}.
The exponential function is a particular case of the Mittag-Leffler function, namely $E_{1,1}(z)=e^z$. Two  important functions derived  from this family
are $E_{\alpha,1}(-\lambda t^\alpha)$ and $t^{\alpha-1}E_{\alpha,\alpha}(-\lambda t^\alpha)$,
which occur in the solution operators for the initial value problem \eqref{homogeneous-fractional-diffusion}
and the nonhomogeneous problem \eqref{nonhomogeneous-fractional-diffusion}, respectively.

We will use the next lemma for the derivatives of some contour integrals.

\begin{lemma}\label{upper-bound}
		Let $\alpha_0,\alpha_1, \beta_0\in\mathbb{R}$ satisfy $0<\alpha_0<\alpha_1<2$, $\alpha_1<2\alpha_0$. {Let $\varphi\in (\frac{\pi\alpha_1}{2},\pi\alpha_0)$ },  and put
		$$ g_0(r)
		:=\exp\left(r^{1/\alpha_0}\cos\left(
		\frac{\varphi}{\alpha_1}\right) \right). $$
		Then the function { $r^\mu |\ln(r)|^\nu g_0(r)$ is in  $L^1(\rho,\infty)$ for every $\mu,\nu\in\mathbb{R}, \nu>0,\rho>0$ } and
		$$  \left|\exp({r^{1/\alpha}e^{\pm i\varphi/\alpha}})\right|\leq g_0(r).$$
\end{lemma}

	\begin{proof}

	We note that
	$$  \left|\exp({r^{1/\alpha}e^{\pm i\varphi/\alpha}})\right|=\exp\left(r^{1/\alpha}\cos\left(
	\frac{\varphi}{\alpha}\right) \right).  $$
From the choice of $\varphi$, we obtain
	$$  \frac{\pi}{2}<\frac{\varphi}{\alpha_1}<\frac{\varphi}{\alpha}< \frac{\varphi}{\alpha_0}< \pi.$$
	Since the function $\cos x$ is decreasing in $(\pi/2,\pi)$, we obtain
	$$   -1<\cos\left(
	\frac{\varphi}{\alpha_0}\right)\leq \cos\left(
	\frac{\varphi}{\alpha}\right)\leq \cos\left(
	\frac{\varphi}{\alpha_1}\right)<0  $$
	 for every $\alpha\in[\alpha_0,\alpha_1]$. Therefore
	$$   \left|\exp({r^{1/\alpha}e^{\pm i\varphi/\alpha}})\right|=\exp\left(r^{1/\alpha}\cos\left(
	\frac{\varphi}{\alpha}\right) \right)\leq g_0(r).$$
	Since $\cos\left(
	\frac{\varphi}{\alpha_1}\right)<0$, the function { $r^\mu  g_0(r)|\ln r|^\nu$ is Lebesgue integrable on
	$[\rho,\infty)$ for every $\mu,\nu\in \mathbb{R}, \nu>0$. This completes the proof of our lemma.}
	
	\end{proof}

In the following lemma, we show some inequalities which  hold uniformly for all the fractional parameter $\alpha$ in an interval $[\alpha_0,\alpha_1]$. A few of these inequalities will not to be used in the present paper. However, they also are presented here since they can be applied for  future papers.
This lemma also establishes some   properties of some generalized Gronwall type inequalities.

\begin{lemma}\label{Mittag-Leffler}
Let $\alpha>0,\beta>0$. Then  $E_{\alpha,\beta}(z)$ is differentiable
with respect to $\alpha,\beta, z$. Moreover, assume that $a_0, b_0, M,\alpha_0,\alpha_1, \beta_0,\beta_1\in\mathbb{R}$ satisfy $a_0,b_0,M>0$, $0<\alpha_0<\alpha_1<2$, $\alpha_1<2\alpha_0$, $\beta_0<\beta_1$.
\begin{enumerate}[\bf \upshape(a)]

	\item
 For $\alpha\geq a_0,\beta\geq b_0$, there exists a constant $C_E=C_E(a_0,b_0)$ such that
\begin{align*}
    |E_{\alpha,\beta}(z)|&\leq C_E E_{a_0,b_0}(M)\ \ \ {\rm for\ all}\
z\in\mathbb{C},0\leq |z|\leq M,\\
0\leq E_{\alpha,\beta}(z)&\leq C_E E_{a_0,b_0}(M)\ \ \ {\rm for\ all}\
z\in\mathbb{R},0\leq z\leq M,
\end{align*}
and, for $z_0\in\mathbb{R}$, $\alpha_0\leq\alpha\leq\alpha_1$ there exists a constant $C=C(z_0,\alpha_0,\alpha_1,\beta_0)>0$ such
	that
	\begin{align*}
	|E_{\alpha,\beta}(z)|+
	\left|\frac{\partial E_{\alpha,\beta}}{\partial \alpha}(z)\right|+
	\left|\frac{\partial E_{\alpha,\beta}}{\partial \beta}(z)\right|
&\leq \frac{C}{1+|z|} ,\nonumber \\
\left|\frac{\partial E_{\alpha,\beta}}{\partial z}(z)\right|
&\leq \frac{C}{(1+|z|)^2} \ {\rm for\ all}\ z<z_0.
\label{beta-M-L-bound}
	\end{align*}
We also have the Lipschitz continuity
			\begin{eqnarray*}
		 |E_{\alpha,\beta}(z_1)-E_{\alpha,\beta}(z_2)|\leq C|z_1-z_2|
		 	\end{eqnarray*}
		for every $z_1,z_2$ in $(-\infty,z_0]$.

	\smallskip

	\item

	 Let $z_1\in\mathbb{R}, z_1>0$ and put $\phi_0(\alpha,\beta,z)=\frac{1}{\alpha}z^{(1-\beta)\alpha}e^{z^\frac{1}{\alpha}}$. For $z\geq z_1>0$, there exists a constant $C=C(z_1,\alpha_0,\alpha_1,\beta_0)>0$ such that
	 \begin{eqnarray*}
	 	\left|E_{\alpha,\beta}(z)- \phi_0(\alpha,\beta,z)\right|&\leq&\frac{C}{1+|z|},\\
	 	\left|\frac{\partial E_{\alpha,\beta}}{\partial \alpha}(z) - \frac{\partial\phi_0}{\partial\alpha}(\alpha,\beta,z)\right|&\leq&\frac{C}{1+|z|},\\
	 	\left|\frac{\partial E_{\alpha,\beta}}{\partial \beta}(z)- \frac{\partial\phi_0}{\partial\beta}(\alpha,\beta,z)\right|
	 	&\leq&\frac{C}{1+|z|}.
	 \end{eqnarray*}

	\smallskip
	
             \item For $z\geq z_1>0, \alpha_0\leq\alpha\leq\alpha_1,\beta_0\leq \beta\leq\beta_1$, there exists constants $C^{-}, C^{+} > 0$ depending only on $z_1, \alpha_{0}, \alpha_{1},\beta_0,\beta_1$ such that
         \begin{eqnarray*}
	C^-\phi_0(\alpha,\beta,z)\leq E_{\alpha,\beta}(z)\leq
C^+\phi_0(\alpha,\beta,z).
	\end{eqnarray*}
	Especially, for $\beta=1$, $z\geq 0$, we have
$$\frac{C^-}{\alpha} e^{z^{\frac{1}{\alpha}}} \le E_{\alpha,1}(z) \le \frac{C^+}{\alpha} e^{z^{\frac{1}{\alpha}}}.
$$	
		\smallskip	

	\item
	$E_{\alpha,\alpha}(z)\geq 0$ for $z\in\mathbb{R}$.

		\smallskip

		\item  Let $0 < \alpha_{0} < \alpha_{1} < 1.$ Then there exists constants $C^{-}, C^{+} > 0$ depending only on $\alpha_{0}, \alpha_{1}$ such that
		\begin{eqnarray*}
		\frac{C^-}{\Gamma(1-\alpha)} \frac{1}{1-z} \le E_{\alpha,1}(z) \le \frac{C^+}{\Gamma(1-\alpha)} \frac{1}{1-z},~~ \forall  z \le 0.
		\end{eqnarray*}

		\smallskip

		\item  We have
		\begin{equation*} \frac{d}{dz}E_{\alpha,1}(z)=\frac{1}{\alpha}E_{\alpha,\alpha}(z).
		\end{equation*}

		\smallskip

		\item	 Let a function $g\in L^1(0,T)$ and
		$\lambda\in\mathbb{C}$. Then the integral equation
		\begin{equation*}
		u(t)=g(t)+\frac{\lambda}{\Gamma(\alpha)} \int_0^t\frac{u(s)}{(t-s)^{1-\alpha}}ds
		\end{equation*}
		has a unique solution
		\begin{equation*}
		u(t)=g(t)+ \lambda\int_0^t E_{\alpha,\alpha}(\lambda(t-s)^\alpha)g(s)ds
		\end{equation*}
		in the space $L^1(0,T)$.

		\smallskip
		
		\item  Let $g, \varphi\in L^1(0,T)$  and $\lambda\in\mathbb{R}$,
		$\lambda\geq 0$. If $\varphi$ satisfies the integral inequality
		$$\varphi(t)\leq g(t)+\frac{\lambda}{\Gamma(\alpha)} \int_0^t\frac{\varphi(s)}{(t-s)^{1-\alpha}}ds$$
		then
		$$   \varphi(t)\leq g(t)+ \lambda\int_0^t E_{\alpha,\alpha}(\lambda(t-s)^\alpha)g(s)ds.$$
Moreover, if $g\in L^p(0,T)$, $1\leq p\leq \infty$ then
$$  \Vert\varphi\Vert_{L^p(0,T)}\leq (1+\lambda T E_{\alpha,\alpha}(\lambda T^\alpha))\Vert g\Vert_{L^p(0,T)}. $$
	
\end{enumerate}

\end{lemma}

\begin{proof}
\noindent {\bf Proof of (a):}
The first inequality of part (a) can  be verified directly from the definition of the Mittag-Leffler function. In fact, we have
$$   |E_{\alpha,\beta}(z)|\leq\sum_{k\leq 2/{a_0}}\frac{|z|^k}{\Gamma(k\alpha+\beta)}+\sum_{k> 2/{a_0}}\frac{|z|^k}{\Gamma(k\alpha+\beta)}:=E_1+E_2.$$
Since $\Gamma(x)$ is increasing as $x>2$, we obtain
$$ E_2\leq \sum_{k> 2/{a_0}}\frac{M^k}{\Gamma(ka_0+b_0)}.$$
For $0\leq k\leq 2/{a_0}$, $a_0\leq\alpha\leq a_0+2$, $b_0\leq\beta\leq b_0+2$ we have
\begin{align*}
\Gamma(k\alpha+\beta)&=\frac{\Gamma(k\alpha+\beta+2)}{(k\alpha+\beta)(k\alpha+\beta+1)}\\
&\geq\frac{\Gamma(ka_0+b_0+2)}{(k\alpha+\beta)(k\alpha+\beta+1)}   \\
&=
\frac{(ka_0+b_0)(ka_0+b_0+1)}{(k\alpha+\beta)(k\alpha+\beta+1)}\Gamma(ka_0+b_0)
\end{align*}
which gives
$$  \Gamma(ka_0+b_0)\leq \frac{(k\alpha+\beta)(k\alpha+\beta+1)}{(ka_0+b_0)(ka_0+b_0+1)}\Gamma(k\alpha+\beta)
\leq C_E\Gamma(k\alpha+\beta), $$
where
$$C_E=\frac{(2(a_0+2)/a_0+b_0+2)(2(a_0+2)/a_0+b_0+3)}{b_0(b_0+1)}.$$
For $0\leq k\leq 2/{a_0}$, $a_0\leq\alpha\leq a_0+2$, $\beta> b_0+2$ we have
$$ \Gamma(ka_0+b_0)\leq C_E\Gamma(ka_0+b_0+2)\leq C_E\Gamma(k\alpha+\beta).$$
So we obtain
$$ \Gamma(ka_0+b_0)\leq C_E\Gamma(k\alpha+\beta)\ \ \ \forall 0\leq k\leq 2/{a_0}, a_0\leq\alpha\leq a_0+2,\beta\geq b_0.  $$
Finally, for
$1\leq k\leq 2/{a_0}$, $a_0+2<\alpha$, $\beta\geq b_0$ we have
$k\alpha+\beta\geq k(a_0+2)+b_0\geq 2$. Hence
$$ \Gamma(ka_0+b_0)\leq C_E\Gamma(k(a_0+2)+b_0)\leq C_E\Gamma(k\alpha+\beta).$$
Combining all cases gives
$$ E_1\leq \sum_{k\leq 2/{a_0}}C_E\frac{M^k}{\Gamma(ka_0+b_0)}.$$
From the estimate for $E_1,E_2$ and the inequality $C_E\geq 1$, we obtain
$$  |E_{\alpha,\beta}(z)|\leq E_1+E_2\leq C_EE_{a_0,b_0}(M),\ \ \ \forall z\in\mathbb{C},|z|\leq M. $$
To prove Part (b) and the second inequality of Part (a) of the lemma, we make of use of some preliminary results.
For $0<\alpha_0<\alpha_1<2$, $\alpha_1<2\alpha_0$  we have
	$$  0<\frac{\pi\alpha_1}{2}<\min\{\pi,\pi\alpha_0\}. $$
	Hence, we can choose $\varphi$ satisfying
	$$0<\frac{\pi\alpha_1}{2}<\varphi<\min\{\pi,\pi\alpha_0\}\leq \pi.$$
Put $\varphi_0=\frac{\pi\alpha_1}{2}$ and choose $0<\rho_0<\rho_1$. For $\varphi\in (\varphi_0,\pi],\rho\in [\rho_0,\rho_1]$,
define the curve
\begin{equation}
\gamma_{\rho,\varphi}=C^-_{\rho,\varphi}\cup C_{\rho, \varphi}\cup C^+_{\rho,\varphi}
\label{gammarhovarphi}
\end{equation}
 where
\begin{eqnarray*}
C^+_{\rho,\varphi}&=&\{re^{i\varphi}:\ r\geq \rho\},\\
C^-_{\rho,\varphi}&=&\{re^{-i\varphi}:\ r\geq \rho\},\\
C_{\rho, \varphi}&=&\{ \rho e^{i\theta}: -\varphi<\theta< \varphi\}.
\end{eqnarray*}
Put
\begin{align*}
I_{1,\rho}(\alpha,\beta,z)&=\frac{1}{2\alpha\pi i}\int_{C^+_{\rho,\varphi}}
\frac{\zeta^{(1-\beta)\alpha}e^{\zeta^{1/\alpha}}}{\zeta-z}d\zeta,\\
I_{2,\rho}(\alpha,\beta, z)&=-\frac{1}{2\alpha\pi i}\int_{C^-_{\rho,\varphi}}
\frac{\zeta^{(1-\beta)\alpha}e^{\zeta^{1/\alpha}}}{\zeta-z}d\zeta,\\
I_{3,\rho}(\alpha,\beta, z)&=\frac{1}{2\alpha\pi i}\int_{C_{\rho,\varphi}}
\frac{\zeta^{(1-\beta)\alpha}e^{\zeta^{1/\alpha}}}{\zeta-z}d\zeta.\\
\end{align*}
For every $z\in \mathbb{R}$ and $|z-\rho|>\rho_0$ we can find a $C=C(\varphi_0,\rho_0,\rho_1)$ such that
\begin{equation}
  \frac{1}{|\zeta-z|}\leq \frac{C}{1+|z|}\ \   {\rm for\ every} \ \zeta\in\gamma_{\rho,\varphi}.
\label{denominator-upper-bound}
  \end{equation}

Using  Lemma \ref{upper-bound}, we can verify directly by the Lebesgue dominated convergence theorem that the functions $I_{j,\rho}$, $j=1,2,3,$ are differentiable with respect to the parameters $\alpha,\beta, z$. Moreover, for $j=1,2,3$ we claim that
\begin{align*}
 |I_{j,\rho}(\alpha,\beta,z)|+
\left|\frac{\partial I_{j,\rho}}{\partial \alpha}(\alpha,\beta,z)\right|
+\left|\frac{\partial I_{j,\rho}}{\partial \beta}(\alpha,\beta,z)\right|&\leq \frac{C}{1+|z|},\\
\left|\frac{\partial I_{j,\rho}}{\partial z}(\alpha,\beta,z)\right|&\leq
\frac{C}{(1+|z|)^2},
 \end{align*}
where $z\in\mathbb{R}$, $z\leq z_0$, $C=C(\varphi_0,\rho_0, \rho_1,\alpha_0,\alpha_1,\beta_0,\beta_1,z_0)$.
Here, to figure out the idea of proving,  we give an outline of the proof for the claim. We note that
	$$  \left|\exp({r^{1/\alpha}e^{\pm i\varphi/\alpha}})\right|=\exp\left(r^{1/\alpha}\cos\left(
	\frac{\varphi}{\alpha}\right) \right) . $$
	We have
	\begin{eqnarray*}
		I_{1,\rho}(\alpha,\beta,z) &=&\frac{e^{i\varphi}}{2\alpha\pi i}
		\int_\rho^\infty
		\frac{
			r^{(1-\beta)\alpha}
			e^{i\varphi(1-\beta)\alpha}
			\exp(r^{1/\alpha} e^{i\varphi/\alpha})
		}
		{re^{i\varphi}-z}dr,\\
		I_{2,\rho}(\alpha,\beta,z)&= &-\frac{e^{-i\varphi}}{2\alpha\pi i}
		\int_\rho^\infty
		\frac{
			r^{(1-\beta)\alpha}
			e^{-i\varphi(1-\beta)\alpha}
			\exp(r^{1/\alpha} e^{-i\varphi/\alpha})
		}
		{re^{i\varphi}-z}dr,\\
		I_{3,\rho}(\alpha,\beta,z)&= &\frac{1}{2\alpha\pi }
		\int_{-\varphi}^\varphi
		\frac {\rho^{(1-\beta)\alpha}e^{i(1-\beta)\alpha\theta}\exp(\rho e^{i\theta/\alpha})}
		{ \rho e^{i\theta}-z} \rho e^{i\theta}d\theta.
	\end{eqnarray*}
	Now, we have
	$$|I_1(\alpha,\beta,z)|\leq \frac{1}{2\alpha_0\pi }
	\int_\rho^\infty
	\frac{
		r^{(1-\beta)\alpha}
		g_0(r)
	}
	{|re^{i\varphi}-z|}dr.$$
	Using Lemma \ref{upper-bound} and (\ref{denominator-upper-bound}), we obtain
	$$|I_{1,\rho}(\alpha,\beta,z)|\leq \frac{C}{2\alpha_0\pi }
	\int_\rho^\infty
	\frac{
		r^{(1+|\beta_0|)\alpha_1}
		g_0(r)
	}
	{1+|z|}dr.$$
	It follows that
	$$|I_{1,\rho}(\alpha,\beta,z)|\leq\frac{C}{1+|z|}\ \ \ \ {\rm for }\ z\in \mathbb{R}.$$
	Similarly
	$$|I_{2,\rho}(\alpha,\beta,z)|\leq\frac{C}{1+|z|}\ \ \ \ {\rm for }\ z\in \mathbb{R}.$$
	Now, we estimate $I_{3,\rho}(\alpha,\beta,z)$. We have
	$$ |I_{3,\rho}(\alpha,\beta,z)|\leq \frac{\rho}{2\alpha\pi }
	\int_{-\varphi}^\varphi
	\frac {\rho^{(1-\beta)\alpha}\exp\left(\rho^{1/\alpha}\cos\left(\frac{\theta}{\alpha}\right)\right)}
	{ |\rho e^{i\theta}-z|} d\theta\leq
	\frac{1}{2\alpha\pi }
	\int_{-\varphi}^\varphi
	\frac {e}
	{ |\rho e^{i\theta}-z|} d\theta.
	$$
	From (\ref{denominator-upper-bound}), we have therefore
	$$ |I_{3,\rho}(\alpha,\beta,z)|\leq \frac{C}{1+|z|}\ \ \ \ {\rm for\ every}\ z\in\mathbb{R}, |z-\rho|\geq\epsilon_0>0. $$
	Now we consider the derivatives of $I_{j,\rho}$, $j=1,2,3$. We first consider $I_{1,\rho}(\alpha,\beta,z)$. Put the integrand of $I_1(\alpha,\beta)$ by
	$$  F(\alpha,\beta,r,z)=\frac{
		r^{(1-\beta)\alpha}
		e^{i\varphi(1-\beta)\alpha}
		\exp(r^{1/\alpha} e^{i\varphi/\alpha})
	}
	{re^{i\varphi}-z}.$$
	We have
	\begin{eqnarray*}
		(re^{i\varphi}-z)\frac{\partial F}{\partial \alpha}(\alpha,\beta,r,z)&=&
		r^{(1-\beta)\alpha}(1-\beta)
		e^{i\varphi(1-\beta)\alpha}
		\exp(r^{1/\alpha} e^{i\varphi/\alpha})\ln r \\
		& &+
		r^{(1-\beta)\alpha}
		i\varphi(1-\beta)e^{i\varphi(1-\beta)\alpha}
		\exp(r^{1/\alpha} e^{i\varphi/\alpha})\\
		& &+
		r^{(1-\beta)\alpha}
		e^{i\varphi(1-\beta)\alpha}
		\left(-\frac{1}{\alpha^2}\right)(r^{1/\alpha} e^{i\varphi/\alpha}\ln r+r^{1/\alpha}i\varphi e^{i\varphi/\alpha} )\exp(r^{1/\alpha} e^{i\varphi/\alpha}).
	\end{eqnarray*}
	Using Lemma \ref{upper-bound}, we get after some rearrangements
	$$ (1+|z|)\left|\frac{\partial F}{\partial \alpha}(\alpha,\beta,r,z)\right|\leq C(1+\ln r)r^{1+|\beta_0|\alpha_1}g_0(r)
	\ \ \ {\rm for}\ r\geq 1.  $$
	As mentioned in Lemma \ref{upper-bound},  the function in the left hand side of the latter inequality is Lebesgue integrable on $[1,\infty)$. Hence,
	the Lebesgue dominated convergence theorem gives
	$$ \frac{\partial I_{1,\rho}}{\partial\alpha}(\alpha,\beta,z)=\frac{e^{i\varphi}}{2\alpha\pi i}\int_1^\infty
	\frac{\partial F}{\partial \alpha}(\alpha,\beta,r,z)dr-\frac{e^{i\varphi}}{2\alpha^2\pi i}\int_1^\infty
	F(\alpha,\beta,r,z)dr$$
	and
	$$ \left|\frac{\partial I_{1,\rho}}{\partial\alpha}(\alpha,\beta,z)\right|\leq \frac{C}{1+|z|}.  $$
	Similarly, we can get
	$$ \left|\frac{\partial I_{1,\rho}}{\partial\beta}(\alpha,\beta,z)\right|\leq \frac{C}{1+|z|},
\ \left|\frac{\partial I_{1,\rho}}{\partial z}(\alpha,\beta,z)\right|\leq \frac{C}{(1+|z|)^2} .  $$
	Using the same argument, we can prove analogous inequality for $I_{2,\rho}(\alpha,\beta,z),I_{3,\rho}(\alpha,\beta,z)$. Combining the inequalities thus obtained, we get the desired results.

 Choosing $\rho>z_0$,  we have in view of Theorem 1.1 in \cite[Chap. 1, page 30]{Polubny}
\begin{equation*}
E_{\alpha,\beta}(z)=\frac{1}{2\alpha\pi i}\int_{\gamma_{\rho,\varphi}}
\frac{\zeta^{(1-\beta)\alpha}e^{\zeta^{1/\alpha}}}{\zeta-z}d\zeta\ \ \ \ \ \ {\rm for\ all}\ z<z_0.
\end{equation*}
It follows that

\begin{equation*}
E_{\alpha,\beta}(z)= I_{1,\rho}(\alpha,\beta,z)+I_{2,\rho}(\alpha,\beta,z)+I_{3,\rho}(\alpha,\beta,z).
\end{equation*}

Combining the inequalities for $I_{1,\rho} (\alpha,\beta,z), I_{2,\rho}(\alpha,\beta,z),I_{3,\rho}(\alpha,\beta,z)$ gives
the stated results. To prove the Lipschitz property, we use the mean value theorem and the proved inequalities to obtain
$$ |E_{\alpha,\beta}(z_1)-E_{\alpha,\beta}(z_2)|\leq \sup_{z\leq z_0}\Big|\frac{\partial E_{\alpha,\alpha}}{\partial z}(z)\Big|\ |z_1-z_2|\leq C|z_1-z_2|.  $$

\noindent {\bf Proof of (b):} For $z\geq z_1$, we choose a number $\rho_1\in (0,z_1)$. We have in view of Theorem 1.1
\cite[Chap. 1, page 30]{Polubny}
\begin{equation*}
 E_{\alpha,\beta}(z)=\frac{1}{\alpha}z^{(1-\beta)\alpha}e^{z^{1/\alpha}}+\frac{1}{2\alpha\pi i}\int_{\gamma_{\rho_1,\varphi}}
\frac{\zeta^{(1-\beta)\alpha}e^{\zeta^{1/\alpha}}}{\zeta-z}d\zeta\ \ \ \ \ \ {\rm for\ all}\ z>z_1.
\end{equation*}
Using the part (a), we obtain (b).

\noindent {\bf Proof of (c):} Put $G(z,\alpha,\beta)=E_{\alpha,\beta}(z)\phi_0(\alpha,\beta,z)^{-1}$.
From Part (b) we have for $z\geq 1$
\begin{eqnarray*}
 |G(z,\alpha,\beta)-1| &\leq &
\frac{C\max\{z^{(1-\beta_0)\alpha},z^{(1-\beta_1)\alpha} \}}{(1+|z|)e^{z^{1/\alpha}}}\\
&\leq & \frac{C(z^{(1-\beta_0)\alpha}+z^{(1-\beta_1)\alpha} \}}{(1+|z|)e^{z^{1/\alpha}}}\\
&\leq & \frac{C(z^{(1-\beta_0)\alpha_0}+z^{(1-\beta_0)\alpha_1}+z^{(1-\beta_1)\alpha_0}+z^{(1-\beta_1)\alpha_1} \}}{(1+|z|)e^{z^{1/\alpha_1}}}:=\psi(z).
\end{eqnarray*}
Since $\lim_{z\to+\infty}\psi(z)=0$, we can find an $M>z_1>0$ independent of $z,\alpha,\beta$ such that $0\leq \psi(z)\leq \frac{1}{2}$ for $z\geq M$. It follows that
$$        \frac{1}{2}\leq G(z,\alpha,\beta)\leq \frac{3}{2}\ \ \ \ {\rm for}\ z\geq M.$$
Now, put $D=[z_1,M]\times [\alpha_0,\alpha_1]\times [\beta_0,\beta_1]$, $c^-=\inf_D G(z,\alpha,\beta)$, $c^+=\sup_D G(z,\alpha,\beta)$. Using compactness argument, we obtain
$$c^-=\min_D G(z,\alpha,\beta)>0,\  c^+=\max_D G(z,\alpha,\beta)>0$$
and
$$   c^-\leq G(z,\alpha,\beta)\leq c^+\ \ \ \ {\rm for}\ (z,\alpha,\beta)\in D. $$
Putting $C^-=\min\{c^-,\frac{1}{2}\}, C^+=\max\{c^+,\frac{3}{2}\}$, we have
$$        C^-\leq G(z,\alpha,\beta)\leq C^+\ \ \ \ {\rm for}\ z\geq z_1.$$
Hence the desired result follows. The case $\beta=1$ is similar (in fact, easier).

\noindent {\bf Proof of (d):} From the definition we have $E_{\alpha,\alpha}(z)\geq 0$ for $z\geq 0$. For $z<0$, the proof can be found in \cite{Miller-Samko} which uses the fact that the Mittag-Leffler function is completely monotonic.

\noindent {\bf Proof of (e):} The  proof can be found in Simon \cite{simon} for the inequality.

\noindent {\bf Proof of (f):}  We have
\begin{equation*}
E_{\alpha,1}(z) = \sum_{k=0}^\infty \frac{z^k}{\Gamma(k\alpha+1)}\quad z\in \mathbb{C}.
\end{equation*}
Hence
\begin{equation*}
\frac{d}{dz}E_{\alpha,1}(z) = \sum_{k=1}^\infty \frac{kz^{k-1}}{k\alpha\Gamma(k\alpha)}=
\frac{1}{\alpha}\sum_{k=1}^\infty \frac{z^{k-1}}{\Gamma(k\alpha)}=\frac{1}{\alpha}E_{\alpha,\alpha}(z).
\end{equation*}

\noindent {\bf Proof of (g):} See \cite[page 63]{Gorenflo-Kilbas-Mainardi-Rogosin}.

\noindent {\bf Proof of (h):} Putting
\begin{equation*}
\psi(t)=\varphi(t)-\frac{\lambda}{\Gamma(\alpha)} \int_0^t\frac{\varphi(s)}{(t-s)^{1-\alpha}}ds,
\end{equation*}
 we obtain $\psi(t)\leq g(t)$. We deduce in view of the part (g)
\begin{eqnarray*}
\varphi(t)&=&\psi(t)+ \lambda\int_0^t E_{\alpha,\alpha}(\lambda(t-s)^\alpha)\psi(s)ds\\
&\leq&g(t)+ \lambda\int_0^t E_{\alpha,\alpha}(\lambda(t-s)^\alpha)g(s)ds.
\end{eqnarray*}
Now, we prove that last inequality of the lemma. The case $p=1$ and $p=\infty$ can be proved easily. Hence we omit it.  We consider the case $1< p <\infty$, $g\in L^p(0,T)$. Putting $q$ such that $\frac{1}{p}+\frac{1}{q}=1$, using H\"older's inequality we can estimate directly
\begin{eqnarray*}
\varphi(t)
&\leq&g(t)+ \lambda\int_0^t E_{\alpha,\alpha}(\lambda(t-s)^\alpha)g(s)ds\\
&\leq& g(t)+\lambda E_{\alpha,\alpha}(\lambda T^\alpha)\int_0^t g(s)ds\\
&\leq& g(t)+\lambda E_{\alpha,\alpha}(\lambda T^\alpha)t^{1/q}\left(\int_0^t |g(s)|^pds\right)^{1/p}\\
&\leq &g(t)+\lambda E_{\alpha,\alpha}(\lambda T^\alpha)T^{1/q}\Vert g\Vert_{L^p(0,T)}.
\end{eqnarray*}
So we have
\begin{eqnarray*}
 \Vert \varphi\Vert_{L^p(0,T)}&\leq &\Vert g\Vert_{L^p(0,T)}+
\lambda E_{\alpha,\alpha}(\lambda T^\alpha)T^{1/q}\Vert g\Vert_{L^p(0,T)}T^{1/p}  \\
&\leq & (1+\lambda T E_{\alpha,\alpha}(\lambda T^\alpha))\Vert g\Vert_{L^p(0,T)}.
\end{eqnarray*}

\end{proof}

We note that the final inequality of the lemma will be used to estimate the solution of fractional problems in many cases. This inequality is an extension  of the Gronwall inequality. Some other  results  of interest for the Mittag-Leffler functions can be found in
\cite{Gorenflo-Kilbas-Mainardi-Rogosin}  or \cite{Liu}.

\section{Continuity of the solutions of some fractional differential equations}

\setcounter{equation}{0}

In this section, we will investigate the continuity of solutions of   a class of
general fractional differential equations with respect to the fractional parameter $\alpha\in (0,1)$. First, we will transform these equations into  general Abel equations and then study  their properties.
The investigation on fractional differential equations with sequential derivatives can be applied directly to
other equations with the Riemann-Liouville, or  the Caputo fractional derivatives.
Hence, we start with the general equation. Let $\sigma_0=0<\sigma_1<...<\sigma_k$ satisfy $0<\sigma_{j}-\sigma_{j-1}\leq 1$, $j=1,...,k$. The fractional differential equation with sequential derivatives  reads as
\begin{equation}
	{\cal D}^{\sigma_k}_t y(t)+\sum_{j=1}^kp_j(t){\cal D}^{\sigma_{k-j}}_t y(t)+p_k(t)y(t)=f(t),\ \ \  0<t\leq T,
	\label{general-FDE}
\end{equation}
subject to the conditions
\begin{equation*}
	\left. {\cal D}^{\sigma_j-1}_ty(t)\right|_{t=0}=b_j,\ \ \ \ \ j=1,...,k.
\end{equation*}
In view of \cite[page 122]{Polubny}, the solution of the equation  ${\cal D}^{\sigma_k}_t y(t)=\psi(t)$  is given by
\begin{equation*}
	y(t)=\sum_{j=1}^kb_j\frac{t^{\sigma_j-1}}{\Gamma(\sigma_j)}+\frac{1}{\Gamma(\sigma_k)}\int_0^t(t-s)^{\sigma_k-1}\psi(s)ds.
\end{equation*}
Using this  equality we can rewrite the equation (\ref{general-FDE}) as
\begin{equation}
	\label{general-Abel-Equation}
	\psi(t)+\int_0^t \frac{K(t,s,\beta, \psi(s))}{(t-s)^{1-\sigma_k}}ds=g(t)
\end{equation}
where $\eta_j:=\sigma_j-\sigma_{j-1}$, $\beta:=(\eta_1,...,\eta_k)$ and
\begin{align*}
	K(t,s,\beta,\psi) &= \left(p_k(t)\frac{(t-s)^{\sigma_k-\eta_k}}{\Gamma(\sigma_k)}+
	\sum_{j=1}^{k-1}p_{k-j}(t)\frac{(t-s)^{\sigma_k-\sigma_j-\eta_k}}{\Gamma(\sigma_k)}\right)\psi,\\
	g(t) &= f(t)-p_k(t)\sum_{j=1}^k\frac{b_jt^{\sigma_j-1}}{\Gamma(\sigma_j)}-
	\sum_{j=1}^{k-1}p_{k-j}(t)\sum_{\ell=j+1}^k\frac{b_\ell t^{\sigma_\ell-\sigma_j-1}}{\Gamma(\sigma_\ell-\sigma_j)}.
\end{align*}
Put $\nu=\min\{\eta_1,...,\eta_k\}$, under mild conditions on the functions $f, $ and $p_k$'s the function $g^*(t)=t^{1-\nu}g(t)$ is a
continuous function on $[0,T]$. By this fact,  we consider our problem in the following  space
\begin{equation*}
	C_\gamma(T,\mathbb{X}) =\Big \{v\in C((0,T],\mathbb{X}):\  \sup_{0<t\leq T} t^\gamma\Vert v(t)\Vert<\infty \Big \}
\end{equation*}
where $0<\gamma<1$, $(\mathbb{X}, ||\cdot||)$ is a Banach space. The space $C_\gamma(T,\mathbb{X})$ is a Banach space with the norm
$$\Vert v\Vert_{C_\gamma}=\sup_{0<t\leq T} t^\gamma\Vert v(t)\Vert. $$
From the definition of the norm, we deduce an inequality which will be used often  in the the rest of our paper
$$    \Vert v(t)\Vert\leq t^{-\gamma}\Vert v\Vert_{C_\gamma(T)}.  $$
For convenience, we denote $C([0,T];\mathbb{X})$ by $C_0(T;\mathbb{X})$.

\subsection{Some properties of solutions of the generalized Abel equations of the second kind}

{We will establish  existence and continuity of  solutions of the genralized Abel equations of second kind. The main results are given in Theorems  \ref{existence-Abel-equation} and \ref{main-theorem-Lipschitz-Abel-equation}.}

For $k\in \mathbb{N}$, we denote by $P$ a compact subset in $\mathbb{R}^k$.
Letting $T>0, 0< \alpha_0<\alpha_1$, we put
\begin{equation*}
	\Delta_T=\Big\{(t,s,\alpha,z):\ 0\leq s\leq t\leq T,\ \alpha_0\leq \alpha\leq\alpha_1,z \in P \Big\}.
\end{equation*}
Assume that
\begin{equation*}
	K
	: \Delta_T \times \mathbb{X}\to \mathbb{X},\  g: (0,T]\times [\alpha_0,\alpha_1]
	\times P\to \mathbb{X}.
\end{equation*}
Suggested by the integral form of the general fractional differential equations, we  consider the nonlinear Abel integral equation of second kind of finding $u_{\alpha,z}: (0,T]\to \mathbb{X}$ that satisfy the following
\begin{equation}\label{firs-equation}
u_{\alpha,z}(t)=g(t,\alpha,z)+ \int_0^t \frac{K(t,s,\alpha,z, u_{\alpha,z}(s))}{(t-s)^{1-\alpha}}ds.
\end{equation}
For every $\alpha\in [\alpha_0,\alpha_1], v\in C_\gamma(T,\mathbb{X})$, we put
\begin{equation*}
	A_{\alpha,z} v(t)=\int_0^t \frac{K(t,s,\alpha,z,v(s))}{(t-s)^{1-\alpha}}ds.
\end{equation*}
To emphasize the dependence  on  the parameter $(\alpha,z)$, we also denote $g(t,\alpha,z)$ by $g_{\alpha,z}(t)$.
With these notations, we can write the  equation \eqref{firs-equation} as
\begin{equation}
	u_{\alpha,z}(t)=g_{\alpha,z}(t)+ A_{\alpha,z}u_{\alpha,z}(t).
	\label{generalized-Abel-equation}
\end{equation}
For $p,q>0$, recalling the definition of Beta function
$ B(p,q)=\int_0^1(1-\theta)^{p-1}\theta^{q-1}d\theta,$
we have
$$
\int_0^t(t-s)^{p-1}s^{q-1}ds=t^{p+q-1}B(p,q)=t^{p+q-1}\frac{\Gamma(p)\Gamma(q)}{\Gamma(p+q)}.
$$
In the next lemma we  establish some estimates of singular integrals.

\begin{lemma}\label{singular-integral}

	\begin{enumerate}[{\bf \upshape(a)}]
		
		\item   For $h\in [0,1]$, $T>0$,  $0<\nu_0\leq\nu_1$, $\eta,\nu\in [\nu_0,\nu_1]$, $p\geq 1$, there exists a constant $C=C(\nu_0,\nu_1, T)$ such that
		\begin{equation*}
			\int_0^ts^{\eta-1}(t-s)^{\nu-1}\Big(|s^h-1|^p+|(t-s)^h-1|^p\Big)ds\leq Cht^{\eta+\nu-1}(1+|\ln t|^p)\ \ \ \ {\rm for\ every}\ 0\leq t\leq T.
		\end{equation*}
		
		\smallskip
		
		\item
		
		Let $t\in(0,T]$. For $v\in C_\gamma(T;\mathbb{X})$,  we have
		$$ \Vert J^\alpha v(t)\Vert\leq  J^\alpha \Vert v(t)\Vert
		\leq t^{\alpha-\gamma}\frac{\Vert v\Vert_{C_\gamma(T;\mathbb{X})}}{\Gamma(\alpha)}B(\alpha,1-\gamma)   $$
		and for $v\in L^p(0,T;X)$, $1\leq p\leq \infty$,
		$$ \Vert J^\alpha v(t)\Vert_{L^p(0,T;\mathbb{X})}\leq \Vert J^\alpha \Vert v(t)\Vert\Vert_{L^p(0,T)}
		\leq\frac{T^\alpha}{\alpha\Gamma(\alpha)}\Vert v\Vert_{L^p(0,T;\mathbb{X})}.$$
		Let $\mu_0, h\in (0,1]$. If $v\in C([0,T];\mathbb{X})$ is H\"{o}lder, i.e., there exist $\kappa>0, \mu\in (\mu_0,1)$ such that
		$\Vert v(t)-v(s)\Vert\leq \kappa |t-s|^\mu$ for every $t,s\in [0,T]$, then there exists a constant $C(\mu_0)$ independent of $v,\mu$
		such that
		$$  \Vert J^hv(s)-v(s)\Vert \leq C(\mu_0)(\Vert v\Vert_{C([0,T];\mathbb{X})}+[v]_{\mu})(h+|s^h-1|)$$
		where $[v]_{\mu}=\sup_{0\leq t\not=s\leq T}\frac{\Vert v(t)-v(s)\Vert}{|t-s|^\mu}$.
		We also have
		\begin{eqnarray*}
			\lim_{h\to 0^+}\Vert J^h w-w\Vert_{L^p(0,T;\mathbb{X})}&=&0,\ \ \ \ \forall p\in [1,\infty), w\in L^p(0,T;\mathbb{X}),\\
			\Vert J^{\alpha'}w-J^\alpha w\Vert_{L^p(0,T;\mathbb{X})} &\leq & C(\alpha_0)\Vert w\Vert_{L^p(0,T;\mathbb{X})}|\alpha'-\alpha|,\ \ \ \forall \alpha_0\in (0,1],\ \alpha',\alpha\in [\alpha_0,1].
		\end{eqnarray*}

		\item
		Let $K\in C(\Delta_T\times\mathbb{X};\mathbb{X})$, $K=K(t,s,\alpha,z,w)$. We assume that
		$K$ is Lipschitz with respect to
		the variable $w\in \mathbb{X}$, i.e., there exists a $\kappa>0$ such that
		\begin{equation}\label{Lipschitz-1}
			\Vert K(t,s,\alpha,z,w_1)-K(t,s,\alpha,z,w_2)\Vert\leq \kappa \Vert w_1-w_2\Vert\ \ \ \ {\rm for\ every\ }w_1,w_2\in \mathbb{X}.
		\end{equation}
		Put $M_0=\sup_{(t,s,\alpha,z)\in\Delta_T}\Vert K(t,s,\alpha,z,0)\Vert$.
		For $0<t\leq T,\ v, v_1,v_2\in C_\gamma(T,\mathbb{X})$ we have $A_\alpha v\in C_\gamma(T,\mathbb{X}).$
		Moreover, we have
		\begin{align*}
			\Vert  A_{\alpha,z} v(t)\Vert &\leq \frac{M_0}{\alpha_0}t^{\alpha} +\kappa t^{\alpha-\gamma}B(\alpha,1-\gamma) \Vert v\Vert_{C_\gamma(T,\mathbb{X})},\\
			\Vert A_{\alpha,z} v_1(t)- A_{\alpha,z} v_2(t)\Vert &\leq
			\kappa t^{\alpha-\gamma} B(\alpha,1-\gamma)  \Vert v_1-v_2 \Vert_{C_\gamma(T,\mathbb{X})}
		\end{align*}
		and
		\begin{equation*}
			\Vert A_{\alpha,z} v_1-A_{\alpha,z}v_2\Vert_{C_\gamma(T,\mathbb{X})}\leq \kappa T^\alpha B(\alpha,1-\gamma)  \Vert v_1-v_2 \Vert_{C_\gamma(T,\mathbb{X})}.
		\end{equation*}
		\item
		Now, let $v\in L^p(0,T;\mathbb{X})$. For $1\leq p<\infty$, we have
		$$  \Vert  A_{\alpha,z} v\Vert_{L^p(0,T;\mathbb{X})} \leq \frac{M_0}{\alpha(p\alpha+1)^{1/p}}T^{\alpha+1/p} +\frac{\kappa T^\alpha}{\alpha}\Vert v\Vert_{L^p(0,T;\mathbb{X})}. $$
		For $p=\infty$, we have
		$$  \Vert  A_{\alpha,z} v\Vert_{L^\infty(0,T;\mathbb{X})} \leq \frac{M_0}{\alpha}T^{\alpha} +\frac{\kappa T^\alpha}{\alpha}\Vert v\Vert_{L^p(0,T;\mathbb{X})}. $$
		For $v_1,v_2\in L^p(0,T;\mathbb{X})$, $1\leq p\leq\infty$, we also have
		$$\Vert A_{\alpha,z} v_1- A_{\alpha,z} v_2\Vert_{L^p(0,T;\mathbb{X})}
		\leq  \frac{\kappa T^\alpha}{\alpha}\Vert v_1-v_2\Vert_{L^p(0,T;\mathbb{X})}.$$
		
		\item
		Let $u_1,u_2, g_1,g_2\in L^p(0,T;\mathbb{X})$ satisfy the equations $u_i=g_i+A_{\alpha,z}u_i$, $i=1,2$.
		Then we have
		$$\Vert u_2-u_1\Vert_{L^p(0,T;\mathbb{X})}\leq (1+\kappa\Gamma(\alpha)TE_{\alpha,\alpha}(\kappa
		\Gamma(\alpha)T^\alpha)) \Vert g_2-g_1\Vert_{L^p(0,T;\mathbb{X})}.$$
	\end{enumerate}	
	
\end{lemma}

\begin{proof}
	{\bf Proof of  (a):} The proof follows as
	\begin{eqnarray*}
		\lefteqn{\int_0^ts^{\eta-1}(t-s)^{\nu-1}(|s^h-1|^p+|(t-s)^h-1|^p)ds}\\
		&=& \int_0^ts^{\eta-1}(t-s)^{\nu-1}
		\left(\left|\int_0^hs^\mu\ln s d\mu\right|^p+\left|\int_0^h(t-s)^\mu\ln (t-s)d\mu\right|^p\right)ds\\
		&\leq& \int_0^ts^{\eta-1}(t-s)^{\nu-1}
		\left(\int_0^hs^\mu|\ln s| d\mu+\int_0^h(t-s)^\mu|\ln (t-s)|d\mu\right)^pds\\
		&\leq& C_0h\int_0^ts^{\eta-1}(t-s)^{\nu-1}
		(|\ln s|^p+|\ln (t-s)|^p)ds\\
		&\leq &C_0t^{\eta+\nu-1}h\int_0^1\theta^{\eta-1}(1-\theta)^{\nu-1}
		(|\ln \theta|^p+|\ln (1-\theta)|^p+2|\ln t|^p)d \theta\\
		&\leq & Cht^{\eta+\nu-1}(1+|\ln t|^p).
	\end{eqnarray*}
	
	{\bf Proof of  (b):} We verify the inequalities in Part (b). We have
	\begin{eqnarray*}
		\Vert J^\alpha v(t) \Vert &\leq&  J^\alpha \Vert v(t) \Vert\\
		&=&\frac{1}{\Gamma(\alpha)}\int_0^t(t-s)^{\alpha-1}s^{-\gamma}s^\gamma\Vert v(s)\Vert ds\\
		&\leq & \frac{\Vert v\Vert_{C_\gamma(T;\mathbb{X})}}{\Gamma(\alpha)}\int_0^t(t-s)^{\alpha-1}s^{(1-\gamma)-1} ds\\
		&=& \frac{\Vert v\Vert_{C_\gamma(T;\mathbb{X})}}{\Gamma(\alpha)}t^{\alpha-\gamma}B(\alpha, 1-\gamma).
	\end{eqnarray*}
	Now, if $v\in L^p(0,T;\mathbb{X})$, since we can prove the cases $p=1$ and $p=\infty$ easily, we only consider the case $1<p<\infty$. Let $q$ be such that $\frac{1}{p}+\frac{1}{q}=1$, we get
	\begin{eqnarray*}
		\Gamma(\alpha)\Vert \left(J^\alpha \Vert v(t)\Vert\right)\Vert_{L^p(0,T)} \leq  \int_0^t
		\frac{ds}{(t-s)^{1-\alpha}}
		\Vert v\Vert_{L^p(0,T;\mathbb{X})}
		\leq  \frac{T^{\alpha}}{\alpha}
		\Vert v\Vert_{L^p(0,T;\mathbb{X})}.
	\end{eqnarray*}
	Now, we consider the limit of $J^hv$ as $h\to 0^+$. Put $I(s)=\Vert J^hv(s)-v(s)\Vert$ and
	$e(\tau)=\frac{1}{\tau}\int_{s-\tau}^s v(\theta)d\theta-v(s)$.
	Using integration by parts we can write
	\begin{eqnarray*}
		I(s)&=&\frac{1}{\Gamma(h)}\int_0^s(s-\tau)^{h-1}v(\tau)d\tau-v(s) \\
		&=& \frac{1}{\Gamma(h)}s^{h-1}\int_0^s v(\tau)d\tau+ \frac{1-h}{\Gamma(h)}
		\int_0^s\tau^{h-2}\int_{s-\tau}^s v(\theta)d\theta d\tau-v(s)\\
		&=& \frac{1}{\Gamma(h)}s^{h-1}\int_0^s v(\tau)d\tau+\frac{1-h}{\Gamma(h)}
		\int_0^s \tau^{h-1}e(\tau)d\tau\\
		& &+ \left(\frac{1-h}{\Gamma(1+h)}s^h-1\right) v(s).
	\end{eqnarray*}
	{Next we estimate $e(\tau)$. Letting $\theta=s-\eta \tau$ gives
	\begin{eqnarray*}
		\Vert e(\tau)\Vert&=& \left\Vert\int_0^1 \left(v(s-\eta \tau)-v(s)\right) d\eta\right\Vert\\
		&\leq &[v]_\mu\int_0^1| s-(s-\tau\eta) |^{\mu} d\eta\leq[v]_\mu\tau^{\mu}.
	\end{eqnarray*}
}
	Therefore
	\begin{eqnarray*}
		\Vert I(s)\Vert
		&\leq& Ch\Vert v\Vert_{C([0,T];\mathbb{X})}+C \frac{[v]_\mu(1-h)}{\Gamma(h)}
		\int_0^s \tau^{h-1+\mu}d\tau+ \left(\frac{1-h}{\Gamma(1+h)}s^h-1\right) \Vert v(s)\Vert.
	\end{eqnarray*}
	By direct computation, we get
	\begin{equation*}
		\Vert I(s)\Vert\leq C\left(h+|1-s^h|\right)(\Vert v\Vert_{C([0,T];\mathbb{X})}+[v]_\mu).
	\end{equation*}
	Now, we prove the final equality of Part (b). Let $\epsilon>0$, since $C^1([0,T];\mathbb{X})$
	is dense in $L^p(0,T;\mathbb{X})$, we can find a function $v\in C^1([0,T];\mathbb{X})$ such that
	$ \Vert w-v\Vert_{L^p(0,T;\mathbb{X})}\leq \epsilon.$ From the estimate of $I(s)$ we obtain
	\begin{align*}
		\Vert J^h w-w\Vert_{L^p(0,T;\mathbb{X})} &\leq \Vert J^h (w-v)\Vert_{L^p(0,T;\mathbb{X})}+
		\Vert J^h v-v\Vert_{L^p(0,T;\mathbb{X})}+\\
		& \Vert v-w\Vert_{L^p(0,T;\mathbb{X})}\\
		&\leq C\epsilon+C \Vert v\Vert_{C^1([0,T];\mathbb{X})}\left(h+\left(\int_0^T |1-s^h|^pds
		\right)^{1/p}\right).
	\end{align*}
	Using Part (a), we obtain the claimed limit. Now we prove the last inequality of Part (b).
	We have in view of Part (a)
	\begin{align*}
		\Vert J^{\alpha'} w-J^\alpha w\Vert_{L^p(0,T;\mathbb{X})} &= \left\Vert\int_0^t\left((t-s)^{\alpha'-1}-
		(t-s)^{\alpha-1}\right)w(s)ds\right\Vert_{L^p(0,T;\mathbb{X})}\\
		&\leq \sup_{0<t\leq T}\int_0^t(t-s)^{\alpha-1}\left|
		(t-s)^{\alpha'-\alpha}-1\right|ds \Vert w\Vert_{L^p(0,T;\mathbb{X})}\\
		&\leq C|\alpha'-\alpha| \sup_{0<t\leq T}t^\alpha(1+|\ln t|)\Vert w\Vert_{L^p(0,T;\mathbb{X})}.
	\end{align*}
	Since $\alpha\in (\alpha_0,1]$ we get $\sup_{0<t\leq T}t^\alpha(1+|\ln t|)\leq C(\alpha_0)<\infty$. Hence
	$$  \Vert J^{\alpha'} w-J^\alpha w\Vert_{L^p(0,T;\mathbb{X})}\leq C(\alpha_0)|\alpha'-\alpha| \Vert w\Vert_{L^p(0,T;\mathbb{X})}.   $$
	{\bf Proof of  (c):}
	For $v\in\mathbb{X}$, using equation \eqref{Lipschitz-1} we have
	\begin{equation}
		\Vert K(t,s,\alpha,z,v)\Vert \leq M_0+\kappa \Vert v\Vert.
		\label{K-inequality}
	\end{equation}
	Hence, we obtain in view of Part a) and the inequality $\alpha_0\leq \alpha$ that
	\begin{eqnarray*}
		\Vert A_{\alpha,z} v(t)\Vert&\leq &
		\Gamma(\alpha)(J^\alpha M_0+J^\alpha \kappa\Vert v(t)\Vert)\\
		&\leq & \frac{M_0}{\alpha_0}t^\alpha +\kappa t^{\alpha-\gamma}B(\alpha,1-\gamma) \Vert v\Vert_{C_\gamma(T,\mathbb{X})}.
	\end{eqnarray*}
	Similarly we  obtain the following estimates
	\begin{eqnarray*}
		\Vert A_{\alpha,z} v_1(t)-A_{\alpha,z} v_2(t)\Vert&\leq&\Gamma(\alpha)\kappa J^\alpha \Vert v_1(t)-v_2(t)\Vert
		\leq \kappa t^{\alpha-\gamma}\Vert v\Vert_{C_\gamma(T,\mathbb{X})}B(\alpha,1-\gamma),\\
		\Vert A_{\alpha,\beta} v(t)-A_{\alpha,\beta} (0)\Vert&\leq&\Gamma(\alpha)\kappa J^\alpha \Vert v(t)\Vert
		\leq\kappa t^{\alpha-\gamma}\Vert v\Vert_{C_\gamma(T,\mathbb{X})}B(\alpha,1-\gamma).
	\end{eqnarray*}
	{\bf Proof of  (d):} We have
	\begin{eqnarray*}
		\Vert A_{\alpha,z} v\Vert_{L^p(0,T;\mathbb{X})}&\leq &
		\Gamma(\alpha)\Vert J^\alpha M_0\Vert_{L^p(0,T)}
		+\Gamma(\alpha)\Vert J^\alpha \kappa\Vert v(t)\Vert\Vert_{L^p(0,T;\mathbb{X})}\\
		&\leq & \frac{M_0T^{\alpha+1/p}}{\alpha(p\alpha+1)^{1/p}}+
		\frac{T^{\alpha}}{\alpha}\Vert v\Vert_{L^p(0,T;\mathbb{X})}.
	\end{eqnarray*}
	We also have
	\begin{eqnarray*}
		\Vert A_{\alpha,z} v_1(t)- A_{\alpha,z} v_2(t)\Vert &\leq&
		\kappa \Gamma(\alpha)J^\alpha \Vert v_1(t)-v_2(t)\Vert.
	\end{eqnarray*}
	Therefore, we obtain in view of Part (a)
	$$\Vert A_{\alpha,z} v_1- A_{\alpha,z} v_2\Vert_{L^p(0,T;\mathbb{X})}
	\leq  \frac{\kappa T^\alpha}{\alpha}\Vert v_1-v_2\Vert_{L^p(0,T;\mathbb{X})}.$$
	{\bf Proof of  (e):} We have
	$$  \Vert u_2(t)-u_1(t)\Vert\leq \Vert g_2(t)-g_1(t)\Vert+\kappa\int_0^t\frac{\Vert u_2(s)-u_1(s)\Vert}{(t-s)^{1-\alpha}}.$$
	Applying Lemma \ref{Mittag-Leffler} (h) with $\lambda=\kappa \Gamma(\alpha)$ gives
	$$  \Vert u_2-u_1\Vert_{L^p(0,T;\mathbb{X})}\leq (1+\kappa\Gamma(\alpha)TE_{\alpha,\alpha}(\kappa
	\Gamma(\alpha)T^\alpha)) \Vert g_2-g_1\Vert_{L^p(0,T;\mathbb{X})}. $$
\end{proof}

The following theorem gives the existence    of  solutions of  the nonlinear integral equation  (\ref{generalized-Abel-equation}) in $ C_\gamma(T)$ and $L^p(0,T;\mathbb{X})$. \\

\begin{theorem}\label{existence-Abel-equation}
	Let $0<\alpha_0<\alpha_1$, $\alpha\in [\alpha_0,\alpha_1]$ and $\gamma\in [0,1)$.
	Let $K\in C(\Delta_T\times\mathbb{X};\mathbb{X})$, $K=K(t,s,\alpha,z,v)$. We assume that
	$K$ is a Lipschitz function with respect to
	the variable $v\in \mathbb{X}$, i.e., there exists a $\kappa>0$ such that
	\begin{equation*}
		\Vert K(t,s,\alpha,z,v_1)-K(t,s,\alpha,z,v_2)\Vert\leq \kappa \Vert v_1-v_2\Vert\ \ \ \ {\rm for\ every\ }v_1,v_2\in \mathbb{X}.
	\end{equation*}
	
	\begin{enumerate}[\bf {\upshape(a)}]
		
		\item   If  $g\in C\left([\alpha_0,\alpha_1]\times P; C_\gamma(T,\mathbb{X})\right)$ then the nonlinear equation (\ref{generalized-Abel-equation}) has a unique solution
		$u_{\alpha,z}\in C_\gamma(T)$ such that
		\begin{equation*}
			\Vert u_{\alpha,z}\Vert_{C_\gamma(T,\mathbb{X})}\leq \Gamma(1-\gamma)E_{\alpha,1-\gamma}(\kappa\Gamma(\alpha)T^\alpha) \Vert g_{\alpha,z}^*\Vert_{C_\gamma(T,\mathbb{X})}
		\end{equation*}
		where
		\begin{equation*}
			g_{\alpha,z}^*(t)=g_{\alpha,z}(t)+\int_0^t \frac{K^*(t,s,\alpha,z,0)}{(t-s)^{1-\alpha}}ds.
		\end{equation*}
		
		\smallskip

		\item   If $g_{\alpha,z}\in L^p(0,T;\mathbb{X})$ then the nonlinear equation (\ref{generalized-Abel-equation}) has a unique solution $u_{\alpha,z}\in L^p(0,T;\mathbb{X})$.
		Moreover, we have the following estimate
		\begin{equation*}
			\Vert u_{\alpha,z}\Vert_{L^p(0,T;\mathbb{X})}\leq \Big(1+\kappa T\Gamma(\alpha) E_{\alpha,\alpha}(\kappa\Gamma(\alpha) T^\alpha)\Big)
			\Vert g^*_{\alpha,z}\Vert_{L^p(0,T;\mathbb{X})}.
		\end{equation*}

	\end{enumerate}
	
\end{theorem}

\begin{proof}
	{\bf Proof of  (a):}
	Putting $$K^*(t,s,\alpha,z,w)=K(t,s,\alpha,z,w)-K(t,s,\alpha,z,0), $$
	we obtain $K^*(t,s,\alpha,z,0)\equiv 0$ and, for $w_1,w_2\in \mathbb{X}$,
	$$ \Vert K^*(t,s,\alpha,z,w_1)-K^*(t,s,\alpha,z,w_2)\Vert\leq \kappa \Vert w_1-w_2\Vert. $$
	Defining
	$$  A^*_{\alpha,z} u=\int_0^t \frac{K^*(t,s,\alpha,z,u(s))}{(t-s)^{1-\alpha}}ds, $$
	we can rewrite (\ref{generalized-Abel-equation}) in the operator form
	$$  (I-A^*_{\alpha,z})u_{\alpha,z}=g_{\alpha,z}^*. $$
	We verify the uniqueness of solution of the problem.
	Letting $v,w\in C_\gamma(T,\mathbb{X})$ be two solutions of the latter problem, we obtain
	\begin{align*}
		\Vert v(t)-w(t)\Vert &\leq \int_0^t\frac{\Vert K(t,s,\alpha,z,v(s))-K(t,s,\alpha,z,w(s))\Vert}{(t-s)^{1-\alpha}}ds\\
		&\leq \int_0^t\frac{\kappa\Vert v(s)-w(s)\Vert}{(t-s)^{1-\alpha}}ds.
		\label{uniqueness-Abel-inequality}
	\end{align*}
	From Lemma \ref{Mittag-Leffler}, Part (h), we obtain
	$\Vert v(t)-w(t)\Vert=0$, or $v=w$.
	
	Now we prove the existence of solution of the problem.
	Putting $u_0=0,\ u_1=g_{\alpha,z}^*,\ u_{n+1}=g_{\alpha,z}^*+A^*_{\alpha,z}u_n$,
	we claim that  the series $u_n$ converges in $C_\gamma(T,\mathbb{X})$ and its limit is the solution of (\ref{generalized-Abel-equation}).
	In fact, putting $M_1=\kappa\Gamma(\alpha)$, we can prove by induction that
	\begin{equation*}
		\Vert u_{n}(t)-u_{n-1}(t)\Vert\leq \frac{\Vert g^*\Vert_{C_\gamma(T,\mathbb{X})}\Gamma(1-\gamma)M_1^nt^{n\alpha-\gamma}}{\Gamma(n\alpha+1-\gamma)}\ \ \ \ {\rm for}\ n=1,2,....
	\end{equation*}
	
	For $n=1$, we have in view of Lemma \ref{singular-integral}, Part (c),
	\begin{align*}
		\Vert u_1(t)-u_0(t)\Vert&\leq  \int_0^t \frac{\Vert K^*(t,s,\alpha,z, g_{\alpha,\beta}^*(s))\Vert}{(t-s)^{1-\alpha}} ds\\
		&\leq  \kappa \Vert g_{\alpha,z}^*\Vert_{C_\gamma(T,\mathbb{X})}B(\alpha,1-\gamma)t^{\alpha-\gamma}\\
		&\leq \frac{\Vert g_{\alpha,z}^*\Vert_{C_\gamma(T,\mathbb{X})}M_1\Gamma(1-\gamma)t^{\alpha-\gamma}}{\Gamma(n\alpha+1-\gamma)}.
	\end{align*}
	
	Assume that the inequality holds for $n=k$ $(k\geq 1)$, we claim that it holds for $n=k+1$.
	To this end, we note that
	\begin{eqnarray*}
		\Vert  u_{k+1}(t)-u_{k}(t)\Vert&\leq&\int_0^t \frac{\Vert K^*(t,s,\alpha,z,u_{k}(s))-K^*(t,s,\alpha,z,u_{k-1}(s))\Vert}{(t-s)^{1-\alpha}} ds\\
		&\leq& \kappa\int_0^t \frac{\Vert u_{k}(s)-u_{k-1}(s)\Vert}{(t-s)^{1-\alpha}} ds.
	\end{eqnarray*}
	Using the induction assumptions, we deduce
	\begin{eqnarray*}
		\Vert u_{k+1}(t)-u_{k}(t)\Vert  &\leq&
		\frac{\Vert g_{\alpha,z}^*\Vert_{C_\gamma(T,\mathbb{X})}\Gamma(1-\gamma)M_1^{k+1}}{\Gamma(\alpha)\Gamma(k\alpha+1-\gamma)}\int_0^t (t-s)^{\alpha-1}s^{k\alpha+(1-\gamma)-1}ds\\
		&\leq &\ \frac{\Vert g_{\alpha,z}^*\Vert_{C_\gamma(T,\mathbb{X})}\Gamma(1-\gamma)M_1^{k+1}}{\Gamma(\alpha)}\frac{t^{(k+1)\alpha-\gamma} B(\alpha, k\alpha+(1-\gamma))}{\Gamma(k\alpha+1-\gamma)}\\
		&\leq& \frac{\Vert g_{\alpha,z}^*\Vert_{C_\gamma(T,\mathbb{X})}\Gamma(1-\gamma)M_1^{k+1}t^{(k+1)\alpha-\gamma}}
		{\Gamma((k+1)\alpha+1-\gamma)}.
	\end{eqnarray*}
	The induction principle implies that the inequality holds for every $n=1,2,...$
	Hence  the last  inequality implies
	\begin{equation*}
		\left\Vert u_{n}-u_{n-1}\right\Vert_{C_\gamma(T,\mathbb{X})}\leq \frac{\Gamma(1-\gamma)M_1^nT^{n\alpha}}{
			\Gamma(n\alpha+1-\gamma)} ~\Vert g_{\alpha,z}^*\Vert_{C_\gamma(T,\mathbb{X})}.
	\end{equation*}
	Since $$\sum_{k=0}^\infty \frac{\Vert g_{\alpha,z}^*\Vert_{C_\gamma(T,\mathbb{X})}(M_1T^\alpha)^{k}}{\Gamma(k\alpha+1-\gamma)}=\Vert g_{\alpha,z}^*\Vert_{C_\gamma(T,\mathbb{X})}E_{\alpha,1-\gamma}(M_1T^\alpha)<\infty, $$ the Weierstrass theorem implies that $u_n=\sum_{k=1}^{n} (u_{k}-u_{k-1})$ converges in $C_\gamma(T,\mathbb{X})$ to a function $u_{\alpha,z}$. We can verify directly that $u_{\alpha,z}$ is the solution of the Abel equation and that
	\begin{equation*}
		\Vert u_{\alpha,\beta}\Vert_{C_\gamma}\leq \Gamma(1-\gamma)E_{\alpha,1-\gamma}(M_1T^\alpha) \Vert g_{\alpha,\beta}^*\Vert_{C_\gamma(T,\mathbb{X})}.
	\end{equation*}
	
	{\bf Proof of  (b):} We shall use the part (a) and an approximation argument. We choose a sequence $g_n\in C_c^1([0,T];\mathbb{X})$ such that $g_n\to g^*_{\alpha,z}$
	in $L^p(0,T;\mathbb{X})$. From the part (a), there exist a unique solution $u_n\in C_0(T;\mathbb{X})$ such that
	$$  u_n=g_n+A^*_{\alpha,z}u_n.$$
	Moreover, we can find a constant $\epsilon_n>0$ such that $u_n(t)=0$ for $t\in [0,\epsilon_n]$. By direct estimating, we obtain
	\begin{equation*}
		\Vert u_n(t)-u_m(t)\Vert\leq \Vert g_n(t)-g_m(t)\Vert+\kappa\int_0^t\frac{\Vert u_n(s)-u_m(s)\Vert}{(t-s)^{1-\alpha}}ds.
	\end{equation*}
	Using Lemma \ref{Mittag-Leffler}, Part (h), we obtain
	\begin{align}
		\Vert u_n(t)-u_m(t)\Vert&\leq \Vert g_n(t)-g_m(t)\Vert\nn\\
		&+\kappa\Gamma(\alpha)\int_0^t E_{\alpha,\alpha}(\Gamma(\alpha)\kappa(t-s)^{\alpha})\Vert g_n(s)-g_m(s)\Vert ds\nn\\
		& \leq \Vert g_n(t)-g_m(t)\Vert\nn\\
		&+\kappa\Gamma(\alpha)E_{\alpha,\alpha}(\Gamma(\alpha)\kappa T^{\alpha})\int_0^t \Vert g_n(s)-g_m(s)\Vert ds.
	\end{align}
	Calculating directly gives
	\begin{eqnarray*} \Vert u_n(t)-u_m(t)\Vert&\leq & \Vert g_n(t)-g_m(t)\Vert+ \\
		& &\kappa\Gamma(\alpha)E_{\alpha,\alpha}(\Gamma(\alpha)\kappa T^{\alpha})T^{1/q}
		\Vert g_n(t)-g_m(t)\Vert_{L^p(0,T)}.
	\end{eqnarray*}
	Hence,
	\begin{equation*}
		\Vert u_n-u_m\Vert_{L^p(0,T;\mathbb{X})}\leq \Big(1+\kappa T\Gamma(\alpha)E_{\alpha,\alpha}(\Gamma(\alpha)\kappa T^{\alpha})\Big)
		\Vert g_n-g_m\Vert_{L^p(0,T;\mathbb{X})}.
	\end{equation*}
	This implies $\{u_n\}$ is Cauchy in $L^p(0,T;\mathbb{X})$. So it has a limit
	$u_{\alpha,z}\in L^p(0,T;\mathbb{X})$. From Lemma \ref{singular-integral}, Part (d), we have
	$A_{\alpha,z} u_n\to A_{\alpha,z}u_{\alpha,z}$ in $L^p(0,T;\mathbb{X})$. Since $u_n=g^*_{\alpha,z}+A^*_{\alpha,z}u_n$, we
	get
	$$u_{\alpha,z}=g^*_{\alpha,z}+A^*_{\alpha,z}u_{\alpha,z}. $$
	Using the same estimate as in the proof of existence, we obtain
	\begin{equation*}
		\Vert u_{\alpha,z}\Vert_{L^p(0,T;\mathbb{X})}\leq \Big(1+\kappa T\Gamma(\alpha) E_{\alpha,\alpha}(\kappa\Gamma(\alpha) T^\alpha) \Big)
		\Vert g^*_{\alpha,z}\Vert_{L^p(0,T;\mathbb{X})}.
	\end{equation*}

\end{proof}

Before stating and proving the main result, we consider a  compactness result in $C_\gamma(T,\mathbb{X})$.
\begin{lemma}
	\label{C-compact}
	Let $K\in C(\Delta_T\times\mathbb{X};\mathbb{X})$, $K=K(t,s,\alpha,z,v)$. We assume that
	
	{\bf (i)} $K$ is Lipschitz with respect to
	the variable $v\in \mathbb{X}$ as in Theorem \ref{existence-Abel-equation},
	
	{\bf (ii)} $\lim_{\xi\to 0^+}\omega_{K}(\xi,M)=0$ where
	\begin{equation*}
		\omega_K(\xi,M):=\sup_B\Vert K(t_2,s,\alpha,z,v)-K(t_1,s,\alpha,z,v)\Vert
	\end{equation*}
	and
	\begin{equation*}
		B=\Big\{(t_1,t_2,s,\alpha,z,w):\ |t_2-t_1|\leq\xi, 0<s\leq \min\{t_1,t_2\},\alpha\in [\alpha_0,\alpha_1],
		z\in P, w\in \mathbb{X}, \Vert w\Vert\leq M \Big\}.
	\end{equation*}

	Then  the following results hold
	
	\begin{enumerate}[\bf {\upshape(a)}]
		\item
		the set
		\begin{equation*}
			B(\alpha_0,\alpha_1,P,L)=\Big\{A_{\alpha,z} v:\ \alpha\in [\alpha_0,\alpha_1],z\in P,
			v\in C_\gamma(T,\mathbb{X}), \Vert v\Vert_{C_\gamma(T,\mathbb{X})}\leq L \Big\}
		\end{equation*}
		is precompact in $C_\gamma(T,\mathbb{X})$.

		\smallskip
		
		\item
		let $a_n,\alpha\in (0,1), z_n,z\in P$, $\lim_{n\to\infty}a_n=\alpha$, $\lim_{n\to\infty}z_n=z$. Assume that $w_n,w\in C_\gamma(T,\mathbb{X})$ and $\lim_{n\to\infty}w_n= w$ in  $C_\gamma(T,\mathbb{X})$.
		Then $A_{a_n,z_n}w_n\to A_{\alpha,z} w$ in $C_\gamma(T,\mathbb{X})$ as $n\to\infty$.
	\end{enumerate}

\end{lemma}
~\\
\begin{proof}
	{\bf Proof of  (a):} To this end, we shall prove that $B(\alpha_0,\alpha_1,P, L)$ is  equibounded and equicontinuous on $[\delta,T]$ for every $\delta\in (0,T]$.
	For $w\in C_\gamma(T,\mathbb{X}), \Vert w\Vert_{C_\gamma(T,\mathbb{X})}\leq L$, Lemma \ref{singular-integral} part (c) implies
	\begin{align}
		t^\gamma\Vert A_{\alpha,z} w(t)\Vert & \leq \frac{M_0}{\alpha_0}t^{\alpha+\gamma}+\kappa Lt^{\alpha}B(\alpha,1-\gamma)\nn\\
		&\leq  \max\{T^{\alpha_0}, T^{\alpha_1}\}\left(\frac{M_0}{\alpha_0}T^{\gamma}+\kappa L
		\max_{\alpha_0\leq\alpha\leq\alpha_1}B(\alpha,1-\gamma)\right).
		\label{A1-upper-bound}
	\end{align}
	Hence $B(\alpha_0,\alpha_1,P,L)$ is equibounded in $C_\gamma(T,\mathbb{X})$. This implies that
	$B(\alpha_0,\alpha_1,P,L)$ is equibounded on $C[\delta,T]$. Now, we verify that $B(\alpha_0,\alpha_1,P,L)$ is equicontinuous on $[\delta,T]$. For $\xi>0$, $\delta\leq t_1\leq t_2\leq T$, $|t_2-t_1|\leq\xi$ we have
	\begin{equation*}
		A_\alpha v(t_2)-A_\alpha v(t_1) =   J_1+J_2+J_3
	\end{equation*}
	with
	\begin{align*}
		J_1 &=  \int_0^{t_1}( (t_2-s)^{\alpha-1}-(t_1-s)^{\alpha-1}) K(t_2,s,\alpha,z,v(s))ds,\\
		J_2 &=  \int_0^{t_1} (t_1-s)^{\alpha-1} (K(t_2,s,\alpha,z, v(s))-K(t_1,s,\alpha,z,v(s))ds,\\
		J_3 &= \int_{t_1}^{t_2}(t_2-s)^{\alpha-1}K(t_2,s,\alpha,z,v(s))ds.
	\end{align*}
	Using \eqref{K-inequality} we have
	\begin{align*}
		|J_1|&\leq \int_0^{t_1} \left| (t_2-s)^{\alpha-1}-(t_1-s)^{\alpha-1}\right|(M_0+ \kappa Ls^{-\gamma})ds\\
		&\leq  \frac{M_0}{\alpha}(t_2^\alpha-t_1^\alpha)+\kappa L \int_0^{t_1}(t_1-s)^{\alpha-1} s^{-\gamma}ds-\kappa L\int_0^{t_1}(t_2-s)^{\alpha-1} s^{-\gamma}ds.
	\end{align*}
	We have
	\begin{align*}
		|t_2^\alpha-t_1^\alpha| &\leq |t_2-t_1|^\alpha, \ \ \ \ \ \ \ \ \ \ \ 0<\alpha\leq 1, \\
		&\leq \alpha T^{\alpha-1}|t_2-t_1|,  \ \ \ \ \ \alpha>1, t_1,t_2\in [0,T].
	\end{align*}
	Hence, putting $M'_0=M_0\max\{1,\alpha T^{\alpha-1}\}$, $a=\min\{\alpha,1\}$ we obtain
	\begin{align*}
		|J_1|&\leq \frac{M'_0}{\alpha_0}(t_2-t_1)^a+\kappa L (t_1^{\alpha-\gamma}-t_2^{\alpha-\gamma})B(\alpha,1-\gamma)+\kappa L\int_{t_1}^{t_2}(t_2-s)^{\alpha-1} s^{-\gamma}ds\\
		&\leq \frac{M'_0}{\alpha_0}(t_2-t_1)^a+|\alpha-\gamma|\kappa L|t_2-t_1|\delta^{\alpha-\gamma-1} +\frac{\kappa L\delta^{-\gamma}}{\alpha}(t_2-t_1)^\alpha \\
		&\leq \frac{M'_0}{\alpha_0}\xi^a+|\alpha-\gamma|\kappa L\xi\delta^{\alpha-\gamma-1} +\frac{\kappa L\delta^{-\gamma}}{\alpha}\xi^\alpha.
	\end{align*}
	Similarly,  we obtain
	\begin{equation*}
		|J_3|\leq \frac{M_0}{\alpha_0}(t_2-t_1)^\alpha+\kappa L\int_{t_1}^{t_2}(t_2-s)^{\alpha-1}s^{-\gamma}ds\leq
		\left(\frac{M_0+\kappa L\delta^{-\gamma}}{\alpha_0}\right)\xi^\alpha.
	\end{equation*}
	To estimate $J_2$,  we note that $\Vert v(s)\Vert\leq \delta^{-\gamma}\Vert v\Vert_{C_\gamma(T,\mathbb{X})}\leq \delta^{-\gamma}L$
	for $s\geq\delta$, so we have
	\begin{equation}
		|J_2|\leq \frac{T^\alpha}{\alpha\delta}\omega_K(\xi,\delta^{-\gamma}L). \end{equation}
	
	From these estimates, we obtain that $B(\alpha_0,\alpha_1,P,L)$ is equicontinuous in $C([\delta,T];\mathbb{X})$ for every $\delta>0$.
	Using the Arzela-Ascoli theorem, we deduce that $B(\alpha_0,\alpha_1,P,L)$ is precompact in $C([\delta,T];\mathbb{X})$ for every $\delta>0$.
	
	Now, we prove $B(\alpha_0,\alpha_1,P,L)$ is compact in $C_\gamma(T,\mathbb{X})$ by the diagonal argument. Since $B(\alpha_0,\alpha_1,P, L)$ is compact in
	$C\left([\frac{T}{2}, T];\mathbb{X}\right)$, we can find a sequence $A_{\alpha_{1,n}}w_{1,n}$ and a function $v_1\in C\left([\frac{T}{2}, T];\mathbb{X}\right)$ such that
	$$A_{\alpha_{1,n},z_{1,n}}w_{1,n}\rightarrow v_1\ {\rm in}\ C\left(\left[\frac{T}{2}, T\right];\mathbb{X}\right).$$
	In the sequence $A_{\alpha_{1,n}}w_{1,n}$ we can find a subsequence $A_{\alpha_{2,n}}w_{2,n}$ and
	a function $v_2\in C\left(\left[\frac{T}{2^2}, T\right];\mathbb{X}\right)$ such that
	$$A_{\alpha_{2,n},z_{2,n}}w_{2,n}\rightarrow v_2\ {\rm in}\ C\left(\left[\frac{T}{2^2}, T\right];\mathbb{X}\right).$$
	We note that $$A_{\alpha_{2,n},\beta_{2,n}}w_{2,n}\rightarrow v_1\ {\rm in}\ C\left(\left[\frac{T}{2}, T\right];\mathbb{X}\right).$$
	So we have $\left. v_2\right|_{[T/2,T]}=v_1$. By induction we can construct sequences $A_{\alpha_{k,n},z_{k,n}}w_{k,n}$
	and the function $v_k$ defined on $\left[\frac{T}{2^k},T\right]$ such that
	\begin{enumerate}[\bf \upshape(i)]
		\item $A_{\alpha_{k,n},z_{k,n}}w_{k,n}$ is a subsequence of
		$A_{\alpha_{k-1,n},z_{k-1,n}}w_{k-1,n}$.
		\item $A_{\alpha_{k,n},z_{k,n}}w_{k,n}\rightarrow v_k$ in $C\left(
		\left[\frac{T}{2^k}, T\right];\mathbb{X}
		\right)$.
	\end{enumerate}
	From the very last properties, we deduce that $v_k\left|_{\left[\frac{T}{2^{k-1}},T\right]}\right.=v_{k-1}$.
	Therefore, we can define  function  $v(t)=v_k(t)$  for every $\left[\frac{T}{2^k},T\right]$ in a unique way. We prove that
	$v\in C_\gamma(T,\mathbb{X})$ and  $A_{\alpha_{k,k},z_{k,k}}w_{k,k}\to v$ in $C_\gamma(T,\mathbb{X})$. Since
	$v_k\in C\left(\left[\frac{T}{2^k},T\right];\mathbb{X}\right)$,
	$k=1,2,...$, we obtain $v\in C((0,T];\mathbb{X})$. On the other hand,
	the inequality (\ref{A1-upper-bound}) gives
	\begin{equation*}
		t^\gamma \Vert A_{\alpha_{k,k},z_{k,k}}w_{k,k}(t)\Vert\leq
		\max\{t^{\alpha_0}, t^{\alpha_1}\}\left(\frac{M_0}{\alpha_0}T^{\gamma}+\kappa L
		\max_{\alpha_0\leq\alpha\leq\alpha_1}B(\alpha,1-\gamma)\right)
		.
	\end{equation*}
	Letting $k\to\infty$ in the last inequality we obtain
	\begin{equation}
		t^\gamma \Vert v(t)\Vert\leq
		\max\{t^{\alpha_0}, t^{\alpha_1}\}\left(\frac{M_0}{\alpha_0}T^{\gamma}+\kappa L
		\max_{\alpha_0\leq\alpha\leq\alpha_1}B(\alpha,1-\gamma)\right).
		\label{v-upper-bound}
	\end{equation}
	Hence $v\in C_\gamma(T,\mathbb{X})$. We now have to verify that $A_{\alpha_{k,k},z_{k,k}}w_{k,k}\to v$ in $C_\gamma(T,\mathbb{X})$.
	For $\delta>0$, we have
	\begin{eqnarray*}
		\Vert A_{\alpha_{k,k},z_{k,k}}w_{k,k}- v\Vert_{C_\gamma(T,\mathbb{X})}\leq L_1+L_2
	\end{eqnarray*}
	where
	\begin{eqnarray*}
		L_1 &=& \max_{0<t\leq \delta}t^\gamma |A_{\alpha_{k,k},z_{k,k}}w_{k,k}(t)- v(t)|,\\
		L_2 &=& \max_{\delta<t\leq T}t^\gamma |A_{\alpha_{k,k},z_{k,k}}w_{k,k}(t)- v(t)|.
	\end{eqnarray*}
	For $0<\delta<\min\{1,T\}$, in view of (\ref{A1-upper-bound}) and (\ref{v-upper-bound}) we have
	$$  |L_1|\leq
	\delta^{\alpha_0}\left(\frac{M_0}{\alpha_0}T^{\gamma}+\kappa L
	\max_{\alpha_0\leq\alpha\leq\alpha_1}B(\alpha,1-\gamma)\right)
	.$$
	Since $A_{\alpha_{k,k},z_{k,k}}w_{k,k}\to v$ in $C\left([\delta,T];\mathbb{X}\right)$ as $k\to\infty$, we get
	$$\limsup_{k\to\infty}\Vert A_{\alpha_{k,k},z_{k,k}}w_{k,k}- v\Vert_{C_\gamma(T,\mathbb{X})}\leq
	\delta^{\alpha_0}\left(\frac{M_0}{\alpha_0}T^{\gamma}+\kappa L
	\max_{\alpha_0\leq\alpha\leq\alpha_1}B(\alpha,1-\gamma)\right) $$
	for every $\delta\in (0,\min\{1,T\})$. Let $\delta\to 0^+$, we obtain $\limsup_{k\to\infty}\Vert A_{\alpha_{k,k},z_{k,k}}w_{k,k}- v\Vert_{C_\gamma(T,\mathbb{X})}= 0$. This completes the proof of part (a).
	
	{\bf Proof of  (b):} Since $w_n\to w$ in  $C_\gamma(T,\mathbb{X})$, there exists a constant $L$ such that  $\Vert w\Vert_{C_\gamma(T,\mathbb{X})}$,
	$\Vert w_n\Vert_{C_\gamma(T,\mathbb{X})}\leq L $ for every $n=1,2,..$. Choose $\alpha_0,\alpha_1\in (0,1)$ such that $\alpha_0\leq a_n\leq\alpha_1$ for every
	$n=1,2,...$ Hence  we have $A_{a_n,z_n}w_n\in B(\alpha_0,\alpha_1,P,L)$. Assume that
	$A_{a_n,z_n}w_n\not\to A_{\alpha,\beta}w$ in $C_\gamma(T,\mathbb{X})$ as $n\to\infty$.
	From Lemma \ref{C-compact}, we can find  an $\epsilon_0>0$ and
	a subsequence of ($A_{a_n,z_n}w_n$), still denote by the same sequence, which converges to $z\in C_\gamma(T,\mathbb{X})$ and
	$   \Vert A_{a_{n},z_{n}}w_{n}-A_{\alpha,\beta}w\Vert_{C_\gamma(T,\mathbb{X})}\geq\epsilon_0>0.$ Let $n\to\infty$ we obtain
	$\Vert z-A_{\alpha,z}w\Vert_{C_\gamma(T,\mathbb{X})}\geq\epsilon_0>0.$
	We claim that $z=A_{\alpha,z} w$ which is a contradiction. We have
	
	\begin{align*}
		A_{a_{n},z_{n}}w_{n} (t)&= \int_{0}^t \frac{K(t,s,a_n,z_n,w_n(s))}{(t-s)^{1-a_n}}\\
		&= t^{a_n}\int_0^1\frac{K(t,\theta t, a_n,z_n, w_n(\theta t))}{(1-\theta)^{1-a_n}}d\theta.
	\end{align*}
	Fixing $t>0$, we put
	\begin{equation*}
		F_n(\theta)=\frac{K(t,\theta t, a_n,z_n, w_n(\theta t))}{(1-\theta)^{1-a_n}},\
		F(\theta)=\frac{K(t,\theta t, \alpha,z, w(\theta t))}{(1-\theta)^{1-\alpha}}.
	\end{equation*}
	We have
	\begin{align*}
		\Vert F_n(\theta)-F(\theta)\Vert &\leq \frac{M_0+\kappa \Vert w_n(\theta t)\Vert}{(1-\theta)^{1-a_n}}+
		\frac{M_0+\kappa \Vert w(\theta t)\Vert}{(1-\theta)^{1-\alpha}}\\
		&\leq  \frac{2M_0+2\kappa(\theta t)^{-\gamma} L}{(1-\theta)^{1-a_0}}:= g(\theta).
	\end{align*}
	Since
	$g$ is in $L^1(0,T)$ and that $\lim_{n\to\infty}\Vert F_n(\theta)-F(\theta)\Vert=0 $, we can apply
	the dominated convergence theorem of Lebesgue to obtain $\lim_{n\to\infty}\Vert\int_0^1( F_n(\theta)-F(\theta)) d\theta\Vert=0$.
	It follows that
	$z(t)=\lim_{n\to\infty}A_{a_n,z_n} w_n(t)=A_{\alpha,z} w(t)$.
\end{proof}

Now, we state and prove the continuity of the solutions of  {\bf the general nonlinear Abel integral of the second kind} with respect to the fractional parameter $\alpha$.

\begin{theorem}\label{main-theorem-Lipschitz-Abel-equation}
	Suppose that the assumptions of Theorem \ref{existence-Abel-equation} and Lemma \ref{C-compact} hold.
	Let $0<\alpha_0<\alpha_1$, $\gamma\in [0,1)$, $\nu,\mu_1,\ldots, \mu_k\in (0,1]$,
	$\kappa_0,\kappa>0$, $(\alpha,z), (\alpha',z')\in [\alpha_0,\alpha_1]\times P$.
	\begin{enumerate}[\bf \upshape (a)]
		\item{If  $g_{\alpha,z}\in  C_\gamma(T,\mathbb{X})$ for all $(\alpha,z)\in [\alpha_0,\alpha_1]\times P$ and
			\begin{eqnarray*}
				\lim_{(\alpha',z')\to(\alpha,z)}\Vert g_{\alpha',z'}-g_{\alpha,z}\Vert_{C_\gamma(T,\mathbb{X})}&=&0.
			\end{eqnarray*}
			then the nonlinear equation (\ref{generalized-Abel-equation}) has a unique solution
			and we have
			\begin{equation*}
			\lim_{(\alpha',z')\to(\alpha,z)}\Vert u_{\alpha',z'}-u_{\alpha,z}\Vert_{C_\gamma(T,\mathbb{X})}=0.
			\end{equation*}}
		\item{If $g_{\alpha,z}\in L^p(0,T;\mathbb{X})$   for all $(\alpha,z)\in [\alpha_0,\alpha_1]\times P$ and
			\begin{eqnarray*}
				\lim_{(\alpha',z')\to(\alpha,z)}\Vert g_{\alpha',z'}-g_{\alpha,z}\Vert_{L^p(0,T;\mathbb{X})}&=&0.
		\end{eqnarray*}	
		Then
		\begin{equation*}
			\lim_{(\alpha',z')\to(\alpha,z)}\Vert u_{\alpha',z'}-u_{\alpha,z}\Vert_{L^p(0,T;\mathbb{X})}=0.
		\end{equation*}}
		\item{Letting $z=(\beta_1,\ldots,\beta_k)$, $z'=(\beta'_1,\ldots,\beta'_k)\in P$, $\mu=(\mu_1,\ldots,\mu_k)$,
			$\mu_j\in (0,1], j=1,\ldots,k$, we assume that
			\begin{eqnarray*}
				\Vert g_{\alpha',z'}-g_{\alpha,z}\Vert_{C_\gamma(T,\mathbb{X})}&\leq& \kappa_0(|\alpha'-\alpha|^\nu+|z'-z|^\mu),\\
				\Vert K(t,s,\alpha',z',w)-K(t,s,\alpha,z,w)\Vert &\leq &\kappa(|\alpha'-\alpha|^\nu+|z'-z|^\mu)(\Vert w\Vert+1)
			\end{eqnarray*}
			where  $0\leq s\leq t\leq T$, $|z'-z|^\mu:=\sum_{j=1}^k|\beta'_j-\beta_j|^{\mu_j}$. Then there is a $C=C(\alpha_0,\alpha_1,P)$ such that
			\begin{equation*}
			\Vert u_{\alpha',z'}-u_{\alpha,z}\Vert_{C_\gamma(T,\mathbb{X})}\leq C_0(|\alpha'-\alpha|^\nu+|z'-z|^\mu)
			\end{equation*}
			where
			$$ C_0= C(\kappa_0+\kappa+1)^2 E^2_{\alpha_0,1-\gamma}(\kappa\Gamma(\alpha_1)\max\{T^{\alpha_0}, T^{\alpha_1}\})(1+\Vert g_{\alpha,z}\Vert_{C_\gamma(T,\mathbb{X})}).$$
			{ In addition, let $\lambda\geq \nu,\rho_j\geq\mu_j,$ $j\in\overline{1,k}$ and let $(a_n,z_n)$ }
			be random variables satisfying $(a_n,z_n)\in [\alpha_0,\alpha_1]\times P$,
			$z_n=(z_{1n},\ldots, z_{kn})$.
			Then
			$$ \mathbb{E}\Vert u_{a_n,z_n}-u_{\alpha,z}\Vert_{C_\gamma(T,\mathbb{X})}\leq
			C_0\left((\mathbb{E}|a_n-\alpha|^\lambda)^{\nu/\lambda}+
			\sum_{j=1}^k(\mathbb{E}|z_{jn}-\beta_j|^{\rho_j})^{\mu_j/\rho_j}\right).   $$}
		\item{If $K$ is as in part (c) and
			\begin{align*}
			\Vert g_{\alpha',z'}-g_{\alpha,z}\Vert_{L^p(0,T;\mathbb{X})}&\leq \kappa_0(|\alpha'-\alpha|^\nu+|z'-z|^\mu),
			\end{align*}
			then we have
			\begin{equation*}
			\Vert u_{\alpha',z'}-u_{\alpha,z}\Vert_{L^p(0,T;\mathbb{X})}\leq
			C_0(|\alpha'-\alpha|^\nu+|z'-z|^\mu)
			\end{equation*}
			for a constant
			$$  C_0= C(\kappa_0+\kappa+1)^2 E^2_{\alpha_0,\alpha_0}(\kappa\Gamma(\alpha_1)\max\{T^{\alpha_0}, T^{\alpha_1}\})(\Vert g_{\alpha,z}\Vert_{L^p(0,T;\mathbb{X})}+1). $$
			In addition, let $(a_n,z_n)$ be random variables as in (c).
			Then
			$$ \mathbb{E}\Vert u_{a_n,z_n}-u_{\alpha,z}\Vert_{L^p(0,T:\mathbb{X})}\leq
			C_0\left((\mathbb{E}|a_n-\alpha|^\lambda)^{\nu/\lambda}+
			\sum_{j=1}^k(\mathbb{E}|z_{jn}-\beta_j|^{\rho_j})^{\mu_j/\rho_j}\right).   $$
			\item Assume that $K(t,s,\alpha,z,w)$ has the derivative $\frac{\partial K}{\partial \alpha}$
			and the Frechet derivative $DK$ with respect to the variable $w$. Moreover, we assume that $\frac{\partial g}{\partial \alpha}(.,\alpha)
			\in C_\gamma(T,\mathbb{X})$ with $0<\gamma\leq 1$, $\frac{\partial K}{\partial\alpha}\in C(\Delta_T\times \mathbb{X};\mathbb{X})$ and there exists $\omega:[0,1]\to \mathbb{R}$ such that $\lim_{\delta\to 0}\omega(\delta)=0$ and
			$$  \omega_{DK^*}(w, \delta)\leq \omega(\delta)\Vert w\Vert  $$
			where
			$$ \omega_{DK^*}(w, \delta)=\sup_{B_1}\Vert DK^*(t_2,s,\alpha,z,w)-DK^*(t_1,s,\alpha,z,w)\Vert$$
			and
			$$B_1=\{(t_2,t_1,s,\alpha,z):\ 0\leq s\leq t_1\leq t_2, |t_2-t_1|\leq \delta,\alpha\in [\alpha_0,\alpha_1],z\in P\}.$$ Then
			$u_{\alpha,z}$ is differentiable with respect to $\alpha$
			and
			\begin{equation*}
			\left\Vert\frac{\partial u_{\alpha,z}}{\partial \alpha}\right\Vert_{C_\gamma(T,\mathbb{X})}
			\leq  \Gamma(1-\gamma)E_{\alpha,1-\gamma}(MT^\alpha)  \Vert g_{1\alpha,z}\Vert_{C_\gamma(T,\mathbb{X})},
			\end{equation*}
			where
			\begin{align*}
			g_{1\alpha,z}(t)&=\frac{\partial g^*_{\alpha,z}}{\partial \alpha}+
			\int_0^t \frac{\partial K^*}{\partial\alpha}(t,s,\alpha,z,u_{\alpha,z}(s))
			\frac{ds}{(t-s)^{1-\alpha}} \nn\\
			&+\int_0^t \frac{K^*(t,s,\alpha,u_{\alpha,z}(s))\ln(t-s)}{(t-s)^{1-\alpha}}ds.
			\end{align*}}
	\end{enumerate}
\end{theorem}
\begin{proof}
	
	\noindent {\bf Proof of  (a):}
	Assume that $\Vert u_{\alpha',z'}-u_{\alpha,z}\Vert_{C_\gamma(T,\mathbb{X})}\not\to 0$ as $\alpha'\to\alpha$.
	We can choose an $\epsilon_0>0$ and a sequence $u_{a_n,z_n}$, such that
	\begin{equation*}
		\lim_{n\to\infty}(a_n,z_n)=(\alpha,z)\ {\rm and}\  \Vert u_{a_n,z_n}-u_{\alpha,z}\Vert_{C_\gamma(T,\mathbb{X})}\geq \epsilon_0>0.
	\end{equation*}
	From the first part of Theorem \ref{existence-Abel-equation}, we can find a constant  $L$ such that
	$\Vert u_{a_n,z_n}\Vert_{C_\gamma(T,\mathbb{X})}\leq L$ for every $n\to\infty$. It follows that
	$A^*_{a_n,z_n}u_{a_n,z_n}\in B(\alpha_0,\alpha_1,P,L)$. So by Lemma \ref{C-compact} we can find a subsequence, still denoted by $A^*_{a_n,z_n}u_{a_n,z_n}$,
	and an element $x\in C_\gamma(T,\mathbb{X})$ such that $\lim_{n\to\infty}A^*_{a_n,z_n}u_{a_n,z_n}=x\in C_\gamma(T,\mathbb{X})$.
	It follows that
	\begin{equation*}
		u_{a_n,z_n}=g_{a_n,z_n}^*+A^*_{a_n}u_{a_n}\to g_{\alpha,z}^*+x:=u.
	\end{equation*}
	
	From Lemma \ref{C-compact}, we obtain
	$A^*_{a_n,z_n}u_{a_n,z_n}\to A^*_{\alpha,z} u$ as $n\to\infty$. Hence $u=g^*_{\alpha,z}+A^*_{\alpha,z} u$. But
	$u_{\alpha,z}=g_{\alpha,z}^*+A^*_{\alpha,z} u_{\alpha,z}$, hence, by the uniqueness we have $u=u_{\alpha,z}$ and $u_{a_n,z_n}\to u_{\alpha,z}$ in $C_\gamma(T,\mathbb{X})$. This contradicts with the assumption
	$\Vert u_{a_n,z_n}- u_{\alpha,z}\Vert_{C_\gamma(T,\mathbb{X})}\geq \epsilon_0>0$.\\
	
\noindent {\bf Proof of  (b):} Choose $\varphi\in C_c^\infty(-1,1)$, $|\varphi(t)|\leq 1$ for $t\in [-1,1]$. Let $\delta>0$ and $\varphi_\delta(t)=\frac{1}{C_\varphi\delta}\varphi(\frac{t}{\delta})$ with $C_\varphi=\int_{-1}^1\varphi(s)ds$. We approximate $g^*_{\alpha,z}$ by $G_{\delta,\alpha,z}:=\varphi_\delta*g^*_{\alpha,z}\in C([0,T];\mathbb{X})$.
	Let
	$v_{\delta,\alpha,z}$ be the solution of
	$$v_{\delta,\alpha,z}=G_{\delta,\alpha,z}+A^*_{\alpha,z}v_{\delta,\alpha,z}.$$
	We can use the triangle inequality to get
	\begin{align*}
		\Vert u_{\alpha',z'}-u_{\alpha,z}\Vert_{L^p(0,T;\mathbb{X})}
		&\leq \Vert u_{\alpha',z'}-v_{\delta,\alpha',z'}
		\Vert_{L^p(0,T;\mathbb{X})}+\Vert v_{\delta,\alpha',z'}-v_{\delta,\alpha,z}\Vert_{L^p(0,T;\mathbb{X})} \nn\\
		&+\Vert v_{\delta,\alpha,z}-u_{\alpha,z}\Vert_{L^p(0,T;\mathbb{X})}.
	\end{align*}
	Calculating directly, we obtain
	\begin{align*}
		\Vert G_{\delta,\alpha,z}-g_{\alpha,z}\Vert_{L^p(0,T;\mathbb{X})}
		&\leq C\sup_{|\xi|\leq\delta}\Vert g_{\alpha,z}(.+\xi)-g_{\alpha,z}\Vert_{L^p(0,T;\mathbb{X})},\\
		\Vert G_{\delta,\alpha',z'}-G_{\delta,\alpha,z}\Vert_{L^p(0,T;\mathbb{X})} &\leq
		\Vert g_{\alpha',z'}-g_{\alpha,z}\Vert_{L^p(0,T;\mathbb{X})}.
	\end{align*}
	From the last two inequalities, Lemma \ref{singular-integral} (e) gives
	\begin{eqnarray*}
		\Vert v_{\delta,\alpha,z}-u_{\alpha,z}\Vert_{L^p(0,T;\mathbb{X})}&\leq&C(\alpha_0,\alpha_1)C\sup_{|\xi|\leq\delta}\Vert g_{\alpha,z}(.+\xi)-g_{\alpha,z}\Vert_{L^p(0,T;\mathbb{X})},\\
		\Vert u_{\alpha',z'}-v_{\delta,\alpha',z'}\Vert_{L^p(0,T;\mathbb{X})}&\leq&C(\alpha_0,\alpha_1)C\sup_{|\xi|\leq\delta}\Vert g_{\alpha',z'}(.+\xi)-g_{\alpha',z'}\Vert_{L^p(0,T;\mathbb{X})},
	\end{eqnarray*}
	where
	$$C(\alpha_0,\alpha_1)=1+\kappa\Gamma(\alpha_1)TE_{\alpha_0,\alpha_0}(\kappa
	\Gamma(\alpha_1)\max\{T^{\alpha_0},T^{\alpha_1}\}).$$
	Estimating directly gives
	\begin{align*}
		\Vert g_{\alpha',z'}(.+\xi)-g_{\alpha',z'}\Vert_{L^p(0,T;\mathbb{X})}&\leq\Vert g_{\alpha',z'}-g_{\alpha,z}\Vert_{L^p(0,T;\mathbb{X})}+
		\Vert g_{\alpha',z'}(.+\xi)-g_{\alpha,z}(.+\xi)\Vert_{L^p(0,T;\mathbb{X})}\\
		&+\Vert g_{\alpha,z}(.+\xi)-g_{\alpha,z}\Vert_{L^p(0,T;\mathbb{X})}\\
		&\leq 2\Vert g_{\alpha',z'}-g_{\alpha,z}\Vert_{L^p(0,T;\mathbb{X})}+
		\Vert g_{\alpha,z}(.+\xi)-g_{\alpha,z}\Vert_{L^p(0,T;\mathbb{X})}.
	\end{align*}
	Hence
	\begin{eqnarray*}
		\Vert u_{\alpha',z'}-u_{\alpha,z}\Vert_{L^p(0,T;\mathbb{X})}
		&\leq& C(\alpha_0,\alpha_1)C\sup_{|\xi|\leq\delta}\Vert g_{\alpha,z}(.+\xi)-g_{\alpha,z}\Vert_{L^p(0,T;\mathbb{X})}+\\
		& & 2C(\alpha_0,\alpha_1)\Vert g_{\alpha',z'}-g_{\alpha,z}\Vert_{L^p(0,T;\mathbb{X})}+\Vert v_{\delta,\alpha',z'}-v_{\delta,\alpha,z}\Vert_{L^p(0,T;\mathbb{X})}.
	\end{eqnarray*}
	Now, choosing a sequence $(\alpha'_n,z'_n)\to (\alpha,z)$ as $n\to\infty$ and using Part (a), we can obtain
	$$  \limsup_{n\to\infty}\Vert u_{\alpha'_n,z'_n}-u_{\alpha,z}\Vert_{L^p(0,T;\mathbb{X})}
	\leq  C(\alpha_0,\alpha_1)C\sup_{|\xi|\leq\delta}\Vert g_{\alpha,z}(.+\xi)-g_{\alpha,z}\Vert_{L^p(0,T;\mathbb{X})}.$$
	Letting $\delta\to 0^+$ gives
	$$ \limsup_{n\to\infty}\Vert u_{\alpha'_n,z'_n}-u_{\alpha,z}\Vert_{L^p(0,T;\mathbb{X})}=0. $$
	This completes the proof of Part (b).
	
\noindent {\bf Proof of  (c):}	 We have
	\begin{align}
		u_{\alpha',z'}(t)-u_{\alpha,z}(t)&= g_{\alpha',z'}^*-g_{\alpha,z}^*+
		I_1(\alpha',\alpha,z',z)+I_2(\alpha',\alpha,z',z)\nonumber\\ &+\int_0^t\frac{K^*(t,s,\alpha,z,u_{\alpha',z'}(s))-K^*(t,s,\alpha,z,u_{\alpha,z}(s))}{(t-s)^{1-\alpha}}ds
		\label{Holder-integral-equation}
	\end{align}
	where
	\begin{align*}
		I_1(\alpha',\alpha,z',z)
		& =\int_0^t K^*(t,s,\alpha',z',u_{\alpha',z'}(s))((t-s)^{\alpha'-1}-(t-s)^{\alpha-1})ds,\\
		I_2(\alpha',\alpha,z',z)& =
		\int_0^t\frac{K^*(t,s,\alpha',z',u_{\alpha',z'}(s))-K^*(t,s,\alpha,z,u_{\alpha',z'}(s))}{(t-s)^{1-\alpha}}ds.
	\end{align*}
	Now, put $w=u_{\alpha',z'}-u_{\alpha,z}$, $G\equiv G(\alpha',\alpha,z',z)=g^*_{\alpha',z'}-g^*_{\alpha,z}+
	I_1(\alpha',\alpha,z',z)+I_2(\alpha',\alpha,z',z)$ and
	\begin{equation*}
		K_1(t,s,w(s))=K^*(t,s,\alpha,z,u_{\alpha,z}(s)+w(s))-K^*(t,s,\alpha,z,u_{\alpha,z}(s)).
	\end{equation*}
	We can rewrite the equation (\ref{Holder-integral-equation}) as
	\begin{equation*}
		w(t)=G(t)+\int_0^t\frac{K_1(t,s,w(s))ds}{(t-s)^{1-\alpha}}.
	\end{equation*}
	Using Part (a) of Theorem \ref{existence-Abel-equation} we obtain
	\begin{equation}
		\Vert w\Vert_{C_\gamma(T,\mathbb{X})}\leq  \Gamma(1-\gamma)E_{\alpha,1-\gamma}(\kappa\Gamma(\alpha)T^\alpha)\Vert G\Vert_{C_\gamma(T,\mathbb{X})}.
		\label{C-difference-inequality}
	\end{equation}
	We estimate $\Vert G\Vert_{C_\gamma(T,\mathbb{X})}$. From the definition of the term we have
	$$   \Vert G\Vert_{C_\gamma(T,\mathbb{X})}\leq \Vert g^*_{\alpha',z'}-g^*_{\alpha,z}\Vert_{C_\gamma(T,\mathbb{X})}+\Vert I_1(\alpha',\alpha,z',z)\Vert_{C_\gamma(T,\mathbb{X})}+\Vert I_2(\alpha',\alpha,z',z)\Vert_{C_\gamma(T,\mathbb{X})}. $$
	For $\alpha<\alpha'$, applying directly Lemma \ref{singular-integral}, Part {\color{red} (a)} gives
	\begin{align*}
		\Vert I_1(\alpha',\alpha,z',z)\Vert &\leq  \int_0^t \kappa\Vert u_{\alpha',z'}(s)\Vert (t-s)^{\alpha-1}((t-s)^{\alpha'-\alpha}-1)ds\\
		&\leq  \kappa \Vert u_{\alpha',z'}\Vert_{C_\gamma(T,\mathbb{X})}
		\int_0^t s^{-\gamma} (t-s)^{\alpha-1}((t-s)^{\alpha'-\alpha}-1)ds\\
		&\leq  \kappa \Vert u_{\alpha',z'}\Vert_{C_\gamma(T,\mathbb{X})}C|\alpha'-\alpha|t^{\alpha-\gamma}(1+|\ln t|).
	\end{align*}
	It follows that
	$$\Vert I_1(\alpha',\alpha,z',z)\Vert_{C_\gamma(T,\mathbb{X})}\leq C\kappa\Vert u_{\alpha',z'}\Vert_{C_\gamma(T,\mathbb{X})}|\alpha'-\alpha|.$$
	The same inequality hold for $\alpha>\alpha'$. Next we estimate $I_2(\alpha',\alpha,z',z)$ as follows
	\begin{align*}
		\Vert I_2(\alpha',\alpha,z',z)\Vert & \leq
		\int_0^t\kappa(|\alpha'-\alpha|^\nu+|z'-z|^\mu)(\Vert u_{\alpha',z'}(s)\Vert+1)
		(t-s)^{1-\alpha}ds\\
		&\leq \kappa (\Vert u_{\alpha',z'}\Vert_{C_\gamma(T,\mathbb{X})}+1)C(|\alpha'-\alpha|^\nu+|z'-z|^\mu)t^{\alpha-\gamma}(1+|\ln t|).
	\end{align*}
	Hence
	$$ \Vert I_2(\alpha',\alpha,z',z)\Vert_{C_\gamma(T,\mathbb{X})}\leq C\kappa(\Vert u_{\alpha',z'}\Vert_{C_\gamma(T,\mathbb{X})}+1)(|\alpha'-\alpha|^\nu+|z'-z|^\mu). $$
	Finally, we can estimate similarly to obtain
	$$ \Vert g^*_{\alpha',z'}-g^*_{\alpha,z}\Vert_{C_\gamma(T,\mathbb{X})}\leq
	C(\kappa_0 +\kappa)(|\alpha'-\alpha|^\nu+|z'-z|^\mu). $$
	Substituting these estimates for $I_1(\alpha',\alpha,z',z), I_2(\alpha',\alpha,z',z)$, $g^*_{\alpha',z'}-g^*_{\alpha,z}$ into (\ref{C-difference-inequality}) gives
	$$ \Vert w\Vert_{C_\gamma(T,\mathbb{X})}\leq  C(\kappa_0+\kappa)E_{\alpha,1-\gamma}(\kappa\Gamma(\alpha)T^\alpha)(\Vert u_{\alpha',z'}\Vert_{C_\gamma(T,\mathbb{X})}+1)(|\alpha'-\alpha|^\nu+|z'-z|^\mu).
	$$
	Using the part (a) of Theorem \ref{existence-Abel-equation}, we obtain
	\begin{eqnarray*}
		\Vert w\Vert_{C_\gamma(T,\mathbb{X})}&\leq & C(\kappa_0+\kappa+1)^2E^2_{\alpha,1-\gamma}(\kappa\Gamma(\alpha)T^\alpha)(|\alpha'-\alpha|^\nu+|z'-z|^\mu)\\
		&\leq& C_0(|\alpha'-\alpha|^\nu+|z'-z|^\mu)
	\end{eqnarray*}
	where
	$$  C_0=C(\kappa_0+\kappa+1)^2 E^2_{\alpha_0,1-\gamma}(\kappa\Gamma(\alpha_1)\max\{T^{\alpha_0}, T^{\alpha_1}\})(1+\Vert g_{\alpha,z}\Vert_{C_\gamma(T,\mathbb{X})}).$$
	We turn to the case of random order. { Using the last inequality and the Jensen's inequality
$\mathbb{E}|X|\leq \left(\mathbb{E}|X|^p\right)^{1/p}$ for $p\geq 1$, we have
	\begin{eqnarray*}
		\mathbb{E}\Vert u_{a_n,z_n}-u_{\alpha,z}\Vert_{C_\gamma(T,\mathbb{X})}&\leq &
		C_0\mathbb{E}(|a_n-\alpha|^\nu+|z_n-z|^\mu)\\
		&\leq& C_0\left(\mathbb{E}(|a_n-\alpha|^\lambda)^{\nu/\lambda}+C_0\sum_{j=1}^k
		(\mathbb{E}|z_{jn}-\beta_j|^{\rho_j})^{\mu_j/\rho_j}\right).
	\end{eqnarray*}
	}
	
	\noindent {\bf Proof of  (d):}  We have
	\begin{eqnarray*}
		I_1(\alpha',\alpha,z',z)
		& =&\int_0^t K^*(t,s,\alpha',z',u_{\alpha',z'}(s))((t-s)^{\alpha'-1}-(t-s)^{\alpha-1})ds,\\
		I_2(\alpha',\alpha,z',z)& =&
		\int_0^t\frac{K^*(t,s,\alpha',z',u_{\alpha',z'}(s))-K^*(t,s,\alpha,z,u_{\alpha',z'}(s))}{(t-s)^{1-\alpha}}ds.
	\end{eqnarray*}
	For $\alpha<\alpha'$, applying directly Lemma \ref{singular-integral} gives
	\begin{eqnarray*}
		\Vert I_1(\alpha',\alpha,z',z)\Vert_{L^p(0,T;\mathbb{X})} &\leq &
		\kappa\Vert u_{\alpha',z'}(s)\Vert_{L^p(0,T;\mathbb{X})}\int_0^t (t-s)^{\alpha-1}((t-s)^{\alpha'-\alpha}-1)ds\\
		&\leq & \kappa \Vert u_{\alpha',z'}\Vert_{L^p(0,T;\mathbb{X})} C|\alpha'-\alpha|t^{\alpha}(1+|\ln t|).
	\end{eqnarray*}
	The same inequality hold for $\alpha>\alpha'$. Now we estimate $I_2(\alpha',\alpha,z',z)$.
	\begin{eqnarray*}
		\Vert I_2(\alpha',\alpha,z',z)\Vert_{L^p(0,T;\mathbb{X})} & \leq&
		\kappa \Vert u_{\alpha',z'}\Vert_{L^p(0,T;\mathbb{X})} C(|\alpha'-\alpha|^\nu+|z'-z|^\mu)
		t^{\alpha}(1+|\ln t|).
	\end{eqnarray*}
	Now, put $w=u_{\alpha',z'}-u_{\alpha,z}$, $G\equiv G(\alpha',\alpha,z',z)=g^*_{\alpha',z'}-g^*_{\alpha,z}+
	I_1(\alpha',\alpha,z',z)+I_2(\alpha',\alpha,z',z)$ and
	\begin{equation}
		K_1(t,s,w(s))=K^*(t,s,\alpha,z,u_{\alpha,z}(s)+w(s))-K^*(t,s,\alpha,z,u_{\alpha,z}(s)).
\label{Definition-K-1}
	\end{equation}
	We can rewrite the equation (\ref{Holder-integral-equation}) as
	\begin{equation*}
		w(t)=G(t)+\int_0^t\frac{K_1(t,s,w(s))ds}{(t-s)^{1-\alpha}}.
	\end{equation*}
	Using Theorem \ref{existence-Abel-equation}, we obtain
	\begin{equation*}
		\Vert w\Vert_{L^p(0,T;\mathbb{X})}\leq \left(1+\kappa\Gamma(\alpha)\alpha^{-1}T^\alpha E_{\alpha,\alpha}(\kappa\Gamma(\alpha) T^\alpha)\right)
		\Vert G\Vert_{L^p(0,T;\mathbb{X})}.
	\end{equation*}
	From the estimate of $G$ we complete the proof of the first part of (d). The second part can be verified as in the proof of previous part.

	\noindent {\bf Proof of  (e):}   We consider the equation
	\begin{equation}
		w_{\alpha,z}(t)=g_1(t)+
		\int_0^t \frac{DK(t,s,\alpha,z,u_{\alpha,z})w_{\alpha,z}(s)}{(t-s)^{1-\alpha}}ds.
		\label{linear-Abel-equation-derivative}
	\end{equation}

	The equation (\ref{linear-Abel-equation-derivative}) can be seen as the "derivative" of the linear Abel equation
	(\ref{linear-Abel-equation}). Using the first part of the proof we deduce that the equation
	(\ref{linear-Abel-equation-derivative}) has a unique solution $w_{\alpha,z}\in C_\gamma(T,\mathbb{X})$.
	
	Now, we claim that $\frac{\partial u_{\alpha,z}}{\partial \alpha}=w_{\alpha,z}$.
	
	Put $w_{\alpha,h}=\frac{u_{\alpha+h,z}-u_{\alpha,z}}{h}$ we obtain
	$$ w_{\alpha,h}(t)=h^{-1}G(\alpha+h,\alpha,z',z)(t)+\int_0^t\frac{K_1(t,s,hw_{\alpha,h}(s))}{h(t-s)^{1-\alpha}}ds  $$
{ where $K_1(t,s,.)$ is defined in (\ref{Definition-K-1}).}
	We note that $h^{-1}G(\alpha+h,\alpha,z',z)\to g_1$ in $C_\gamma(T,\mathbb{X})$. From the Lipschitz property of
	$K(t,s,\alpha,z,w)$ with respect to the variable $w$, we obtain
	$$  \Vert K_1(t,s,w)\Vert \leq \kappa\Vert w\Vert.$$
	Therefore, we have
	$$  \Vert w_{\alpha,h}\Vert_{C_\gamma(T,\mathbb{X})} \leq \Vert h^{-1}G(\alpha+h,\alpha,z',z) \Vert_{C_\gamma(T,\mathbb{X})}
	\Gamma(1-\gamma)E_{\alpha,1-\gamma}(M_1T^\alpha):=M_2.
	$$
	We verify the equicontinuity  of
	$K_1$ with respect to the variable $t$. We first have
	\begin{eqnarray*}
		h^{-1}K_1(t,s,hw(s))&=&h^{-1}(K^*(t,s,\alpha,z,u_{\alpha,z}(s)+hw(s))-K^*(t,s,\alpha,z,u_{\alpha,z}(s)))\\
		&=& \int_0^1DK^*(t,s,\alpha,z,u_{\alpha,z}(s)+h\theta w(s))w(s)d\theta.
	\end{eqnarray*}
	So we have
	\begin{eqnarray*}
		\lefteqn{\Vert h^{-1}K_1(t_2,s,hw(s))-h^{-1}K_1(t_1,s,hw(s))\Vert}\\
		&\leq &\int_0^1\Vert DK^*(t_2,s,\alpha,z,u_{\alpha,z}(s)+h\theta w(s))-DK^*(t_1,s,\alpha,z,u_{\alpha,z}(s)+h\theta w(s))w(s)\Vert ds\\
		&\leq&\omega(|t_2-t_1|)\Vert u_{\alpha,z}(s)+h\theta w(s)) \Vert.
	\end{eqnarray*}
	So we have the equi-continuity with respect to $t$ of the family $h^{-1}K_1(t,s,hw(s))$.
	As in Part (a), we can use a compactness argument to prove that $w_{\alpha,h}\to w_{\alpha,z}$ in $C_\gamma(T,\mathbb{X})$.
	This completes the proof of our theorem.
\end{proof}

\subsection{Continuity of the solutions of fractional equations with sequential derivatives}
We start by applying the previous results for  the general equation (\ref{general-FDE}). For convenience, we recall that if  ${\cal D}^{\sigma_k}_t y(t)=\psi(t)$
then we can rewrite the equation (\ref{general-FDE}) as the integral equation
(\ref{general-Abel-Equation}).
\begin{theorem}
	Let $\eta_0\in (0,1)$ and $\eta_0\leq \eta_j,\eta'_j\leq 1$,   $0<B_0<B_1$ and put
	$z=(\eta_1,\ldots,\eta_k)$, $z'=(\eta'_1,\ldots,\eta'_k)$, $\sigma_k=\sum_{j=1}^k\eta_k$,
	$\sigma'_k=\sum_{j=1}^k\eta'_k$. Assume that $p_j\in C[0,T]$, $j=1,\ldots,k$,
	$b_j\in (B_0,B_1)$. We denote by $\psi_{z}$
	the solution of (\ref{general-Abel-Equation}) and by $y_z$ the solution of  (\ref{general-FDE}).
	\begin{enumerate}[\bf \upshape (a)]
		\item If  $\gamma\in (1-\eta_0,1)$, $f\in C_\gamma(T;\mathbb{X})$, then equation (\ref{general-Abel-Equation}) has a unique solution $\psi_z\in C_\gamma(T;\mathbb{X})$ and the equation (\ref{general-FDE}) has a unique solution $y_z\in C_\gamma(T;\mathbb{X})$ such that ${\cal D}^{\sigma_k}_t y_z(t)=\psi(t)$.
		Moreover, there exists a constant $C=C(\eta_0, B_0,B_1)$ such that
		$$\Vert {\cal D}^{\sigma'_k}_t y_{z'}-{\cal D}^{\sigma_k}_t y_z\Vert_{C_\gamma(T;\mathbb{X})}\leq C\left(1+\Vert f\Vert_{C_\gamma(T;\mathbb{X})}+\sum_{j=1}^k|b_j|\right)\sum_{j=1}^k |\eta'_j-\eta_j|.$$
		In addition, let $\rho_j\geq 1$, $j=1,\ldots,k$. If $\eta_{jn}$   are random,  and $\eta_{jn}\in [\eta_0,1]$, $z_n=(\eta_{1n},\ldots,\eta_{kn})$, $\sigma_{kn}=\sum_{j=1}^k\eta_{jn}$ then
		$$\mathbb{E}\Vert {\cal D}^{\sigma_{kn}}_t y_{z_n}-{\cal D}^{\sigma_k}_t y_z\Vert_{C_\gamma(T;\mathbb{X})}\leq C\left(1+\Vert f\Vert_{C_\gamma(T;\mathbb{X})}+\sum_{j=1}^k|b_j|\right)\sum_{j=1}^k
		(\mathbb{E}|\eta_{jn}-\eta_j|^{\rho_j})^{1/\rho_j}.$$
		\item If $p\in \left[1,\frac{1}{1-\eta_0}\right)$, $f\in L^p(0,T;\mathbb{X})$
		then
		$$\Vert {\cal D}^{\sigma_k}_t y_{z'}-{\cal D}^{\sigma'_k}_t y_z\Vert_{L^p(0,T;\mathbb{X})}\leq C\left(1+\Vert f\Vert_{L^p(0,T;\mathbb{X})}+\sum_{j=1}^k|b_j|\right)\sum_{j=1}^k |\eta'_j-\eta_j|.$$
		In addition, in the case of random order, we have
		$$\mathbb{E}\Vert {\cal D}^{\sigma_{kn}}_t y_{z_n}-{\cal D}^{\sigma_k}_t y_z\Vert_{L^p(0,T;\mathbb{X})}\leq C\left(1+\Vert f\Vert_{L^p(0,T;\mathbb{X})}+\sum_{j=1}^k|b_j|\right)\sum_{j=1}^k
		(\mathbb{E}|\eta_{jn}-\eta_j|^{\rho_j})^{1/\rho_j}.$$
	\end{enumerate}
\end{theorem}
\begin{proof}
	We choose $\mathbb{X}=\mathbb{R}$ and verify the conditions of Theorem \ref{main-theorem-Lipschitz-Abel-equation}, Part (c).
	In fact, we have
	$(k-1)\eta_0\leq \sigma_k-\eta_k=\sigma_{k-1}\leq k-1$ and $(k-1-j)\eta_0\leq \sigma_k-\sigma_j-\eta_k=\sigma_{k-1}-\sigma_j\leq k-1-j$ for every $0\leq j\leq k-1$. This implies
	$$  0\leq (t-s)^{\sigma_k-\eta_k}, (t-s)^{\sigma_k-\sigma_j-\eta_k}\leq \max\{1, T^{k-1}\}:=T_k.$$
	Hence
	\begin{align*}
		\left| K(t,s,z,v)-K(t,s,z,w)\right| &\leq \left|p_k(t)\frac{(t-s)^{\sigma_k-\eta_k}}{\Gamma(\sigma_k)}+
		\sum_{j=1}^{k-1}p_{k-j}(t)\frac{(t-s)^{\sigma_k-\sigma_j-\eta_k}}{\Gamma(\sigma_k)}\right| |v-w|\\
		&\leq \frac{kM_pT_k}{\Gamma((k-1)\eta_0)}|v-w|
	\end{align*}
	where $M_p=\max_{1\leq j\leq k}\Vert p_j\Vert_{C[0,T]}$. So, $K(t,s,z,w)$ is uniformly
	Lipschitz with respect to the variable $w\in\mathbb{X}$. We verify the Lipschitz condition with respect to $z=(\eta_1,\ldots,\eta_k)$.
	\begin{align*}
		K(t,s,z',w)-K(t,s,z,w) &= p_k(t)\left(\frac{(t-s)^{\sigma'_k-\eta'_k}}{\Gamma(\sigma'_k)}-
		\frac{(t-s)^{\sigma_k-\eta_k}}{\Gamma(\sigma_k)}\right)w+\\
		& \sum_{j=1}^{k-1}p_{k-j}(t)\left(\frac{(t-s)^{\sigma'_k-\sigma'_j-\eta_k}}{\Gamma(\sigma'_k)}-\frac{(t-s)^{\sigma_k-\sigma_j-\eta_k}}{\Gamma(\sigma_k)}\right)w.
	\end{align*}
	To prove the Lipschitz property, we use the following inequality
	$$   |\tau^{\delta'}-\tau^{\delta}|=\left|\int_{\delta}^{\delta'}\tau^s\ln \tau ds\right|\leq C(\delta_0,T)|\delta'-\delta|, \ \ \ \ \forall \tau\in [0,T], \delta'\geq\delta\geq\delta_0>0.$$
	Calculating directly gives
	$$ |K(t,s,z',w)-K(t,s,z,w)|\leq C(\eta_0,T)|z'-z|. $$
	Now, we verify the Lipschitz condition of $g_z$ where
	\begin{align*}
		g_z(t) &= f(t)-p_k(t)\sum_{j=1}^k\frac{b_jt^{\sigma_j-1}}{\Gamma(\sigma_j)}-
		\sum_{j=1}^{k-1}p_{k-j}(t)\sum_{\ell=j+1}^k\frac{b_\ell t^{\sigma_\ell-\sigma_j-1}}{\Gamma(\sigma_\ell-\sigma_j)}.
	\end{align*}
	Now, we consider the case $f\in C_\gamma(T;\mathbb{X})$. Multiplying $t^\gamma$ to the last equality gives
	\begin{align*}
		t^\gamma g_z(t) &= t^\gamma f(t)-p_k(t)\sum_{j=1}^k\frac{b_jt^{\gamma+\sigma_j-1}}{\Gamma(\sigma_j)}-
		\sum_{j=1}^{k-1}p_{k-j}(t)\sum_{\ell=j+1}^k\frac{b_\ell t^{\gamma+\sigma_\ell-\sigma_j-1}}{\Gamma(\sigma_\ell-\sigma_j)}.
	\end{align*}
	We have $\gamma+\sigma_j-1\geq \gamma+\eta_0-1>0, \gamma+\sigma_\ell-\sigma_j-1\geq \gamma+\eta_0-1>0$ for every $j=1,\ldots,k$. Hence, using the same estimate as for $|K(t,s,z',w)-K(t,s,z,w)|$ we obtain
	$$ \Vert g_{z'}-g_z\Vert_{C_\gamma(T;\mathbb{X})}\leq C(\eta_0,T)\sum_{j=1}^k|b_j||z'-z|.   $$
	From the above estimates we can apply Theorem \ref{main-theorem-Lipschitz-Abel-equation}, Part (c) to get the result in $C_\gamma(T;\mathbb{X})$.
	The $L^p$-case is similar. Hence, we omit it.
	
\end{proof}

\subsection{Abel linear equations of the first kind}
In this subsection, we  apply the order-continuity results  obtained in the previous subsections to the Abel linear equations of first kind.  We recall that $\Delta_T=\{(t,s,\alpha,z):\ 0\leq s\leq t\leq T,\alpha\in[\alpha_0,\alpha_1], z\in P \}$.
For $K_0: \Delta_T\to \mathbb{R}$, we define
\begin{equation*}
	{\cal A}_{\alpha,z} v(x)=\frac{1}{\Gamma(\alpha)}\int_0^t \frac{K_0(t,s,\alpha,z)v(s)}{(t-s)^{1-\alpha}}ds.
\end{equation*}
In this subsection, we shall investigate the continuity--with respect to  the parameters $\alpha,z$--of
solutions of linear Abel equations  which has the form
\begin{equation}
	{\cal A}_{\alpha,z} v(x)=f(t), \ \ \ \ 0\leq t\leq T.
	\label{linear-Abel-equation}
\end{equation}

\begin{theorem}
	\label{first-kind-representation}
	Let $\gamma\in [0,1)$, $0<\alpha_0<\alpha_1<1$, $\alpha\in[\alpha_0,\alpha_1]$, $z\in P\subset\mathbb{R}^k$, $\mu,\nu\in (0,1]$, $\kappa_1>0$,  assume the following
	\begin{enumerate}[\bf \upshape (i)]
		\item $K_0\in C(\Delta_T), K_0(t,t,\alpha,z)=1$;
		\item $\left\Vert\frac{\partial K_0}{\partial t}\right\Vert_{L^\infty(\Delta_T)}\leq M$;
		\item $\sup_{0\leq s\leq t\leq 1}\left|\frac{\partial K_0}{\partial t}(t,s,\alpha',z')-\frac{\partial K_0}{\partial t}(t,s,\alpha,z)\right|\leq \kappa_1(|\alpha'-\alpha|^\nu+|z'-z|^\mu)$;
	\end{enumerate}
	 and denote
	\begin{align*}
		H(t,s,\alpha,z) &= K_0(t,s,\alpha,z)-K_0(s,s,\alpha,z),\ \ \ \  (t,s,\alpha,z)\in \Delta_T,\\
		L_{\alpha,z}  (t,s)&=-\frac{\sin\pi\alpha}{\pi}\int_{s}^t (t-\tau)^{-\alpha}(\tau-s)^{\alpha-1}((\alpha-1)(\tau-s)^{-1}H(\tau,s,\alpha,z)
		+\frac{\partial H}{\partial t}(\tau,s,\alpha,z))d\tau,\\
		B_{\alpha,z} u_{\alpha,z} (t)&=\int_0^t u_{\alpha,z}(s)L_{\alpha,z}(t,s)ds.
	\end{align*}
	\begin{enumerate}[\bf \upshape (a)]
		\item Let  $f\in L^1(0,T;\mathbb{X})$ satisfy $D^\alpha f\in L^p(0,T;\mathbb{X})$. Then the equation
		(\ref{linear-Abel-equation}) has a unique solution $u_{\alpha,z}\in L^p(0,T;\mathbb{X})$
		such that
		\begin{equation*}
		u_{\alpha,z}(t)=D_t^\alpha f+B_{\alpha,z} u_{\alpha,z}.
		\end{equation*}
		This solution satisfies
		\begin{equation*}
		\Vert u_{\alpha,z}\Vert_{L^p(0,T;\mathbb{X})}\leq C(MT,p)\Vert D_t^\alpha f\Vert_{L^p(0,T;\mathbb{X})}
		\end{equation*}
		where $M=\left\Vert \frac{\partial K_0}{\partial t}\right\Vert_{L^\infty(\Delta_T)}$.
		\item If there is  an $\alpha_2\in (\alpha_1,1)$ such that $D_t^{\alpha_2}f\in L^p(0,T;\mathbb{X})$ then there a constant $C=C(\alpha_0,\alpha_1,P)$
		such that
		
		\begin{equation*}
		\Vert u_{\alpha',z'}-u_{\alpha,z}\Vert_{L^p(0,T;\mathbb{X})}\leq C(|\alpha'-\alpha|^\nu+|z'-z|^\mu).
		\end{equation*}
		In addition, if $(\alpha_n,z_n)$ are random as in Theorem \ref{main-theorem-Lipschitz-Abel-equation} then
		$$\mathbb{E}\Vert u_{\alpha_n,z_n}-u_{\alpha,z}\Vert_{L^p(0,T;\mathbb{X})}
		\leq C\left((\mathbb{E}|\alpha_n-\alpha|^\lambda)^{\nu/\lambda}+
		\sum_{j=1}^k(\mathbb{E}|z_{jn}-\beta_j|^\rho)^{\mu_j/\rho}\right).$$
	\end{enumerate}
\end{theorem}
\begin{proof}
	{\bf Proof of (a):}
	From \cite[page 86,87]{Gorenflo-Vessella}, we have the equality
	$$   {\cal A}_{\alpha,z}=J^\alpha (I-B_{\alpha,z}).$$
	Hence, we obtain
	$$   u_{\alpha,z}(t)=D^\alpha f+B_{\alpha,z} u_{\alpha,z}.$$
	Applying Theorem \ref{existence-Abel-equation} we obtain the desired result.
	
\noindent 	{\bf Proof of (b):} We first estimate $\Vert D^{\alpha'}f(t)-D^\alpha f(t)\Vert$. Assume that $\alpha_0\leq\alpha\leq\alpha'\leq\alpha_1$.
	Since $\alpha_2-\alpha',\alpha_2-\alpha\geq\alpha_2-\alpha_1$, we have in view of Lemma \ref{singular-integral} (b)
	\begin{align*}
		\Vert D^{\alpha'}f-D^\alpha f\Vert_{L^p(0,T;\mathbb{X})} &=\Big\Vert J^{\alpha_2-\alpha'}D^{\alpha_2}f-J^{\alpha_2-\alpha}D^{\alpha_2} f\Big\Vert_{L^p(0,T;\mathbb{X})}\\
		&\leq  C(\alpha_2-\alpha_1)\Vert D^{\alpha_2}f\Vert_{L^p(0,T;\mathbb{X})}|\alpha'-\alpha|.
	\end{align*}
	Put $\tau=s+\theta (t-s)$, we obtain
	\begin{equation*}
		L_{\alpha,z} (t,s)=-\frac{(t-s)^{2-2\alpha}\sin \pi\alpha}{\pi}\int_0^1(1-\theta)^{-\alpha}\theta^{1-\alpha}g(t,s,\alpha,z,\theta)
		d\theta
	\end{equation*}
	where
	\begin{align*}
		g(t,s,\alpha,z,\theta) &= (\alpha-1)\theta^{-1}(t-s)^{-1}H(s+\theta(t-s),s,\alpha,z)\nn\\
		&+\frac{\partial H}{\partial\tau}(s+\theta(t-s),s,\alpha,z).
	\end{align*}
	Since
	$$|H(t,s,\alpha,z)|=|K_0(t,s,\alpha,z)-K_0(s,s,\alpha,z)|\leq M|s-t|,$$
	the Lagrange theorem implies that
	\begin{equation*}
		|g(t,s,\alpha,z,\theta)|\leq 2M.
	\end{equation*}
	We also claim that
	\begin{equation}\label{L-ineq}
		|L_{\alpha',z'} (t,s)-L_{\alpha,z} (t,s)|\leq C(|\alpha'-\alpha|^\nu+|z'-z|^\mu).
	\end{equation}
	In fact, the assumption (iii) implies
	\begin{align*}
		|H(s+\theta(t-s),s,\alpha',z')-H(s+\theta(t-s),s,\alpha,z)|&\leq 2\kappa_1M(|\alpha'-\alpha|^\nu+
		|z'-z|^\mu)\theta|t-s|,\\
		\left| \frac{\partial H}{\partial\tau}(s+\theta(t-s),s,\alpha',z')-\frac{\partial H}{\partial\tau}(s+\theta(t-s),s,\alpha,z)\right| &\leq 2\kappa_1(|\alpha'-\alpha|^\nu+
		|z'-z|^\mu).
	\end{align*}
	Hence,
	$$  |g(t,s,\alpha',z',\theta)-g(t,s,\alpha,z,\theta)| \leq 2\kappa_1(1+M|\alpha-1|)
	(|\alpha'-\alpha|^\nu+|z'-z|^\mu).$$
	Substituting into the formula of $L_{\alpha',z'} (t,s)-L_{\alpha,z} (t,s)$ gives
	\begin{align*}
		|L_{\alpha',z'} (t,s)-L_{\alpha,z} (t,s)| &\leq \frac{2M}{\pi}\int_0^1
		|\sin\pi\alpha'(1-\theta)^{-\alpha'}\theta^{1-\alpha'}-\sin\pi\alpha(1-\theta)^{\alpha}\theta^{1-\alpha}|d\theta+\\
		& \frac{2\kappa_1(1+M|\alpha-1|)}{\pi}(|\alpha'-\alpha|^\nu+|z'-z|^\mu)\int_0^1\sin\pi\alpha(1-\theta)^{\alpha}\theta^{1-\alpha}
		d\theta
	\end{align*}
	which gives directly the inequality \eqref{L-ineq}.
	From the estimates for $D^{\alpha'}f-D^\alpha f$ and $L_{\alpha,z}$, we can use  Part (d) of Theorem
	\ref{main-theorem-Lipschitz-Abel-equation}  to obtain the final inequality of Theorem.
\end{proof}

\subsection{A special  Abel integral equation}

In this section we use the notation  $\phi_\lambda(t,\alpha)$ for the function  which is a solution to the equation
\begin{equation*}
	\partial^\alpha_t\phi_\lambda=\lambda\phi_\lambda,\ \ \ \phi_\lambda(0)=1.
\end{equation*}
We can transform this  equation into the Abel integral equation (see, e.g., \cite[page 131]{Gorenflo-Vessella}
$$ \phi_\lambda(t,\alpha)=1+\lambda J^\alpha\phi_\lambda(t,\alpha).$$
Hence the function $\phi_\lambda$ can be represented using the function $E_{\alpha,1}(z)$. For $\lambda>0$, $\al>0$, and $t>0$ we have
$$\phi_\lambda(t,\alpha)=\phi_\lambda(t)=E_{\alpha,1}(\lambda t^\alpha).$$

Using Theorem \ref{main-theorem-Lipschitz-Abel-equation} with
$$   K(t,s,\alpha,w)=\frac{\lambda}{\Gamma(\alpha)}(t-s)^{\alpha-1}w $$
and $\kappa=\frac{|\lambda|}{\Gamma(\alpha_0)}$ we can get the Lipschitz continuity --with respect to $\alpha\in [\alpha_0, \alpha_1]$-- of the function $\phi_\lambda(t,\alpha)$
as follows
\begin{equation}\label{Mittag-Leffler-lipschitz}
	\Vert \phi_\lambda(.,\alpha')-\phi_\lambda(.,\alpha)\Vert_{C[0,T]}\leq
	C|\alpha'-\alpha| E^2_{\alpha_0,1}(|\lambda|\Gamma(\alpha_1)\max\{T^{\alpha_0}, T^{\alpha_1}\}/\Gamma(\alpha_0)).
\end{equation}

Lemma \ref{Mittag-Leffler} shows that the Lipschitz constant of the inequality is of order $e^{C'T|\lambda|^{1/\alpha}}$ which is very large when $\lambda\to\infty$. Hence, we will look for a better estimate for the case $\lambda<0$. In fact we have

\begin{lemma}\label{mittag-leffler-lemma}
	Letting $\lambda <0$, we obtain the following estimates
	\begin{enumerate}[\bf \upshape (a)]
		\item For  $0<\alpha\leq 1$, we have
		$$|\phi_\lambda(t_1,\alpha)-\phi_\lambda(t_2,\alpha)|\leq C|\lambda|.| t_1-t_2|^\alpha.$$
		\item For $0<\alpha_0<\alpha_1<1,\lambda<-1$ and $\alpha,\alpha'\in[\alpha_0,\alpha_1], \alpha<\alpha'$ and let $\beta, \beta' \in [\beta_0, \beta_1]$. Then there exist constants $C=C(\alpha_0,\alpha_1,\beta_0, \beta_1)$ such that
		\begin{eqnarray*}
			\Big| \phi_\lambda(t,\alpha')-\phi_\lambda(t,\alpha) \Big| & \leq& C|\lambda|.|\alpha'-\alpha|,\\
			\Big| E_{\alpha',1}(-|\lambda|^{\beta'} t^{\alpha'}) -E_{\alpha,1}(-|\lambda|^\beta t^\alpha) \Big| & \le& C |\la|^{\beta_1}|\ln\lambda| \Big( |\beta'-\beta|+ |\alpha'-\alpha| \Big)
		\end{eqnarray*}
		for  $t\in[0,T]$.
		\item Letting $f\in L^2(0,T)$, we put
		\begin{align*}
		G_{\alpha,\lambda}(t)=\int_0^t(t-s)^{\alpha-1}E_{\alpha,\alpha}(\lambda (t-s)^\alpha)f(s)ds.
		\end{align*}
		Then, we have the following estimate
		
		\begin{align}
		\Vert G_{\alpha',\lambda'}-G_{\alpha,\lambda}(t)\Vert_{L^2(0,T)}\leq
		C\Big[|\alpha'-\alpha|(1+|\lambda|)+|\lambda'-\lambda|\Big]\ \Vert f(t)\Vert_{L^2(0,T)}.
		\label{homo-singular-integral}
		\end{align}
	\end{enumerate}
\end{lemma}
\begin{proof}
{\bf Proof of (a):} From Lemma \ref{Mittag-Leffler} (a) the inequality $|t_1^\alpha-t_2^{\alpha}|\leq |t_2-t_1|^\alpha$
	($0\leq \alpha\leq 1$), we have
	\begin{eqnarray*}
		|\phi_\lambda(t_1,\alpha)-\phi_\lambda(t_2,\alpha)|&=&
		|E_{\alpha,1}(\lambda t_1^\alpha)-E_{\alpha,1}(\lambda t_2^\alpha)|\nonumber\\
		&\leq& C|\lambda||t^\alpha_1-t^\alpha_2|
		\leq C|\lambda| |t_1-t_2|^\alpha.\nonumber
	\end{eqnarray*}
	
{\bf Proof of (b):}	 Using Lemma \ref{Mittag-Leffler} (a) gives
\begin{align*}	
 |\phi_\lambda(t,\alpha')-\phi_\lambda(t,\alpha)|  &\leq C|\lambda|.|t^{\alpha'}-t^\alpha|\\
 &\leq C|\lambda| \sup_{\alpha_0\leq\alpha\leq\alpha_1}t^\alpha|\ln t|\ |\alpha'-\alpha|\\
& \leq C'|\lambda|.|\alpha'-\alpha|.
\end{align*}
	Finally, we have
	\begin{align}
		\Big| E_{\alpha',1}(-|\lambda|^{\beta'} t^{\alpha'}) -E_{\alpha,1}(-|\lambda|^\beta t^\alpha) \Big|  &\le \Big| E_{\alpha',1}(-|\lambda|^{\beta'} t^{\alpha'}) -E_{\alpha',1}(-|\lambda|^\beta t^{\alpha'}) \Big|\nn\\
		&\quad \quad  + \Big| E_{\alpha',1}(-|\lambda|^\beta t^{\alpha'}) -E_{\alpha,1}(-|\lambda|^\beta t^\alpha) \Big| \nn\\
		&\quad \quad  \le C \Big| |\lambda|^{\beta'} t^{\alpha'} -|\lambda|^{\beta} t^{\alpha'} \Big|+ C|\lambda|^{\beta}|\alpha'-\alpha|\nn\\
		&\quad \quad  \le C T^{\alpha_1} \Big| |\lambda|^{\beta'} -|\lambda|^{\beta}  \Big|+C|\lambda|^{\beta}|\alpha'-\alpha|\nn\\
		&\quad \quad \le C T^{\alpha_1}  |\la|^{\beta'}  \ln(|\lambda|) |\beta'-\beta|+ C\lambda^{\beta}|\alpha'-\alpha|. \nn
	\end{align}
	
{\bf Proof of (c):} For $\lambda',\lambda<0$, we have
	\begin{eqnarray*}
		G_{\alpha',\lambda'}(t)-G_{\alpha,\lambda}(t)&=&
		\int_0^t[(t-s)^{\alpha'-1}-(t-s)^{\alpha-1}]E_{\alpha',\alpha'}(\lambda' (t-s)^{\alpha'})f(s)ds\\
& &+\int_0^t(t-s)^{\alpha-1}[E_{\alpha',\alpha'}(\lambda' (t-s)^{\alpha'})-E_{\alpha',\alpha'}(\lambda (t-s)^{\alpha'})]f(s)ds\\
		& &+\int_0^t(t-s)^{\alpha-1}[E_{\alpha',\alpha'}(\lambda (t-s)^{\alpha'})-E_{\alpha,\alpha}(\lambda (t-s)^{\alpha'})]f(s)ds\\
		& &+\int_0^t(t-s)^{\alpha-1}[E_{\alpha,\alpha}(\lambda (t-s)^{\alpha'})-E_{\alpha,\alpha}(\lambda (t-s)^{\alpha})]f(s)ds.
	\end{eqnarray*}
Inserting appropriate terms deduces
	\begin{eqnarray*}
		\Vert G_{\alpha',\lambda'}-G_{\alpha,\lambda}\Vert^2_{L^2(0,T)}&\leq&
		4(K^2_1+K^2_2+K^2_3+K^2_4) \Vert f\Vert^2_{L^2(0,T)}
	\end{eqnarray*}
where
\begin{eqnarray*}
		K_1&=&
		\int_0^T\left|[s^{\alpha'-1}-s^{\alpha-1}]E_{\alpha',\alpha'}(\lambda' s^{\alpha'})\right|ds,\\
K_2 & =&\int_0^T\left|s^{\alpha-1}[E_{\alpha',\alpha'}(\lambda' s^{\alpha'})-E_{\alpha',\alpha'}(\lambda s^{\alpha'})]\right|ds,\\
K_3		& =&\int_0^T\left|s^{\alpha-1}[E_{\alpha',\alpha'}(\lambda s^{\alpha'})-E_{\alpha,\alpha}(\lambda s^{\alpha'})]\right|ds,\\
K_4		& =&\int_0^T\left|s^{\alpha-1}[E_{\alpha,\alpha}(\lambda s^{\alpha'})-E_{\alpha,\alpha}(\lambda s^{\alpha})]\right|ds.
	\end{eqnarray*}
{	
Using Lemma \ref{singular-integral} (a) gives
\begin{eqnarray*}
		K_1&\leq&
		C\int_0^T|s^{\alpha'-1}-s^{\alpha-1}|ds\leq C|\alpha'-\alpha|.
\end{eqnarray*}
For $K_2$, we use Lemma \ref{Mittag-Leffler} (a) to obtain
\begin{eqnarray*}
K_2 &\leq &\int_0^T\left|s^{\alpha-1}[E_{\alpha',\alpha'}(\lambda' s^{\alpha'})-E_{\alpha',\alpha'}(\lambda s^{\alpha'})]\right|ds\\
&\leq& C|\lambda'-\lambda|\int_0^Ts^{\alpha-1}s^{\alpha'}ds\leq C'|\lambda'-\lambda|.
\end{eqnarray*}
Similarly
\begin{eqnarray*}
K_4	\leq C|\lambda|.|\alpha'-\alpha|.
	\end{eqnarray*}
Finally, for $z\leq 0$
\begin{eqnarray*}
  |E_{\alpha',\alpha'}(z)-E_{\alpha,\alpha}(z)|&\leq&
|E_{\alpha',\alpha'}(z)-E_{\alpha,\alpha'}(z)|+|E_{\alpha,\alpha'}(z)-E_{\alpha,\alpha}(z)|
\leq C|\alpha'-\alpha|
\end{eqnarray*}
where
$$   C=\sup\left\{  \left|\frac{\partial E_{\alpha,\beta}}{\partial\alpha}(z)\right|+
\left|\frac{\partial E_{\alpha,\beta}}{\partial\beta}(z)\right|:\
(\alpha,\beta,z)\in[\alpha_0,\alpha_1]\times[\beta_0,\beta_1]\times (-\infty,0]\right\}.$$
Hence, we can use Lemma \ref{Mittag-Leffler} (a) to obtain
\begin{eqnarray*}
K_3		& \leq &\int_0^T\left|s^{\alpha-1}[E_{\alpha',\alpha'}(\lambda s^{\alpha'})-E_{\alpha,\alpha}(\lambda s^{\alpha'})]\right|ds\leq C|\alpha'-\alpha|.\\
\end{eqnarray*}
Combining the estimates for $K_1,K_2,K_3,K_4$ gives
	\begin{align*}
		\Vert G_{\alpha',\lambda'}-G_{\alpha,\lambda}\Vert_{L^2(0,T)}\leq
C\Big[|\alpha'-\alpha|(1+|\lambda|)+|\lambda'-\lambda|\Big]\ \Vert f\Vert_{L^2(0,T)}.
	\end{align*}
	This completes the proof of (c).
}	
\end{proof}

\section{Continuity of the solutions of some space-time fractional partial differential equations.}

\setcounter{equation}{0}

In this section, we will consider the continuity of solutions of some abstract partial differential equations with respect to the parameters $\beta, \alpha$ of some fractional partial differential equations.

\subsection{The abstract fractional diffusion equation in a Banach space}
 We first investigate the continuity  properties  in a Banach space.
To this end, an outline for classical definitions in the  theory of semigroup on Banach spaces is necessary. Let $\mathbb{X}$ be a Banach space as the in previous section and ${\cal L}(\mathbb{X})$ be the set of bounded linear operators on $\mathbb{X}$.  Let $B: D(B)\to \mathbb{X}$
($D(B)\subset\mathbb{X}$) be a closed linear operator. We denote (see \cite{Engel-Nagel}, page 55) the spectrum set of $B$, the resolvent set of $B$  by
\begin{align*}
	\sigma(B)&= \Big\{\lambda\in\mathbb{C}:\ \lambda-B\ {\rm \ is\ not\ bijective} \Big\},\\
	\rho(B)&= \Big\{\lambda\in\mathbb{C}:\ \lambda\not\in \sigma(B) \Big\}
\end{align*}
respectively. For $0\leq \omega<\pi$ we denote the sector with angle $\omega$ by
$$  \Sigma_\omega=\{\lambda\in \mathbb{C}:\ \lambda\not=0, |{\rm arg}\ \lambda|<\omega\}.$$
As in \cite[page 91]{triebel} we say that the operator $B$ is positive if $[0,\infty)\subset \rho(-B)$ and
$$   \sup_{\lambda\geq 0}\Vert (\lambda+1)(\lambda I+B)^{-1}\Vert_{{\cal L}(\mathbb{X})}<\infty. $$
We denote by $\Phi_B$ the set of real numbers $\eta\in (0,\pi]$ satisfying $\overline{\Sigma_\eta}\subset\rho(-B)$ and
$$   \sup_{\lambda\in \overline{\Sigma_\eta}}\Vert (\lambda+1)(\lambda I+B)^{-1}\Vert_{{\cal L}(\mathbb{X})}<\infty. $$
From the positivity of $B$ , we can verify that $\Phi_B\not=\emptyset$. We define the spectral angle of $B$ by
$$\phi_B=\inf \{\omega\in(0,\pi]:\ \pi-\omega\in\Phi_B\}.$$
As in \cite{ClementJDE} we consider two Banach spaces $X_1\subset X_0 $  such that $X_0$ is dense in $X_1$ and that
$A:X_1\to X_0$ is an isomorphism.
We denote by $X_\theta=[X_0,X_1]_\theta$  ($0\leq \theta\leq 1$) the interpolation spaces between $X_0$ and $X_1$ defined by the folowing: element $\xi\in X_\theta $ if and only if
$$\lim_{|\lambda|\to\infty, |{\rm arg}\ \lambda|<\eta} \Vert \lambda^\theta A(\lambda I+A)^{-1}\xi\Vert_{X_0}=0.$$
for every $0\leq\eta<\pi-\phi_A$. With $\eta$ fixed, The space $X_\theta$, called the abstract H\"older space
(see \cite{Engel-Nagel}, page 130),  is a Banach space with the norm
$$   \Vert \xi\Vert_{X_\theta}= \sup_{|{\rm arg}\ \lambda|<\eta,\lambda\not= 0}
\Vert \lambda^\theta A(\lambda I+A)^{-1}\xi\Vert_{X_0}.$$
We can verify directly that $X_{\theta_1}\cap X_{\theta_2}\subset X_{\theta}$ for $0\leq \theta_2\leq\theta\leq\theta_1\leq 1$.

For $\xi\in X_0$, we consider the problem of finding $u: [0,T]\to X_1$ such that

\begin{equation*}
 D^\alpha_t(u-\xi)+Au=0,\ \ \ u(0)=\xi.
\end{equation*}

Let $\alpha\in (0,2)$, $\gamma\in (0,1)$.  We define the following space of functions
\begin{eqnarray*}
	C_{\gamma,0}(T, X)&=&\{h:(0,T]\to \mathbb{X}: h\in C_\gamma(T,X), \lim_{t\to 0+}t^\gamma\Vert h(t)\Vert=0\},\\
	C^\alpha_{\gamma,0}(T,X_0,X_1)&=&\{h: h\in C_{\gamma,0}(T, X_1);\ D_t^\alpha h\in C_{\gamma,0}(T, X_0) \}
\end{eqnarray*}
with the respective norms
\begin{eqnarray*}
	\Vert u\Vert_{C_{\gamma,0}(T, X)} &=& \sup_{t\in(0,1]}t^\gamma\Vert u(t)\Vert_X,\\
	\Vert v\Vert_{C^\alpha_{\gamma,0}(T,X_0,X_1)}&=& \sup_{t\in(0,1]}t^\gamma
	(\Vert D^\alpha_t v(t)\Vert_{ X_0}+\Vert v(t)\Vert_{X_1}).
\end{eqnarray*}
We define the trace operator
$$  Tr: C^\alpha_{\gamma,0}(T,X_0,X_1)\to X_0\ \ {\rm by}\ Tr(u)=u(0).$$

The next theorem establishes  the existence, uniqueness and continuity of solution to time fractional diffusion type equations in a Banach space.
\begin{theorem}
	\begin{enumerate}[\bf \upshape(a)]
		\item Let $\alpha\in (0,2)$, $\gamma\in (0,\min\{1,\alpha\})$.
		We assume that
		$A$ is  nonegative with the spectral angle $\phi_A$ satisfying
		$$  0<\phi_A<\pi \left(1-\frac{\alpha}{2}\right).$$
		Put $\hat{\gamma}=\frac{\gamma}{\alpha}$. If $\xi\in X_{1-\hat{\gamma}}$ then
		$$  Tr(C^\alpha_{\gamma,0}(T,X_0,X_1))=X_{1-\hat{\gamma}} $$
		and the problem
		$$D^\alpha_t(u-\xi)+Au=0$$
		has a unique solution $u_\alpha\in C^\alpha_{\gamma,0}(T,X_0,X_1)$ given by
		\begin{equation*}
		u_\alpha(t)=\frac{1}{2\pi i}\int_{\gamma_{1,\varphi}}e^{\lambda t}(\lambda^\alpha I+A)^{-1}\lambda^{\alpha-1}\xi d\lambda.
		\end{equation*}
		Here $\varphi\in \left(\frac{\pi}{2},\frac{\pi-\phi_A}{\alpha}\right)$
		and  the curve $\gamma_{1,\varphi}$ is defined in (\ref{gammarhovarphi}).
		\item For $ 0<\alpha_0<\alpha_1<2, \alpha\in [\alpha_0,\alpha_1]$, $\gamma\in (0,1)$ and $0< \gamma<\min\{1,\alpha_0\}$,
		if
		$A$ is  nonnegative with the spectral angle $\phi_A$ satisfying
		$$  0<\phi_A<\pi \left(1-\frac{\alpha_1}{2}\right),$$
		then
		$$  Tr(C^\alpha_{\gamma,0}(T,X_0,X_1))=X_{1-\hat{\gamma}} $$
		for every $0< \gamma<\min\{1,\alpha_0\}$.
		Moreover, put
		$$\hat{\gamma}_0=\frac{\gamma}{\alpha_0}, \hat{\gamma}_1=\frac{\gamma}{\alpha_1}.$$
		If   $\xi\in X_{1-\hat{\gamma}_0}\cap X_{1-\hat{\gamma}_1}$ and
		$\varphi\in\left(\frac{\pi}{2},\frac{\pi-\phi_A}{\alpha_1}\right)$ then the problem
		$$D^\alpha_t(u-\xi)+Au=0$$
		has a unique solution $u_{\alpha,\xi}\in C^\alpha_{\gamma,0}(T,X_0,X_1)$ satisfying
		$$   u_{\alpha,\xi}(t)=\frac{1}{2\pi i}\int_{\gamma_{1,\varphi}}e^{\lambda t}(\lambda^\alpha I+A)^{-1}\lambda^{\alpha-1}\xi d\lambda$$
		for every $\alpha\in [\alpha_0,\alpha_1]$, $\varphi\in \left(\frac{\pi}{2},\frac{\pi-\phi_A}{\alpha_1}\right)$. Moreover, there exists a constant $C=C(\alpha_0,\alpha_1)$ such that
		$$  \Vert u_{\alpha',\xi'}(t)-u_{\alpha,\xi}(t)\Vert_{X_0}\leq C(1+\Vert\xi\Vert_{X_0})\left(|\alpha'-\alpha|+
		\Vert \xi'-\xi\Vert_{X_0}
		\right)$$
		for all  $\alpha,\alpha'\in [\alpha_0,\alpha_1],\ \xi,\xi'\in X_{1-{\hat{\gamma}_0}}
		\cap
		X_{1-{\hat{\gamma}_1}}
		.$
	\end{enumerate}
\end{theorem}
\begin{proof}
	{\bf Proof of (a):} Readers can be found in  Clement et al. \cite{ClementTAMS}.\\
	{\bf Proof of (b):} Since $A$ is  nonegative with the spectral angle $\phi_A$ satisfying $ 0<\phi_A<\pi \left(1-\frac{\alpha_1}{2}\right)$ and $\alpha\in [\alpha_0,\alpha_1]$ we have $ 0<\phi_A<\pi \left(1-\frac{\alpha}{2}\right)$. Hence we obtain
	$$  Tr(C^\alpha_{\gamma,0}(T,X_0,X_1))=X_{1-\hat{\gamma}} $$
	as in (a). Since $X_{1-\hat{\gamma}_0}\cap X_{1-\hat{\gamma}_1}\subset X_{1-\hat{\gamma}}$, if
	$\xi\in X_{1-\hat{\gamma}_0}\cap X_{1-\hat{\gamma}_1}$ then $\xi\in X_{1-\hat{\gamma}}$. Applying (a),
	we have
	$$   u_\alpha(t)=\frac{1}{2\pi i}\int_{\gamma_{1,\varphi}}e^{\lambda t}(\lambda^\alpha I+A)^{-1}\lambda^{\alpha-1}\xi d\lambda\ \ \ \ {\rm for\ every}\ \alpha\in [\alpha_0,\alpha_1].$$
	For $\lambda\in\gamma_{1,\varphi}, |\lambda|>1$, we have $\arg(\lambda)=\pm\varphi$. It follows that
	$$|\arg\ \lambda^\alpha|=\varphi\alpha\leq\varphi\alpha_1<\pi-\phi_A\ \ \ {\rm for}\ \lambda\in\gamma_{1,\varphi}, |\lambda|>1,
	\alpha\in[\alpha_0,\alpha_1].$$
	Hence
	\begin{eqnarray}
		\sup_{\lambda\in\gamma_{1,\varphi}}\Vert (\lambda^\alpha I+A)^{-1}\Vert
		&\leq &
		\sup_{\lambda\in\gamma_{1,\varphi}}\Vert \lambda^\alpha(\lambda^\alpha I+A)^{-1}\Vert\nonumber\\
		&\leq &
		C_{\omega_0}\sup_{\lambda\in\gamma_{1,\varphi}}\Vert (1+\lambda^\alpha)(\lambda^\alpha I+A)^{-1}\Vert
		\leq C_{\omega_0}M_{\omega_0}
		\label{spectral-estimate}
	\end{eqnarray}
	where $\omega_0=\pi-\varphi\alpha_1$,
	$C_{\omega_0}=\sup_{\lambda\in \overline{\Sigma_{\omega_0}}}|\lambda (\lambda+1)^{-1}|$ and
	$$M_{\omega_0}=\sup_{\lambda\in \overline{\Sigma_{\omega_0}}}\Vert (\lambda+1)(\lambda I+A)^{-1}\Vert.$$
	We consider the last inequality. For $\alpha,\alpha'\in[\alpha_0,\alpha_1]$ we can write
	\begin{eqnarray*}
		u_{\alpha',\xi'}(t)-u_{\alpha,\xi}(t) &=& J_1+J_2+J_3
	\end{eqnarray*}
	where
	\begin{eqnarray*}
		J_1 &=& \frac{1}{2\pi i}\int_{\gamma_{1,\varphi}}e^{\lambda t}((\lambda^{\alpha'} I+A)^{-1}-
		(\lambda^\alpha I+A)^{-1})\lambda^{\alpha'-1}\xi' d\lambda,\\
		J_2&=& \frac{1}{2\pi i}\int_{\gamma_{1,\varphi}}e^{\lambda t}
		(\lambda^\alpha I+A)^{-1}(\lambda^{\alpha'-1}-\lambda^{\alpha-1})\xi' d\lambda,\\
		J_3&=& \frac{1}{2\pi i}\int_{\gamma_{1,\varphi}}e^{\lambda t}
		(\lambda^\alpha I+A)^{-1}\lambda^{\alpha-1}(\xi'-\xi) d\lambda.
	\end{eqnarray*}
	From (\ref{spectral-estimate}) we have
	$$\Vert J_3\Vert \leq  \frac{1}{2\pi}\int_{\gamma_{1,\varphi}}e^{Re\lambda t}
	C_{\omega_0}M_{\omega_0} \Vert \xi'-\xi\Vert_{X_0} |d\lambda|\leq C\Vert\xi'-\xi\Vert_{X_{0}}. $$
	Using $(\lambda^{\alpha} I+A)^{-1}-
	(\lambda^{\alpha'} I+A)^{-1}=(\lambda^{\alpha'}-\lambda^{\alpha})
	(\lambda^{\alpha} I+A)^{-1}(\lambda^{\alpha'} I+A)^{-1}$ and  estimating directly $J_1$, we obtain in view of
	(\ref{spectral-estimate})
	\begin{eqnarray*}
		\Vert J_1\Vert_{X_0}&\leq &\frac{1}{2\pi}\int_{\gamma_{1,\varphi}}e^{Re\lambda t}
		|\lambda^{\alpha'}-\lambda^{\alpha}|
		C^2_{\omega_0}M^2_{\omega_0} \Vert \xi'\Vert_{X_0} |d\lambda|
		\leq C|\alpha'-\alpha|\Vert\xi'\Vert_{X_{0}}.
	\end{eqnarray*}
	Similarly,
	\begin{eqnarray*}
		\Vert J_2\Vert_{X_0}&\leq &\frac{1}{2\pi}\int_{\gamma_{1,\varphi}}e^{Re\lambda t}
		|\lambda^{\alpha'}-\lambda^{\alpha}|
		C_{\omega_0}M_{\omega_0} \Vert \xi'\Vert_{X_0} |d\lambda|
		\leq C|\alpha'-\alpha|\Vert\xi'\Vert_{X_{0}}.
	\end{eqnarray*}
	This completes the proof of the theorem.
\end{proof}

\subsection{The abstract fractional diffusion equation in a Hilbert space}\label{tfpde-hilbert-space}

In this section we will study the existence, uniqueness and the continuity of  solutions to time fractional equation in a Hilbert  space. The main results are theorems \ref{Hilbert-case}, \ref{homogeneous-fractional-diffusion-theorem}, and \ref{nonhomogeneous-fractional-theorem}.

We denote by $V$ and $H$ the real separable Hilbert spaces, $V'$ the dual of $V$ and
$\langle .,.\rangle$ the inner product of $H$. We assume that the space $V$ is dense in $H$ and continuously embedded into $H$. Usually, we write  $V\hookrightarrow H\hookrightarrow V'$ and
$$  \varphi(v):=\langle \varphi,v\rangle_{V'\times V}, \ \ \ \forall \varphi \in V', v\in V.$$
For $x\in H$, $\alpha\in (0,1]$ we put
$$W_\alpha(x,V,H):=\{u\in L^2(0,T;V):\ D_t^{\alpha} (u-x)\in L^2(0,T;V'), u(0)=x \}.$$
We note that $W_\alpha(x,V,H)\subset W_{\alpha'}(x,V,H)$ for $0<\alpha'<\alpha\leq 1$.
For $\gamma\in (0,1)$, we denote by
$$ H^\gamma(\mathbb{R}; V,V')=\{v:\ v\in L^2(\mathbb{R}; V),\
|\tau|^\gamma \hat{v}\in L^2(\mathbb{R}; V')\}$$
where $\hat{v}(\tau)=\int_{-\infty}^\infty v(t)e^{-it\tau}dt.$
The space $H^\gamma(\mathbb{R}; V,V')$ is a Hilbert space with the norm
$\Vert v\Vert^2_{H^\gamma(\mathbb{R}; V,V')}=\Vert v\Vert^2_{L^2(\mathbb{R}; V)}+
\Vert |\tau|^\gamma\hat{v}\Vert^2_{L^2(\mathbb{R}; V')}$. From the latter space we can define the space of all functions having the fractional derivetives of order $\gamma$ by putting the set
$$H^\gamma(0,T; V,V')=\{w:\ w=v\mid_{(0,T)},\ v\in H^\gamma(\mathbb{R}; V,V')\}.$$
We have, see e.g. \cite{Lions}, page 61,
$$H^\gamma(0,T; V,V')=\{w:\ w\in L^2(0,T;V),\ D_t^\gamma w\in L^2(0,T;V')\}.$$

We denote by $K$ a compact subset in $\mathbb{R}^k$. For every $\beta\in K$, let $a_\beta(t,.,.):V\times V\rightarrow \mathbb{R}$ be a bilinear functional and $f_\beta\in L^2(0,T; V')$. We define the operator $B_\beta(t): V\rightarrow V'$ by
$$  \left\langle B_\beta(t)v,w \right\rangle_{V'\times V}=a_\beta(t,w,v) \ \ \ \forall v,w\in V.$$
We consider the problem of finding $u_{\alpha,\beta}\in W_\alpha(x,V,H)$ such that
\begin{equation}
	\label{Hilbert-abstract-equation}
	\frac{d}{dt}\langle J^{1-\alpha}(u_{\alpha,\beta}-x),v\rangle+a_\beta(t,u_{\alpha,\beta}(t),v)=\langle f_\beta(t),v\rangle_{V'\times V},\ \ \ \ v\in V, a.a. t\in (0,T).
\end{equation}
We have the theorem which establishes  the existence, uniqueness and the continuity of  solutions to time fractional diffusion type equation in a Hilbert  space.

\begin{theorem}\label{Hilbert-case}
	\begin{enumerate}[\bf \upshape (a)]
		\item Let $T,M,c,d,$ be constants, $0<\alpha_0<\alpha_1<1$,  $\alpha\in [\alpha_0,\alpha_1]$. We assume that $a_\beta(t,.,.):V\times V\rightarrow \mathbb{R}$ is a bilinear functional satisfying A1--A2 below
		
		{\bf A1)}
		$  |a_\beta(t,v,w)|\leq M \Vert v\Vert_V\Vert w\Vert_V,\ \ \ \ \forall v,w\in V,\beta\in K,$
		
		{\bf A2)}
		$   a_\beta(t,v,v)\geq c\Vert v\Vert^2_V-d\Vert v\Vert^2_H,\ \ \ \ \forall v\in V,$
		
		\noindent for a.a. $t\in (0,T)$. \\
		Then, for $x\in H, f_\beta\in L^2(0,T;V')$ there exists  a  unique function
		$ u_{\alpha,\beta}\in W_\alpha(x,V,H)$
		satisfying (\ref{Hilbert-abstract-equation}).
		
		Moreover, there exists a constant $M_0>0$
		independent of $\alpha,\beta$
		such that
		\begin{equation}
		\Vert D^\alpha_t(u_{\alpha,\beta}-x)\Vert_{L^2(0,T;V')}+\Vert u_{\alpha,\beta}\Vert_{L^2(0,T;V)}\leq M_0\left(\Vert x\Vert+
		\Vert f_\beta\Vert_{L^2(0,T;V')}\right).
		\label{Galerkin-bound}
		\end{equation}
		\item Let $\alpha_n\in [\alpha_0,\alpha_1], \beta_n\in K$ for every $n\in\mathbb{N}$ and
		$$ \lim_{n\to\infty}\alpha_n=\alpha, \lim_{n\to\infty}\beta_n=\beta. $$
		Assume A3 and  A4, or A4'  below hold
		
		A3) $\lim_{\beta_n\to\beta}\Vert f_{\beta_n}-f_\beta\Vert_{L^2(0,T;V')}=0$;

		A4)  $\lim_{n\to\infty}\Vert B_{\beta_n}(.)v-B_\beta(.)v\Vert_{L^2(0,T;V')}=0$ for every $v\in V$;

		A4')  there is a subset $D\subset V$ such that $span\{ D\}$ is dense in $H$ and that
		$\lim_{n\to\infty}\Vert B_{\beta_n}(.)v-B_\beta(.)v\Vert_{L^2(0,T;H')}=0$ for every $v\in D$.
		
		Then we have $u_{\alpha_n,\beta_n}\to u_{\alpha,\beta}$ as $n\to\infty$ in
		$L^2(0,T;V)$.

		In addition, if $V$ is compactly embedded in $H$ then
		$$     \lim_{n\to\infty}\Vert u_{\alpha_n,\beta_n}-u_{\alpha,\beta}\Vert_{L^2(0,T;V')}=0.   $$
	\end{enumerate}
\end{theorem}

\begin{proof}
	The proof of part (a) is given  in \cite{Zacher}. We now prove part (b) of the theorem. By (\ref{Galerkin-bound}) there exist $u\in L^2(0,T;V)$ and a subsequence $u_{\alpha_{n_k},\beta_{n_k}}$, still denote by $u_{\alpha_n,\beta_n}$, such that
	$$u_{\alpha_n,\beta_n}\to u \ \ \ \text{ as } n\to\infty\text{ in }
	L^2(0,T;V).$$
	We will prove that $u$ satisfies the equation
	\begin{equation}
		\langle D^\alpha_t(u-x),v\rangle+a_\beta(t,u,v)=\langle f_\beta,v\rangle,\ \ \ \forall v\in V.
		\label{Galerkin-limit-equation}
	\end{equation}
	We have
	$$ -\int_0^T\varphi'(t)\langle J^{1-\alpha_n}(u_{\alpha_n,\beta_n}-x),v\rangle dt+
	\int_0^T\varphi(t)a_{\beta_n}(t,u_{\alpha_n,\beta_n}(t),v)dt=\int_0^T\varphi(t)\langle f_{\beta_n}(t),v\rangle_{V'\times V}dt.  $$
	We consider the first term of the equality
	\begin{align*}
		\int_0^T\varphi'(t)\langle J^{1-\alpha_n}(u_{\alpha_n,\beta_n}-x),v\rangle dt &=
		\int_0^T\varphi'(t)\langle (J^{1-\alpha_n}-J^{1-\alpha})(u_{\alpha_n,\beta_n}-x),v\rangle dt\\
		&+ \int_0^T\varphi'(t)\langle J^{1-\alpha}u_{\alpha_n,\beta_n},v\rangle dt\equiv I_{1n}+I_{2n}.
	\end{align*}
	We have
	$$ |I_{1,n}|\leq C\Vert \varphi'\Vert_{L^\infty(0,T)}\Vert v\Vert_{L^2(0,T;V)}
	\Vert (J^{1-\alpha_n}-J^{1-\alpha})(u_{\alpha_n,\beta_n}-x)\Vert_{L^2(0,T;V')}.$$
	Using Lemma \ref{singular-integral} (b) gives
	$$ |I_{1,n}|\leq C\Vert \varphi'\Vert_{L^\infty(0,T)}\Vert v\Vert_{L^2(0,T;V)}|\alpha_n-\alpha|\to 0\ \ \ {\rm as}\ n\to\infty. $$
	On the other hand, $I_{2n}\to \int_0^T\varphi'(t)\langle J^{1-\alpha}u,v\rangle dt $, Hence
	$$  \lim_{n\to\infty}\int_0^T\varphi'(t)\langle J^{1-\alpha_n}u_{\alpha_n,\beta_n},v\rangle dt=\int_0^T\varphi'(t)\langle J^{1-\alpha}u,v\rangle dt.$$
	Now, we consider the second term of (\ref{Galerkin-limit-equation}). If A4) holds then
	\begin{align*}
		\int_0^T\varphi(t)a_{\beta_n}(t,u_{\alpha_n,\beta_n},v)dt&=
		\int_0^T\langle u_{\alpha_n,\beta_n}, B_{\beta_n}(t)v-B_{\beta}(t)v\rangle_{V\times V'}+
		\int_0^T\varphi(t)\langle u_{\alpha_n,\beta_n},B_\beta v\rangle_{V\times V'}dt\\
		& \to \int_0^T\varphi(t)\langle u,B_\beta(t) v\rangle_{V\times V'}dt=
		\int_0^T\varphi(t)a_{\beta}(t,u,v)dt\ \ \ {\rm as}\ n\to\infty.
	\end{align*}
	If A4') holds then
	\begin{align*}
		\int_0^T\varphi(t)a_{\beta_n}(t,u_{\alpha_n,\beta_n},v)dt&=
		\int_0^T\langle u_{\alpha_n,\beta_n}, B_{\beta_n}(t)v-B_{\beta}(t)v\rangle_{H\times H'}+
		\int_0^T\varphi(t)\langle u_{\alpha_n,\beta_n},B_\beta v\rangle_{H\times H'}dt\\
		& \to \int_0^T\varphi(t)\langle u,B_\beta(t) v\rangle_{H\times H'}dt=
		\int_0^T\varphi(t)a_{\beta}(t,u,v)dt\ \ \ {\rm as}\ n\to\infty.
	\end{align*}
	These limits imply that the function $u$ is the weak solution of (\ref{Galerkin-limit-equation}). By the uniqueness, we have $u=u_{\alpha,\beta}$.
	Since the limit holds for every subsequence of $u_{\alpha_n,\beta_n}$, we obtain $u_{\alpha_n,\beta_n}\to u_{\alpha,\beta}$ as $n\to\infty$
	in $L^2(0,T;V)$. Finally, for $\gamma\in (0,\alpha_0)$,  we have
	$u_{\alpha_n,\beta_n}, u_{\alpha,\beta}\in H^\gamma(0,T; V,V')$.
	If $V$ is compactly embedded into $H$ then Theorem 5.2 in \cite{Lions}, page 61, implies that $H^\gamma(0,T; V,V')$ is compactly embedded in $L^2(0,T;V')$. Hence
	$$ \lim_{n\to\infty}\Vert u_{\alpha_n,\beta_n}-u_{\alpha,\beta}\Vert_{L^2(0,T;V')}=0.$$
\end{proof}
\subsection{Continuity of time fractional equations with  a countable spectrum in a Hilbert space}
The continuity of the solutions  in  Theorem \ref{Hilbert-case} are  very weak. To get the stronger continuity, we  consider  an {\bf initial value  problem in a Hilbert space}. In particular we consider operators that have a countable spectrum. Let  $H$ be a Hilbert space with the inner-product
$\langle.,.\rangle$ and the norm $\Vert.\Vert$.  Let $A: D(A)\rightarrow H$ be an operator defined on the subset $D(A)$ (called the domain of $A$) which is dense in $H$. We assume that $A$
has the eigenvectors $\phi_p\in D(A)$ corresponding to the eigenvalues $\lambda_p$,
i.e.,
$$
A\phi_p=\lambda_p\phi_p\ \mathrm{for}\ p=1,2,...
$$
We also assume that $0<\lambda_p\leq \lambda_{p+1}$, $\lim_{p\to\infty}\lambda_p=\infty$ and that
$\{\phi_p\}$ is an orthonormal basis of $H$.
For $s \ge 0$, we denote by $H^s \subset H$ the Hilbert space  as collection of function with the bounded norm induced by the following
$$
\|v\|_{H^{s}} =\sqrt{\sum\limits_{p=1}^{\infty}  \la_p^{2s} |\langle v,\phi_{p}\rangle|^{2}},
$$
where $H^0=H$ and $H^1=D(A)$. For $\gamma\ge 0$, we denote by $L^2_{\gamma,s}(T)$ the Hilbert space of functions $f:(0,T)\to H^s$ such that
$$ \Vert f\Vert^2_{L^2_{\gamma,s}(T)}=\int_0^Tt^{2\gamma}\Vert f(t)\Vert^2_{H^s}dt<\infty.   $$
For a constant $\beta>0,$
using the spectral theory, (fractional) powers $A^\beta$ can be defined by
\begin{equation}\label{spectral-frac-operator}
A^\beta v  := \sum_{k=1}^\infty \lambda_p^{\beta} \langle v,\phi_p\rangle \phi_p,
\end{equation}
with $D(A^\beta)=H^{\beta}$. For $\theta\in H^\beta$, $f\in L^2(0,T;H)$, we consider the forward problem for inhomogeneous time-fractional diffusion
\begin{equation}
\left\{ \begin{gathered}
\partial_t^{\alpha}u_{\alpha,\beta,\theta,f} +A^\beta u_{\alpha,\beta,\theta,f} =f(t),\ \ \ \ t \in (0,T),\hfill \\
u_{\alpha,\beta,\theta,f}(0) =\theta, \hfill\\
\end{gathered}   \right.
\label{fractional-diffusion}
\end{equation}
which can be rewritten as
$$
	D_t^{\alpha}(u_{\alpha,\beta,\theta,f}-\theta) +A^\beta u_{\alpha,\beta,\theta,f} =f(t),\ \ \ \ t \in (0,T).
$$
We can define the weak solution of the problem as in the general case with $V=H^{\beta/2}$ and
$$  a_\beta(t;u,v)= \sum_{p=1}^\infty \lambda^\beta_p\langle u,\phi_p\rangle\langle v,\phi_p\rangle\ \ \ \ \forall u,v\in V.$$
To consider the problem \eqref{fractional-diffusion} conveniently, we write $u_{\alpha,\beta,\theta,f}=v_{\alpha,\beta,\theta}+w_{\alpha,\beta,f}$ where
$v_{\alpha,\beta,\theta}:(0,T)\to H$ satisfies the homogeneous problem
\begin{equation}
\left\{ \begin{gathered}
\partial_t^{\alpha}v_{\alpha,\beta,\theta} +A^\beta v_{\alpha,\beta,\theta} =0,\ \ \ \ t \in (0,T),\hfill \\
v_{\alpha,\beta,\theta}(0) =\theta, \hfill\\
\end{gathered}   \right.
\label{homogeneous-fractional-diffusion}
\end{equation}
and  $w_{\alpha,\beta,f}:(0,T)\to H$ satisfies the nonhomogeneous problem with zero initial value
\begin{equation}
\left\{ \begin{gathered}
\partial_t^{\alpha}w_{\alpha,\beta,f} +A^\beta w_{\alpha,\beta,f} =f(t),\ \ \ \ t \in (0,T),\hfill \\
w_{\alpha,\beta,f}(0) =0. \hfill\\
\end{gathered}   \right.
\label{nonhomogeneous-fractional-diffusion}
\end{equation}
We have the following theorem which establishes the existence, uniqueness, and continuity  of solutions of the homogeneous equation \eqref{homogeneous-fractional-diffusion}.

\begin{theorem}\label{homogeneous-fractional-diffusion-theorem}
	Let $s\geq 0,r\geq 0$, $0<\alpha_0<\alpha_1<1$, $0<\beta_0<\beta_1$,  $\alpha\in [\alpha_0,\alpha_1]$ and $\beta_0\leq \beta\leq \beta_1$.
	\begin{enumerate}[ \bf \upshape (a)]
		\item For $\theta\in H^s$, the problem (\ref{homogeneous-fractional-diffusion}) has a unique solution
		$$v_{\alpha,\beta,\theta}\in C([0,T]; H^{s})\cap
		C_\alpha(T; H^{\beta+s}), \partial_t^\alpha v_{\alpha,\beta,\theta}\in C_\alpha(T,H^s)$$
		which has the form
		\begin{equation}\label{solution-homogeneous-fractional-diffusion}
		v(t):=v_{\alpha,\beta,\theta,f}(t)= \sum_{p=1}^\infty F_{\alpha,\beta,\theta,p}(t)\phi_p
		\end{equation}
		where
		\begin{align*}
		F_{\alpha,\beta,\theta,k}(t)&= \langle \theta,\phi_k\rangle E_{\alpha,1}(-\lambda^\beta_kt^\alpha).
		\end{align*}
		Moreover, there exists a constant $C$ independent of $\alpha,\beta,\theta$ such that
		\begin{align*}
		\Vert v\Vert_{C([0,T],H^s)}
		+\Vert v\Vert_{C_\alpha(T,H^{\beta+s})}+
		\Vert \partial_t^\alpha v\Vert_{C_\alpha(T,H^s)}&\leq
		C\Vert\theta\Vert_{H^s}.
		\end{align*}
		If, in addition, $s\ge\beta$   then
		$$   \partial_t^\alpha v_{\alpha,\beta,\theta}\in  C([0,T],H^{s-\beta}) $$
		and
		$$  \Vert \partial^\alpha_tv\Vert^2_{C([0,T];H^{s-\beta})}\leq C\Vert \theta\Vert^2_{H^s}. $$
		\item { Let $s\geq 0$, $\alpha,\alpha'\in (0,1)$, $\beta'\in [\beta_0,\beta_1]$, $\theta,\theta'\in H^s$.
			\begin{enumerate}[\bf \upshape (i)]
				\item If  $\theta'\to\theta$ in $H^s$, $\alpha'\to\alpha$, $\beta'\to\beta$ in $\mathbb{R}$ then
				$$ \Vert v_{\alpha',\beta', \theta'}-v_{\alpha,\beta, \theta}\Vert_{H^s}\rightarrow 0. $$
				\item If $\theta,\theta'\in H^s$, $s>\rho\geq 0$, $\alpha'\in [\alpha_0,\alpha_1],$ $\beta'\in[\beta_0,\beta_1]$ then
				there exists a constant $C=C(\alpha_0,\alpha_1,\beta_0,\beta_1,s,\rho)$ such that
				$$ \Vert v_{\alpha',\beta',\theta'}(.,t)-v_{\alpha,\beta,\theta}(.,t)\Vert_{H^\rho}^2\leq
				C\Vert\theta'-\theta\Vert_{H^s}^2+C\Vert\theta\Vert^2_{H^s}(|\alpha'-\alpha|+|\beta'-\beta)|)^{2\gamma}.  $$
				where $\gamma=\min\{1,(s-\rho)/{\beta_1}\}$.
			\end{enumerate}
		}
	\end{enumerate}
\end{theorem}

\begin{remark}
	For a large class of operators $A$,    Chen et al. \cite{cmn-2012} and Meerschaert et al. \cite{mnv-09}  showed that the solution to equation \eqref{homogeneous-fractional-diffusion} when $\beta=1$,  can be also represented as follows
	\begin{eqnarray}
		u(t,x)&=&\sum_{p=1}^\infty E_{\alpha,1}(-\lambda_pt^\alpha) \langle \theta,\phi_p\rangle \phi_p(x)\nonumber\\
		&=&\mathbb{E}_{x}[\theta(X(E_{t})); \tau_\Omega(X)> E_t]\nonumber\\
		&=& \frac{t}{\alpha}\int_{0}^{\infty}T_\Omega(u)(\theta(x))g_{\alpha}
		(tl^{-1/\alpha })u^{-1/\alpha -1}du= \int_{0}^{\infty}T_\Omega((t/u)^\alpha)(\theta(x))g_{\alpha}
		(l)dl. \nonumber
	\end{eqnarray}
	where $X$ is a process such that $v(t,x)=\mathbb{E}_x(\theta(X(t)), \tau_\Omega>t)=T_\Omega(t)(\theta(x))$ solves the equation \eqref{homogeneous-fractional-diffusion} when $\beta=1$, $\tau_\Omega=\inf \{s>0: X(s)\notin \Omega\}$ is the first exit time of $X$ from $\Omega$, $g_\alpha$ is the density of a stable subordinator $Y_t$  of index $\alpha\in (0,1)$ with the Laplace transform
	$E(e^{-sY_t})=e^{-ts^\alpha}$, and $E_t=\inf\{\tau>0: Y(
	\tau)>t\}$ is the inverse  of $Y$.
\end{remark}

\begin{remark}
In Theorem \ref{homogeneous-fractional-diffusion-theorem} when $\beta=1$,  The operators $A$ include Laplacian $\Delta$ in a bounded domain $\Omega$ in a Euclidean space $\R^d$ with Dirichlet boundary condition, and fractional Laplacian $-(-\Delta)^\gamma$, $\gamma\in (0,1)$ with exterior  Dirichlet  boundary conditions.  The process mentioned in previous remark corresponding to the Laplacian  is  the killed Brownian motion.  The process that corresponds to fractional Laplacian is a symmetric stable process. This fractional Laplacian is different operator than  the operator  one gets by taking the powers of the Laplacian with Dirichlet boundary condition using the spectral theory that was mentioned in equation \eqref{spectral-frac-operator}, see Chen et al. \cite{cmn-2012} for more details. The most important  differences are  the facts  that the eigenvalues are not powers of the eigenvalues of the Laplacian, and the eigenfunctions of the fractional Laplacian are not the same as the eigenfunctions of the Laplacian.
\end{remark}
\begin{proof}[{\bf Proof of Theorem \ref{homogeneous-fractional-diffusion-theorem}}]
	We first prove the existence of solution of the problem (\ref{fractional-diffusion}) by using the Galerkin method combined with the spectral method as in \cite{Sakamoto}. Let
	$V_n=span\{\phi_1,...,\phi_n\}$, $\theta_n=\sum_{p=1}^n\langle\theta,\phi_p\rangle\phi_p$. We  consider the problem of finding $v_n\in
	C([0,T];V_n)$ such that
	$$  \frac{d}{dt}\langle J^{1-\alpha}(v_n-\theta_n,\phi\rangle+a_\beta(v_n,\phi)=0,\ \ \ \
	\forall \phi\in V_n. $$
	Put
\begin{align*}
		v_n(t):= v_{\alpha,\beta,\theta, n}(t)&=\sum_{p=1}^n F_{\alpha,\beta,\theta,p}(t)\phi_p,\\
z_n(t):=z_{\alpha,\beta,\theta, n}(t) &=-\sum_{p=1}^n \lambda_p^\beta c_{p}(t)\phi_p
	\end{align*}
 and choosing $\phi=\phi_k$, $k=\overline{1,n}$ we will obtain the fractional differential equation
	$$ \frac{d}{dt} J^{1-\alpha}\left( c_{nk}-\langle\theta,\phi_k\rangle\right)+\lambda_k^\beta c_{nk}
	=0   $$
	which gives
	$$   c_{nk}(t)=F_{\alpha,\beta,\theta,k}(t).$$
Since $c_{nk}$ is independent of $n$, we will write $c_{nk}=c_k$ and we obtain
	
{\bf Proof of (a):}
	For $0\leq\rho\leq \beta$, direct computation gives
	\begin{eqnarray*}		
\lambda^{2(\rho+s)}_p|F_{\alpha,\beta,\theta,p}(t)|^2
		\leq   		
C\frac
{\lambda^{2(\rho+s)}_p|\langle \theta,\phi_p\rangle|^2}
{(1+\lambda_p^\beta t^\alpha)^2} .
	\end{eqnarray*}
Hence, for $\rho=0$, we obtain
$$ \lambda^{2s}_p|F_{\alpha,\beta,\theta,p}(t)|^2
		\leq C\lambda^{2s}_p|\langle \theta,\phi_p\rangle|^2
\ \ {\rm and}\ \sum_{p=1}^\infty\lambda^{2s}_p|\langle \theta,\phi_p\rangle|^2=\Vert \theta\Vert^2_{H^s}<\infty.$$
Hence we deduce that $\{v_n\}$ is uniformly convergent to the function
$$	v(t)=\sum_{p=1}^\infty F_{\alpha,\beta,\theta,p}(t)\phi_p$$
in $C([0,T];H^s)$ and that
\begin{eqnarray*}
		\Vert v(t)\Vert^2_{H^s}
		\leq  C\sum_{p=1}^\infty 		
\lambda^{2s}_p|\langle \theta,\phi_p\rangle|^2=C\Vert \theta\Vert^2_{H^s}.
	\end{eqnarray*}
Choosing $\rho=\beta$, we obtain
	\begin{eqnarray*}		
\lambda^{2(\beta+s)}_p|F_{\alpha,\beta,\theta,p}(t)|^2
		&\leq &   		
C\frac
{\lambda^{2(\beta+s)}_p|\langle \theta,\phi_p\rangle|^2}
{(1+\lambda_p^\beta t^\alpha)^2},\\
\frac
{t^{2\alpha}\lambda^{2(\beta+s)}_p|\langle \theta,\phi_p\rangle|^2}
{(1+\lambda_p^\beta t^\alpha)^2}
&\leq &
\sum_{p=1}^\infty \lambda^{2s}_p|\langle \theta,\phi_p\rangle|^2=
C\Vert \theta\Vert^2_{H^s}.
	\end{eqnarray*}
We deduce that $\{v_n\}$ is uniformly convergent to the function $v$ in $C_{\alpha}(T,H^{\beta+s})$ and
\begin{eqnarray*}
		t^{2\alpha}\Vert v(t)\Vert^2_{H^{\beta+s}}
		\leq  \sum_{p=1}^\infty 		
C\frac
{t^{2\alpha}\lambda^{2(\beta+s)}_p|\langle \theta,\phi_p\rangle|^2}
{(1+\lambda_p^\beta t^\alpha)^2}
\leq
C\sum_{p=1}^\infty \lambda^{2s}_p|\langle \theta,\phi_p\rangle|^2=
C\Vert \theta\Vert^2_{H^s}.
	\end{eqnarray*}
We deduce that $\{v_n\}$ is uniformly convergent to the function $v$ in $C([0,T];H^s)\cap C_{\alpha}(T,H^{\beta+s})$.
Combining these inequalities, we obtain
$$  \Vert v\Vert^2_{C([0,T];H^s)}+ \Vert v\Vert^2_{C_\alpha(T;H^{\beta+s})} \leq C\Vert \theta\Vert^2_{H^s}.    $$
To estimate $\partial^\alpha_t v_n$ we have
\begin{align*}
\lambda_p^{2s}|\lambda_p^\beta c_{p}(t)|^2 &= \lambda_p^{2(\beta+s)}|F_{\alpha,\beta,\theta,p}(t)|^2
\end{align*}
and
\begin{align*}
 t^{2\alpha}\sum_{p=1}^\infty \lambda_p^{2(\beta+s)}|F_{\alpha,\beta,\theta,p}(t)|^2\leq
C\sum_{p=1}^\infty\frac
{t^{2\alpha}\lambda^{2(\beta+s)}_p|\langle \theta,\phi_p\rangle|^2}
{(1+\lambda_p^\beta t^\alpha)^2}\leq C\sum_{p=1}^\infty
\lambda^{2s}_p|\langle \theta,\phi_p\rangle|^2=\Vert\theta\Vert^2_{H^s}<\infty.
\end{align*}
Hence, $z_n$ converge uniformly to a function $z$ in $C_\alpha(T,H^s)$ and
$$ \Vert z\Vert^2_{C_\alpha(T,H^s)}\leq C\Vert \theta\Vert^2_{H^s}. $$
 We can verify as in Theorem \ref{Hilbert-case} that $v$ is a solution of the problem
(\ref{homogeneous-fractional-diffusion}). Moreover,
since $z_n=\partial^\alpha_t v_n$, we can obtain by integration by parts
$$ \int_0^T\varphi(t)z_n(t)dt=-\int_0^T\varphi'(t)J^{1-\alpha}v_n(t)dt\ \ \ \ \ {\rm for\ all}\ \varphi\in C_c^\infty(0,T).   $$
Let $n\to\infty$ gives
$$ \int_0^T\varphi(t)z(t)dt=-\int_0^T\varphi'(t)J^{1-\alpha}v(t)dt\ \ \ \ \ {\rm for\ all}\ \varphi\in C_c^\infty(0,T). $$
Hence $z=\partial^\alpha_t v.$
Combining the estimates for $v,\partial_t^\alpha v$ gives
	\begin{align*}
		\Vert v\Vert_{C([0,T],H^s)}
		+\Vert u\Vert_{C_\alpha(T,H^{\beta+s})}+
 \Vert \partial_t^\alpha u\Vert_{C_\alpha(T,H^s)}&\leq
C\Vert\theta\Vert_{H^s}.
	\end{align*}
Now, if $s\ge\beta$ and $f\in C([0,T]; H^r)$ then
\begin{align*}
\lambda_p^{2(s-\beta)}|\lambda_p^\beta c_{p}( t)|^2 =\lambda_p^{2(s-\beta) }\lambda_p^{2\beta}|F_{\alpha,\beta,\theta,p}(t)|^2
\leq \lambda_p^{2(s-\beta)}\lambda^{2\beta}_p|\langle \theta,\phi_p\rangle|^2
= C\lambda^{2s}_p|\langle \theta,\phi_p\rangle|^2,
\end{align*}
and
$$  \sum_{p=1}^\infty C\lambda^{2s}_p|\langle \theta,\phi_p\rangle|^2=
C\Vert \theta\Vert^2_{H^s}\leq\infty. $$
Hence $z_n$ converges uniformly to the function $\partial^\alpha_tv$ in $C([0,T];H^{s-\beta})$ and
$$  \Vert \partial^\alpha_tv\Vert^2_{C([0,T];H^{s-\beta})}
\leq C\Vert \theta\Vert^2_{H^s}. $$

{\bf Proof of (b):}
In the rest of the proof we will consider the continuity of the solution with respect to the parameter $\alpha,\beta$, the initial condition $\theta$.  From Part (b) of Lemma \ref{mittag-leffler-lemma}, we obtain
	\begin{align*}
		\Big| E_{\alpha,1}(-\lambda^\beta_k t) -E_{\alpha',1}(-\lambda^{\beta'}_k t)\Big|  \le C \la_k^{2\beta_1} \Big( |\beta'-\beta|+ |\alpha'-\alpha| \Big).
	\end{align*}
Recall that we have	
	\begin{align*}
		v_{\alpha,\beta, \theta}(t)= \sum_{p=1}^{\infty} E_{\alpha,1}(-\lambda^\beta_p t) \langle\theta,\varphi_p\rangle \varphi_p.
	\end{align*}
{	So, for $\rho>0$ we obtain
\begin{align*}
		\|v_{\alpha,\beta, \theta}(t) - v_{\alpha',\beta', \theta'}(t)\|^2_{H^\rho}
&\leq 2\sum_{p=1}^{\infty} \lambda_p^{2\rho}\Big[E_{\alpha,1}(-\lambda^\beta_p t) -E_{\alpha',1}(-\lambda^{\beta'}_p t)\Big]^2 \langle\theta, \varphi_p\rangle^2\Big]\\
&+ 2\sum_{p=1}^{\infty} \lambda_p^{2\rho}E^2_{\alpha',1}(-\lambda^{\beta'}_p t)
\langle\theta'-\theta, \varphi_p\rangle^2.
	\end{align*}	
We separate the first sum into two sums, one sum is from $p=1$ to $p=N$ and one sum is from
$p=N+1$ to infinity. Using the fact that $0\leq E_{\alpha,1}(z)\leq 1$ for $z\leq 0$ and Part (b) of Lemma \ref{mittag-leffler-lemma} we obtain
\begin{align*}
		\|v_{\alpha',\beta', \theta'}(t) - v_{\alpha,\beta, \theta}(t)\|^2_{H^\rho}
& \leq C(|\alpha'-\alpha|+|\beta'-\beta|)^2\sum_{p=1}^N \lambda_p^{2(\beta_1+\rho)} \langle\theta, \varphi_p\rangle^2 +\\
& 2\sum_{p=N+1}^\infty \lambda_p^{2\rho}\langle\theta, \varphi_p\rangle^2+ 2\Vert \theta'-\theta\Vert^2_{H^\rho}.
	\end{align*}
}

(i) We choose a sequence $(\alpha_n,\beta_n,\theta_n)\in (0,1)\times[\beta_0,\beta_1]\times H$ such that
$(\alpha_n,\beta_n,\theta_n)\to (\alpha,\beta,\theta)$. It follows that
$$   \limsup_{n\to\infty}\|v_{\alpha,\beta, \theta}(t) - v_{\alpha_n,\beta_n, \theta_n}(t)\|_{H^s}^2\leq 2\sum_{p=N+1}^\infty \lambda_p^{2s}\langle\theta, \varphi_p\rangle^2.$$
Letting  $N\to\infty$ gives
$$\limsup_{n\to\infty}\|v_{\alpha,\beta, \theta}(t) - v_{\alpha_n,\beta_n, \theta_n}(t)\|_{H^s}^2=0.   $$
Hence, we otain (i).\\
(ii) Now, if $\theta\in H^s$, $s>\beta_1+\rho$ then
\begin{align*}
		\|v_{\alpha',\beta', \theta'}(t) - v_{\alpha,\beta, \theta}(t)\|_{H^\rho}^2
& \leq C(|\alpha'-\alpha|+|\beta'-\beta|)^2\sum_{p=1}^\infty \lambda_p^{2(\beta_1+\rho)} \langle\theta, \varphi_p\rangle^2 +
 2\Vert \theta'-\theta\Vert_{H^\rho}^2\\
&\leq C(|\alpha'-\alpha|+|\beta'-\beta|)^2\sum_{p=1}^\infty  \lambda_p^{2s}\langle\theta, \varphi_p\rangle^2 +2\Vert \theta'-\theta\Vert_{H^s}^2\\
&\leq C(|\alpha'-\alpha|+|\beta'-\beta|)^2\Vert \theta\Vert^2_{H^s}+2\Vert \theta'-\theta\Vert_{H^s}^2.
	\end{align*}
If $\theta\in H^s$, $0<s\leq\beta_1+\rho$ then
\begin{align*}
		\|v_{\alpha',\beta', \theta'}(t) - v_{\alpha,\beta, \theta}(t)\|_{H^\rho}^2
& \leq C(|\alpha'-\alpha|+|\beta'-\beta|)^2\sum_{p=1}^N \lambda_p^{2(\beta_1+\rho)} \langle\theta, \varphi_p\rangle^2 +\\
& 2\sum_{p=N+1}^\infty \lambda_p^{2\rho}\langle\theta, \varphi_p\rangle^2+ 2\Vert \theta'-\theta\Vert_{H^\rho}^2\\
&\leq C(|\alpha'-\alpha|+|\beta'-\beta|)^2\lambda_N^{2(\beta_1+\rho-s)}\sum_{p=1}^N  \lambda_p^{2s}\langle\theta, \varphi_p\rangle^2 +\\
& 2\lambda_{N+1}^{-2(s-\rho)}\sum_{p=N+1}^\infty \lambda_p^{2s}\langle\theta, \varphi_p\rangle^2+ 2\Vert \theta'-\theta\Vert^2\\
& \leq C\left((|\alpha'-\alpha|+|\beta'-\beta|)^2\lambda_N^{2(\beta_1+\rho-s)}+2\lambda_{N+1}^{-2(s-\rho)}\right)\Vert \theta\Vert^2_{H^s}+2\Vert \theta'-\theta\Vert_{H^s}^2.
	\end{align*}
Choose $N$ such that $\lambda_{N+1}^{-1}\leq (|\alpha'-\alpha|+|\beta'-\beta|)^{1/{\beta_1}}\leq \lambda_N^{-1} $ gives
\begin{align*}
		\|v_{\alpha',\beta', \theta'}(t) - v_{\alpha,\beta, \theta}(t)\|_{H^\rho}^2
\leq
C(|\alpha'-\alpha|+|\beta'-\beta|)^{2(s-\rho)/{\beta_1}}\Vert \theta\Vert^2_{H^s}+2\Vert \theta'-\theta\Vert_{H^s}^2.
	\end{align*}
\end{proof}

Next we establish results for the non-homogeneous problem  \eqref{nonhomogeneous-fractional-diffusion} with zero initial value.

\begin{theorem}\label{nonhomogeneous-fractional-theorem}
	Let $r\geq 0$, $0<\alpha_0<\alpha_1<1$, $0<\beta_0<\beta_1$,  $\alpha\in [\alpha_0,\alpha_1]$ and $\beta_0\leq \beta\leq \beta_1$.
	\begin{enumerate}[\bf \upshape (a)]
		\item For  $f\in L^2(0,T;H^r)$,  the problem (\ref{nonhomogeneous-fractional-diffusion}) has a unique solution
		$$w_{\alpha,\beta,f}\in
		C([0,T]; H^{\beta+r}), \partial_t^\alpha w_{\alpha,\beta,f}\in L^2(0,T;H^r)$$
		which has the form
		\begin{equation}\label{solution-fractional-diffusion}
		w(t):=w_{\alpha,\beta,f}(t)= \sum_{p=1}^\infty G_{\alpha,\beta,f,p}(t)\phi_p
		\end{equation}
		where
		\begin{align*}
		G_{\alpha,\beta,f, k}(t)&=\int_0^t\langle f(t),\phi_k\rangle(t-\tau)^{\alpha-1}E_{\alpha,\alpha}(-\lambda^\beta_k(t-\tau)^\alpha)d\tau.
		\end{align*}
		Moreover, there exists a constant $C$ independent of $\alpha,\beta,f$ such that
		\begin{align*}
		\Vert w\Vert_{C([0,T],H^{\beta+ r})}+
		\Vert \partial_t^\alpha w\Vert_{L^2_{\alpha, r}(T)}&\leq
		C\Vert f\Vert_{L^2(0,T;H^r)}.
		\end{align*}
		If, in addition, $f\in C([0,T],H^r)$ then
		$$  \partial_t^\alpha w_{\alpha,\beta,f}\in  C([0,T],H^r) $$
		and
		$$  \Vert \partial^\alpha_tw\Vert^2_{C([0,T];H^r)}\leq C\Vert f\Vert^2_{C([0,T];H^r)}. $$
		\item For $f\in L^2(0,T;H^r)$, we have
		$$ w_{\alpha',\beta',f'}\rightarrow w_{\alpha,\beta,f}\ \ \ \text{ in } L^2(0,T;H^{\beta_0+r})$$
		as $(\alpha',\beta',f')\to (\alpha,\beta,f)$ in $\mathbb{R}\times\mathbb{R}\times L^2(0,T;H^r)$
		and
		\begin{align*}
		\Vert w_{\alpha',\beta',f'}-w_{\alpha,\beta,f}\Vert^2_{L^2(0,T;H^{\rho+r})}&\leq C_\delta
		(|\alpha-\alpha'|+|\beta-\beta'|)^{2(\beta_0-\rho)\mu} \Vert f\Vert^2_{L^2(0,T;H^r)}+\\
		&C\Vert f'-f\Vert^2_{L^2(0,T;H^{r})}
		\end{align*}
		for $\delta>0,0<\rho<\beta_0$ and $\mu=(\beta_0+\beta_1+\delta)^{-1}$.
	\end{enumerate}
\end{theorem}
\begin{proof}
 Put
	$V_n=span\{\phi_1,...,\phi_n\}$ as in the proof of the previous theorem. We  consider the problem of finding $w_n\in
	C([0,T];V_n)$ such that
	$$  \frac{d}{dt}\langle J^{1-\alpha}w_n,v\rangle+a_\beta(w_n,v)=\langle f(t),v\rangle,\ \ \ \
	\forall v\in V_n. $$
	By putting $w_n(t)=w_{\alpha,\beta,f,n}(t):=\sum_{p=1}^n d_{np}(t)\phi_p$ and choosing $v=\phi_k$, $k=\overline{1,n}$ we will obtain the fractional differential equation
	$$ \frac{d}{dt} J^{1-\alpha} d_{nk}+\lambda_k^\beta d_{nk}
	=\langle f(t),\phi_k\rangle   $$
	which gives
	$$   d_{nk}(t)=G_{\alpha,\beta,f,k}(t).$$
Since $d_{nk}$ is independent of $n$, we will denote $d_{nk}=d_k$.	
We denote
	\begin{align*}
		w_n(t):=w_{\alpha,\beta,f, n}(t) &=\sum_{p=1}^n G_{\alpha,\beta,f,p}(t)\phi_p,\\
                h_n(t):=h_{\alpha,\beta,f, n}(t) &=\sum_{p=1}^n (-\lambda_p^\beta d_{p}(t)+ \langle f(t),\phi_p\rangle)\phi_p.
	\end{align*}
{\bf Proof of (a):}
We estimate the Fourier coefficients of the function $w_n$. We first have
	\begin{eqnarray*}
		|G_{\alpha,\beta,f,p}(t)|^2&\leq&  \int_0^t|\langle f(\tau),\phi_p\rangle|^2d\tau
		\left(\int_0^t
		\tau^{\alpha-1}E_{\alpha,\alpha}(-\lambda_p^\beta \tau^\alpha)d\tau\right)^2\\
		&=& \frac{1}{\lambda_p^{2\beta}}\int_0^t|\langle f(\tau),\phi_p\rangle|^2d\tau
		\left(\int_0^t \frac{d}{d\tau}E_{\alpha,1}(-\lambda_p^\beta \tau^\alpha)d\tau\right)^2\\
		&=& \frac{\left(1-E_{\alpha,1}(-\lambda_p^\beta t^\alpha)\right)^2}{\lambda_p^{2\beta}}\int_0^t|\langle f(\tau),\phi_p\rangle|^2d\tau \\
		&\leq& \frac{1}{\lambda_p^{2\beta}}\int_0^T|\langle f(\tau),\phi_p\rangle|^2d\tau.
	\end{eqnarray*}
	Hence, we have
	\begin{eqnarray*}
		 \lambda_p^{2(\beta+r)}|G_{\alpha,\beta,f,p}(t)|^2
		&\leq & \lambda_p^{2 r}\int_0^T|\langle f(\tau),\phi_p\rangle|^2d\tau\\
\sum_{p=1}^\infty \lambda_p^{2(\beta+r)}|G_{\alpha,\beta,f,p}(t)|^2
		&\leq &\sum_{p=1}^\infty \lambda_p^{2r}\int_0^T|\langle f(\tau),\phi_p\rangle|^2d\tau=C\Vert f\Vert^2_{L^2(0,T;H^r)}.
	\end{eqnarray*}
This implies that $\{w_n\}$ is uniformly convergent to the function
$$		w(t):=w_{\alpha,\beta,f}(t) =\sum_{p=1}^\infty G_{\alpha,\beta,f,p}(t)\phi_p$$
in $C([0,T];H^{\beta+r})$ and
$$\Vert w(t)\Vert^2_{H^{\beta+ r}}=\sum_{p=1}^\infty \lambda_p^{2(\beta+ r)}|G_{\alpha,\beta,f,p}(t)|^2
		\leq \sum_{p=1}^\infty \lambda_p^{2 r}\int_0^T|\langle f(\tau),\phi_p\rangle|^2d\tau=C\Vert f\Vert^2_{L^2(0,T;H^r)}.$$
We can verify as in Theorem
	\ref{Hilbert-case} that $w$ is a solution of the problem (\ref{nonhomogeneous-fractional-diffusion}).
Now, we have
\begin{align*}
\lambda_p^{2r}|\lambda_p^\beta d_{p}(t)+ \langle f(t),\phi_p\rangle|^2 &\leq
 2\lambda_p^{2(\beta+r)}|G_{\alpha,\beta,f,p}(t)|^2) +2\lambda_p^{2r}|\langle f(t),\phi_p\rangle|^2,
\end{align*}
and
\begin{align*}
\int_0^T \sum_{p=1}^\infty \left\{
 2\lambda_p^{2(\beta+r)}|G_{\alpha,\beta,f,p}(t)|^2) +
 2\lambda_p^{2r}|\langle f(t),\phi_p\rangle|^2
\right\} dt
\leq C\Vert f\Vert^2_{L^2(0,T;H^r)}.
\end{align*}
Hence, $h_n$ converge uniformly to a function $h$ in $L^2(0,T;H^r)$ and
$$ \Vert h\Vert^2_{L^2(0,T;H^r)}\leq C\Vert f\Vert^2_{L^2(0,T;H^r)}. $$
  Moreover,
since $h_n=\partial^\alpha_t w_n$,
we can prove as in the proof of the previous theorem that $h=\partial^\alpha_t u.$
Combining the estimates for $w,\partial_t^\alpha w$ gives
	\begin{align*}
		\Vert w\Vert_{C([0,T];H^{\beta+r})}+
 \Vert \partial_t^\alpha w\Vert_{L^2(0,T;H^r)}&\leq
C\Vert f\Vert_{L^2(0,T;H^r)}).
	\end{align*}
Now, if $f\in C([0,T]; H^r)$ then
\begin{align*}
\sum_{p=1}^\infty \left\{
 2\lambda_p^{2(\beta+r)}|G_{\alpha,\beta,f,p}(t)|^2) +
 2\lambda_p^{2r}|\langle f(t),\phi_p\rangle|^2
\right\}
\leq C\Vert f\Vert^2_{C([0,T];H^r)}.
\end{align*}
Hence $h_n$ converges uniformly to the function $\partial^\alpha_tw$ in $C([0,T];H^r)$ and
$$  \Vert \partial^\alpha_tw\Vert^2_{C([0,T];H^r)}\leq C\Vert f\Vert^2_{C([0,T];H^r)}. $$
 {\bf Proof of (b):}   Finally, for $0<\rho\leq\beta_0$ we estimate
\begin{align*}
\Vert w_{\alpha',\beta',f'}-w_{\alpha,\beta,f}\Vert^2_{L^2(0,T;H^{\rho+r})} &=
\sum_{p=1}^\infty \lambda_p^{\rho+r}\left|G_{\alpha',\beta',f',p}(t)-G_{\alpha,\beta,f,p}(t)\right|^2\\
&\leq 2\sum_{p=1}^\infty \lambda_p^{\rho+r}\left|G_{\alpha',\beta',f',p}(t)-G_{\alpha',\beta',f,p}(t)\right|^2\\
&+2\sum_{p=1}^\infty \lambda_p^{\rho+r}\left|G_{\alpha',\beta',f,p}(t)-G_{\alpha,\beta,f,p}(t)\right|^2.
\end{align*}
For the first term, we have
\begin{align*}
\sum_{p=1}^\infty \lambda_p^{\rho+r}\left|G_{\alpha',\beta',f',p}(t)-G_{\alpha',\beta',f,p}(t)\right|^2
&=2\sum_{p=1}^\infty \lambda_p^{\rho+r}\left|G_{\alpha',\beta',f'-f,p}(t)\right|^2\\
&\leq C\sum_{p=1}^\infty \lambda_p^{\rho-\beta'}\lambda_p^r|\langle f'-f,\phi_p\rangle|^2\leq C\Vert f'-f\Vert^2_{L^2(0,T;H^{r})}.
\end{align*}
We consider the second term. Choose an $N\in\mathbb{N}$. From Lemma \ref{mittag-leffler-lemma} part (c) we obtain
\begin{align*}
\Vert w_{\alpha',\beta',f}-w_{\alpha,\beta,f}\Vert^2_{L^2(0,T;H^{\rho+r})} &\leq
C\sum_{p=1}^N \lambda_p^{2(\rho+r)}\Big[|\alpha-\alpha'|(1+|\lambda_p|^\beta)+|\lambda_p^{\beta'}-\lambda_p^\beta)|\Big]^2 \int_0^T| \langle f(.),\phi_p\rangle|^2dt\\
&+ 2\sum_{p=N+1}^\infty \lambda_p^{2(\rho+r)}\int_0^T\left\{G^2_{\alpha',\beta',f',p}(t)+G^2_{\alpha,\beta,f,p}(t)\right\}dt\\
&\leq C\sum_{p=1}^N \lambda_p^{2(\rho+r)}\Big[|\alpha-\alpha'|(1+|\lambda_p|^\beta)+|(\beta-\beta')\lambda_p^{\beta_1}\ln\lambda_p|\Big]^2 \int_0^T|\langle f(t),\phi_p\rangle|^2dt\\
&+C\sum_{p=N+1}^\infty \lambda_p^{2(\rho+r)}\left\{\frac{1}{\lambda_p^{2\beta'}}+\frac{1}{\lambda_p^{2\beta}}\right\}\int_0^T|\langle f(t),\phi_p\rangle|^2dt\\
&\leq C\lambda_N^{2(\rho+\beta_1)}|\ln\lambda_N|^2(|\alpha-\alpha'|+|\beta-\beta'|)^2\Vert f\Vert^2_{L^2(0,T;H^r)}\\
&+\lambda_{N+1}^{2(\rho-\beta_0)}\sum_{p=N+1}^\infty \int_0^T\lambda_p^r|\langle f(t),\phi_p\rangle|^2dt.
\end{align*}
Now, for $\rho=\beta_0$, we can verify directly that
$$ w_{\alpha',\beta,f'}\rightarrow w_{\alpha,\beta,f}\ \ \ \ {\rm in}\ L^2(0,T;H^{\beta_0+r})$$
as $(\alpha',\beta',f')\to (\alpha,\beta,f)$ in $\mathbb{R}\times\mathbb{R}\times L^2(0,T;H^r)$.
For $0\leq \rho<\beta_0$, we choose $N$ such that
$$ \lambda_N\leq (|\alpha-\alpha'|+|\beta-\beta'|)^{-\mu}\leq \lambda_{N+1}$$
where $\mu$ satisfies $1-\mu(\rho+\beta_1+\delta)=\mu(\beta_0-\rho)$, i.e., $\mu=(\beta_0+\beta_1+\delta)^{-1}$.
Then we have
\begin{align*}
  \Vert w_{\alpha',\beta',f}-w_{\alpha,\beta,f}\Vert^2_{L^2(0,T;H^{\rho+r})}
&\leq C_\delta  (|\alpha-\alpha'|+|\beta-\beta'|)^{2-2\mu(\rho+\beta_1+\delta)}\Vert f\Vert^2_{L^2(0,T;H^r)}+\\
&C (|\alpha-\alpha'|+|\beta-\beta'|)^{2(\beta_0-\rho)\mu}\Vert f\Vert^2_{L^2(0,T;H^r)}\\
&\leq (C_\delta+C)(|\alpha-\alpha'|+|\beta-\beta'|)^{2(\beta_0-\rho)\mu}\Vert f\Vert^2_{L^2(0,T;H^r)}.
 \end{align*}
It follows that
$$ \Vert w_{\alpha',\beta',f'}-w_{\alpha,\beta,f}\Vert^2_{L^2(0,T;H^{\rho+r})}\leq C_\delta
(|\alpha-\alpha'|+|\beta-\beta'|)^{2(\beta_0-\rho)\mu} \Vert f\Vert^2_{L^2(0,T;H^r)}+
C\Vert f'-f\Vert^2_{L^2(0,T;H^{r})}.$$

\end{proof}

\section{Instability of the solutions of some Ill-posed problems}

\setcounter{equation}{0}

In this section, we will give some definitions and examples for showing the instability of solutions
in the case of  the fractional order is noised. First, we introduce general theory of stabiliy and instability of inverse problems which depended on the (noise) fractional order. Next, we present some examples for this general theory.

\subsection{General theory}
Let $X,Y$ be two Banach spaces, $K: X\to Y$. In literature of inverse problems, we often use  Hadamard's definition of ill-posedness for the problem $Kx=y$.  Now,  we first  consider the instability of inverse problems which depended on the (noise) fractional order $\alpha$. To make the situation clear, we  develop a concept of instability for a family of operator $K_\beta$ upon the generic parametric $\beta$ and discuss some examples.
Let $\beta_0\in \mathbb{R}$, $x_0, u_0\in X$, $y_0, v_0\in Y$ and $[a,b]$ is an interval in $\mathbb{R}$
such that $\beta_0\in [a,b]$. We consider the family of (linear or nonlinear) operators $K_\beta: X\to Y$ where
$\beta\in [a,b]$. Assume that $K_{\beta_0}x_0=y_0$, $K u_0=v_0$.
\begin{definition}
	We say that
	\begin{enumerate}[\bf \upshape (a)]
		\item the operator $K$ has an unstable inverse at $u_0$ if there exist
		sequences $(u_n)\subset X$ such that
		$ Ku_n\to v_0$ but $u_n\not\to u_0$ as $n\to\infty$.
		\item the family $(K_\beta)$ has a $\beta_0$-unstable inverse at $x_0$ if there exist
		sequences $(\beta_n)$ $\subset (a,b)$, $(x_n)\subset X$ such that
		$\beta_n\to\beta_0$, $ K_{\beta_n}x_n\to y_0$ but $x_n\not\to x_0$ as $n\to\infty$.
		\item the family $(K_\beta)$ has a properly $\beta_0$-unstable inverse at $x_0$ if there exist
		sequences $(\beta_n)$ $\subset (a,b)$, $(x_n)\subset X$, $(y_n)\subset Y$ such that $K_{\beta_n}x_n=y_n$, $K_{\beta_0}x^*_n=y_n$ and
		$$\beta_n\to\beta_0,
		y_n\to y_0  , x^*_n\to x_0\ {\rm but} \ x_n\not\to x_0\ {\rm as}\ n\to\infty.$$
		\item the operator $R_{\delta}: Y\to X$ is a regularization at $\beta_0$ of the family $(K_\beta)$  if
		$$   \lim_{(\delta,\beta)\to(0,\beta_0)}R_{\delta}K_\beta x=x\ \ \ {\rm for}\ x\in X. $$
	\end{enumerate}
\end{definition}
We have
\begin{theorem}
	With the notations in the  definition above, we obtain the following  results.
	\begin{enumerate}[\bf \upshape (a)]
		\item If $K_\beta x\to K_{\beta_0} x$ as $\beta\to\beta_0$ for every $x\in X$ and $K_{\beta_0}$ has an unstable inverse at $x_0$ then
		$(K_\beta)$ has a $\beta_0$-unstable inverse at $x_0$.
		\item If $sup_{\Vert x\Vert_X\leq M}\Vert K_\beta x-K_{\beta_0}x\Vert\to 0$ as $\beta\to\beta_0$ and $(K_\beta)$ has the $\beta_0$-unstable inverse at $x_0$ with $\sup_{n}\Vert x_n\Vert_X<\infty$ then $K_{\beta_0}$ has the unstable inverse at $x_0$.
		\item If $K_\beta$ is bounded linear and that $\lim_{\beta\to\beta_0}\Vert K_{\beta}-K_{\beta_0}\Vert=0$
		and $(K_\beta)$ has the $\beta_0$-unstable inverse at $x_0$ then $K_{\beta_0}$ has the unstable inverse at $x_0$.
	\end{enumerate}
\end{theorem}
\begin{proof}
	{\bf Proof of (a):} Assume that $K_{\beta_0}(u_n)=v_n$, $v_n\to v_0$ in $Y$, $u_n\not\to u_0$ in $X$. For each
	$n\in\mathbb{N}$, we choose $\beta_n\in (a,b)$ such that
	$$|\beta_n-\beta_0| +\Vert K_{\beta_n}u_n-K_{\beta_0}u_n\Vert\leq \frac{1}{n}\ {\rm as}\  n\to\infty.$$
	Put $y_n=K_{\beta_n}u_n, x_n=u_n$ we obtain that $(K_\beta)$ has the $\beta_0$-unstable inverse at $x_0$.\\
	{\bf Proof of (b):} Assume that $K_{\beta_n}x_n=y_n$, $y_n\to y_0$, $x_n\not\to x_0$ and $\sup_n\Vert x_n\Vert_X\leq M$.
	We have
	$$  \lim_{n\to\infty}\Vert K_{\beta_0}x_n-K_{\beta_n}x_n\Vert_Y=0.$$
	It follows that
	$K_{\beta_0}x_n\to y_0$ and $x_n\not\to x_0$ as $n\to\infty$. Hence $K_{\beta_0}$ has the unstable inverse at
	$x_0$.\\
	{\bf Proof of (c):} Assume that $K_{\beta_n}x_n=y_n$, $y_n\to y_0$, $x_n\not\to x_0$. If  $\sup_n\Vert x_n\Vert_X\leq M$ we can use (b) to prove the result. If $\sup_n\Vert x_n\Vert_X=\infty$, we can choose a subsequence
	$x_{n_k}$ such that $\lim_{k\to\infty}\Vert x_{n_k}\Vert_X=\infty$. Putting $\tilde{x}_k=x_0+z_k$ with
	$z_k=\frac{x_{n_k}}{\Vert x_{n_k}\Vert_X}$ we have $\Vert\tilde{x}_k-x_0\Vert_X=\Vert z_k\Vert_X=1$,
	$$\lim_{k\to\infty}\Vert K_{\beta_{n_k}}z_k-K_{\beta_0}z_k\Vert_Y=0,\ \lim_{k\to\infty}
	\Vert K_{\beta_{n_k}}z_k\Vert_Y=\lim_{k\to\infty}\Vert y_{n_k}\Vert_Y/\Vert x_{n_k}\Vert_X=0.$$
	Hence $\lim_{k\to\infty}K_{\beta_0}z_k=0$ and
	$$  K_{\beta_0}\tilde{x}_k=y_0+ K_{\beta_0}z_k\to y_0\ \ \ \ {\rm as}\ k\to\infty. $$
	It follows that $K_{\beta_0}$ has the unstable inverse at $x_0$.
\end{proof}
\subsection{Some examples for instability of fractional order}
\subsubsection{The first example}
In this subsection, we present an example to show that the $\beta_0$-unstable inverse cannot imply the unstable inverse
of $K_{\beta_0}$. For $0<\alpha<1$, we consider the Abel operators $J^\alpha_{-\infty}: L^2(\mathbb{R})
\to L^2(\mathbb{R})$ defined by
$$ J^\alpha_{-\infty}u(x)=\frac{1}{\Gamma(\alpha)}\int_{-\infty}^x (x-t)^{\alpha-1}u(t)dt. $$
From \cite[page 96]{Gorenflo-Vessella}, we have $\lim_{\alpha\to 0}J^\alpha_{-\infty} u = u$ for every
$u\in H^1(\mathbb{R})$. So we can define $J^0_{-\infty}=Id$ and $(J^0_{-\infty})^{-1}=Id$ is continuous.
Now, we prove that the family $(J^\alpha_{-\infty})$ has the 0-unstable inverse.
If $J^\alpha_{-\infty}u(x)=f(x)$ then $\hat{u}(\tau)=(i\tau)^\alpha \hat{f}(\tau)$. For $a_n,\delta_n>0$, we put
$$ f_n(x)=\frac{1}{2\pi}\left(\int_{a_n}^{a_n+\delta_n}e^{itx}dt+\int_{-a_n-\delta_n}^{-a_n}e^{itx}dt\right). $$
We have $\hat{f}_n=\chi_{(a_n,a_n+\delta_n)}+\chi_{(-a_n-\delta_n,-a_n)}$. It follows that $\Vert f\Vert^2=\frac{1}{2\pi}\Vert \hat{f}_n\Vert^2=\frac{\delta_n}{\pi}$.
On the other hand, we have $J^{\alpha_n}_{-\infty}u_n(x)=f_n(x)$ for $\hat{u}_n(\tau)=(i\tau)^{\alpha_n}
\left(\chi_{(a_n,a_n+\delta_n)}+\chi_{(-a_n-\delta_n,-a_n)}\right)$. So we have
\begin{equation*}
	\Vert u_n\Vert^2=\frac{1}{2\pi}\Vert \hat{u}_n\Vert^2=\frac{1}{\pi}\int_{a_n}^{a_n+\delta_n}|\tau|^{2\alpha_n}d\tau \geq \frac{1}{\pi}a_n^{2\alpha_n}\delta_n.
\end{equation*}
Now, if we choose $a_n=n^n, \alpha_n=\delta_n=\frac{1}{n}$ then $J^{\alpha_n}_{-\infty}u_n=f_n$ with
$\alpha_n\to 0$, $f_n\to 0$ but $u_n\not\to 0$. Hence $(J^{\alpha}_{-\infty})$ has the 0-unstable inverse.

\subsubsection{The second  example}
We consider the problem of finding $u_f \in L^2(\mathbb{R}) $ from the given exact data $f \in L^2(\mathbb{R})$ such that
\begin{align*}
	\hat{u}_{\al,f}(\tau)= e^{a \tau^\alpha} \hat{f}(\tau)
\end{align*}
where $a>0$ is constant.

First, we consider the  given data $\alpha >0$ and $f_0=0$. Then  $\hat{u}_{\al,{f_0}}(\tau)=0$.
Assume that $ (\alpha, f_0)$ is noised by $(\alpha+\ep_n, f_n)$
where $f_n \in L^2(\mathbb{R})$ such that  $\hat{f}_n=n\chi_{(n^n, n^n+\frac{1}{n^3})} $ and $\epsilon_n =\frac{1}{n}$.
Since the equality
\begin{align*}
	\|f_n\|_{L^2(\mathbb{R}) }= \|\hat{f}_n\|_{L^2(\mathbb{R}) }= \int_{n^n}^{n^n+\frac{1}{n^3}} n^2 d\tau=\frac{1}{n}
\end{align*}
we know that
\begin{align*}
	|\alpha+\ep_n -\alpha|+ \|f_n-f_0\|_{L^2(\mathbb{R}) } =\frac{2}{n} \to 0,~~ n\to \infty.
\end{align*}
And we have
\begin{align*}
	\hat{u}_{\alpha+\epsilon_n,f_n}(\tau)= e^{a \tau^{\alpha+\epsilon_n}} \hat{f_n}(\tau) =n\chi_{(n^n, n^n+\frac{1}{n^3})} e^{a \tau^{\alpha+\epsilon_n}}
\end{align*}
The norm of $u_{\alpha+\epsilon_n,f_n}$ in $ L^2 $ is estimated  as follows
\begin{align*}
	\Big\|\hat{u}_{\alpha+\epsilon_n,f_n} -u_{\alpha,f_0} \Big\|_{L^2(\mathbb{R}) }&= \|\hat{ u }_{\alpha+\epsilon_n,f_n}\|_{L^2(\mathbb{R}) }\nn\\
	&= \int_{n^n}^{n^n+\frac{1}{n^3}}  n^2  e^{2a \tau^{\alpha+\epsilon_n}}  \chi_{(n^n, n^n+\frac{1}{n^3})}  d\tau\ge \frac{e^{2an}}{n} \to +\infty,~n \to +\infty
\end{align*}

\end{document}